\documentclass[leqno]{amsart}
\usepackage{cases}
\usepackage{cases}
\usepackage{amsmath}
\usepackage{amsmath}
\usepackage{amsmath}
 \usepackage{graphicx}
 \usepackage{mathrsfs}
\usepackage{bbm}
\usepackage{amsfonts}
\usepackage{amssymb}
\usepackage{ams,amsmath,amsthm}

\usepackage{hyperref}

\def\opn#1#2{\def#1{\operatorname{#2}}} 
\opn\chara{char} \opn\length{\ell}
\opn\projdim{proj\,dim} \opn\injdim{inj\,dim} \opn\rank{rank}
\opn\depth{depth} \opn\grade{grade} \opn\height{height}
\opn\embdim{emb\,dim} \opn\codim{codim}

\opn\Tr{Tr} \opn\bigrank{big\,rank}
\opn\superheight{superheight}\opn\lcm{lcm}
\opn\trdeg{tr\,deg}%
\opn\reg{reg} \opn\lreg{lreg}
\opn\Ker{Ker} \opn\Coker{Coker} \opn\Im{Im} \opn\Hom{Hom}
\opn\Tor{Tor} \opn\Ext{Ext} \opn\End{End} \opn\Aut{Aut} \opn\id{id}

\opn\nat{nat}
\opn\pff{pf}
\opn\Pf{Pf} \opn\GL{GL} \opn\SL{SL} \opn\mod{mod} \opn\ord{ord}

\def\Implies{\ifmmode\Longrightarrow \else
     \unskip${}\Longrightarrow{}$\ignorespaces\fi}
\def\implies{\ifmmode\Rightarrow \else
     \unskip${}\Rightarrow{}$\ignorespaces\fi}
\def\iff{\ifmmode\Longleftrightarrow \else
     \unskip${}\Longleftrightarrow{}$\ignorespaces\fi}

\let\:=\colon

\newtheorem{Theorem}{Theorem}[section]
\newtheorem{Lemma}[Theorem]{Lemma}
\newtheorem{Corollary}[Theorem]{Corollary}
\newtheorem{Proposition}[Theorem]{Proposition}
\newtheorem{Remark}[Theorem]{Remark}

\newtheorem{Example}[Theorem]{Example}

\newtheorem{Algorithm}[Theorem]{Algorithm}

\theoremstyle{definition}

\opn\ini{in} \opn\inm{inm} \opn\Sym{Sym} \opn\diag{diag}
\opn\Ii{(i)} \opn\Iii{(ii)}

\title{ Adaptive finite element method for elliptic optimal control problems: Convergence and optimality}
\author{Wei Gong $^*$}
\author{Ningning Yan $^\diamond$}
\thanks{$^*$ LSEC, Institute of Computational Mathematics, Academy of Mathematics and Systems Science, Chinese Academy of Sciences, Beijing 100190, China.
 Email: {\tt wgong@lsec.cc.ac.cn}}
\thanks{$^\diamond$ NCMIS, LSEC, Institute of Systems Science, Academy of Mathematics and Systems Science, Chinese Academy of Sciences, Beijing 100190, China.
 Email: {\tt ynn@amss.ac.cn}}

\date{\today}

\begin{document}
\maketitle

{\bf Abstract:}\hspace*{10pt} {In this paper we consider the convergence analysis of adaptive finite element method for elliptic optimal control problems with pointwise control constraints. We use variational discretization concept to discretize the control variable and piecewise linear and continuous finite elements to approximate the state variable. Based on the well-established convergence theory of AFEM for elliptic boundary value problems, we rigorously prove  the convergence and quasi-optimality of AFEM for optimal control problems with respect to the state and adjoint state variables, by using the so-called perturbation argument. Numerical experiments confirm our theoretical analysis.  }

{{\bf Keywords:}\hspace*{10pt} optimal control problem,
 elliptic equation, control constraint, adaptive finite element method, convergence and optimality }


\section{Introduction}
\setcounter{equation}{0}
Adaptive finite element method (AFEM for short), contributed to the pioneer work of Babu\v{s}ka and Rheinboldt (\cite{Babuska}), becomes nowadays a popular approach in the community of engineering and scientific computing. It aims at distributing more mesh nodes around the area where the singularities happen to save the computational cost. Various types of reliable and efficient a posteriori error estimators, which are used to detect the location of singularity and  essential for the success of AFEM, have been developed in the last decades for different kind of problems, we refer to \cite{Verfuth} for an overview. 

Although AFEM has been successfully applied for more than three decades, the convergence analysis is rather recent which started with D\"{o}rfler \cite{Dorfler} and was further studied in \cite{Binev,Morin,Nochetto,Mekchay,Cascon}. Besides convergence, optimality is another important issue in AFEM which was firstly addressed by Binev et al. \cite{Binev} and further studied by Stevenson (\cite{Stevenson,Rob}). The so-called D\"{o}rfler's marking proposed in \cite{Dorfler} and quasi-error introduced in \cite{Cascon} consisting of the sum of the energy error and the scaled estimator are crucial to prove the contraction of the errors and quasi-optimal cardinality of the standard AFEM which avoids marking for oscillation (\cite{Dorfler}) and circumvents the interior node property of mesh refinement (\cite{Morin,Nochetto}). 

AFEM also finds successful application in optimal control problems governed by partial differential equations, starting from Liu, Yan \cite{Liu2} and Becker, Kapp, Rannacher \cite{Becker}. In \cite{Becker} the authors proposed a dual-weighted goal-oriented adaptivity for optimal control problems while in \cite{Liu2} residual type a posteriori error estimates were derived. We refer to \cite{Hoppe,Hoppe1,Li,Liu5,Liu3,Liu4,LiuYan08book} for more details of recent advance. Recently, Kohls, R\"{o}sch and Siebert derived in \cite{Kohls} an error equivalence property which enables one to derive reliable and efficient a posteriori error estimators for the optimal control problems with either variational discretization or full control discretization. 

There also exist some attempts to prove the convergence of AFEM for optimal control problems. In \cite{Gaevskaya} the authors considered the piecewise constant approximation of the control variable and gave a error reduction property for the quadruplet $(u,y,p,\sigma)$, where $u,y,p$ denote the optimal control, state, and adjoint state variables and $\sigma$ the associated co-control variable. However, additional requirement on the strict complementarity of the continuous problem and non-degeneracy property of the discrete control problem are assumed and marking strategy is extended to include the discrete free boundary between the active and inactive control sets. In \cite{Mao} the authors viewed the control problems as a nonlinear elliptic system of the state and adjoint variables and gave a convergence proof for adaptive algorithm involving the marking of data oscillation. In  \cite{Kohls1} the authors proved that the sequence of adaptively generated discrete solutions converged to the true solutions  for optimal control problems, but obtained only the plain convergence of adaptive algorithm without convergence rate and optimality. In this paper we intend to give a rigorous convergence proof for the adaptive finite element algorithm of elliptic optimal control problem  in an optimal control framework. We want to stress that the AFEM adopted in the current paper uses D\"{o}rfler's marking (\cite{Dorfler}) and is  a standard algorithm in that it employs only the error indicators and does not use the oscillation indicators.

Inspired by the work \cite{Dai} of Dai, Xu and Zhou where the convergence and optimality of AFEM for elliptic eigenvalue problem are proved by exploiting the certain relationship between the finite element eigenvalue approximation and the associated finite element boundary value approximation, in this paper we will provide a rigorous convergence analysis of the adaptive finite element algorithm 
for the optimal control problems governed by linear elliptic equation. Under mild assumption on the  initial mesh from which the adaptive algorithm starts, we show that the energy norm errors of the state and adjoint state variables are equivalent to the boundary value approximations of the state and adjoint state equations up to a higher order term. Then based on the well-known convergence result of AFEM for elliptic boundary value problems, we are able to prove the convergence of AFEM for the optimal control problems (OCPs for short). To be more specific, the AFEM for OCPs is a contraction, for the sum of the energy errors and the scaled error estimators of the state $y$ and the adjoint state $p$, between two consecutive adaptive loops. We also show that the AFEM yields a decay rate of the energy errors of the state $y$ and the adjoint state $p$ plus oscillations of the state and adjoint state equations in terms of the number of degrees of freedom. This result is an improvement over the plain convergence result presented in \cite{Kohls1}.

The rest of the paper is organised as follows. In Section 2 we recall some well-known results on the convergence analysis of AFEM for elliptic boundary values problems. In Section 3 we introduce the finite element approximation of the optimal control problems and derive a posteriori error estimates. Adaptive finite element algorithm for the optimal control problems based on D\"{o}rfler's marking is also presented. In Section 4 we give a rigorous convergence analysis of the AFEM for optimal control problems and the quasi-optimal cardinality is proved in Section 5. Numerical experiments are carried out in Section 6 to validate our theoretical result. Finally, we give a conclusion in Section 7 and outlook the possible extensions and future work.

Let $\Omega\subset\mathbb{R}^d$ ($d=2,3$) be a bounded polygonal or polyhedral domain. We denote by $W^{m,q}(\Omega)$ the usual Sobolev space of order $m\geq 0$, $1\leq q<\infty$ with norm $\|\cdot\|_{m,q,\Omega}$ and seminorm $|\cdot|_{m,q,\Omega}$. For $q=2$ we denote $W^{m,q}(\Omega)$ by $H^m(\Omega)$ and $\|\cdot\|_{m,\Omega}=\|\cdot\|_{m,2,\Omega}$, which is a Hilbert space. Note that $H^0(\Omega)=L^2(\Omega)$ and $H_0^1(\Omega)=\{v\in H^1(\Omega):\ v=0\ \mbox{on}\ \partial\Omega\}$. We denote $C$ a generic positive constant which may stand for different values at its different occurrences but does not depend on mesh size. We use the symbol $A\lesssim B$ to denote $A\leq CB$ for some constant $C$ that is independent of mesh size.

\section{Preliminaries}
\setcounter{equation}{0}
In this section, we recall some well-known results on the adaptive finite element approximation to a linear elliptic boundary value problem, which are then used for the convergence analysis of AFEM for optimal control problems. Some of the results are collected from \cite{Cascon} and \cite{Dai}, see also \cite{He}. 

Consider the following second order elliptic equation
\begin{equation}\label{elliptic}
\left\{\begin{array}{llr}
Ly=f \quad&\mbox{in}\ \Omega, \\
 \ y=0  \quad &\mbox{on}\ \partial\Omega,
\end{array}\right.
\end{equation}
where $L$ is a linear second order elliptic operator of the following form:
\begin{eqnarray}
Ly:= -\sum\limits_{i,j=1}^d\frac{\partial}{\partial x_j}(a_{ij}\frac{\partial y}{\partial x_i})+cy.\nonumber
\end{eqnarray}
We denote $L^*$  the adjoint operator of $L$
\begin{eqnarray}
L^*y:= -\sum\limits_{i,j=1}^d\frac{\partial}{\partial x_j}(a_{ji}\frac{\partial y}{\partial x_i})+cy.\nonumber
\end{eqnarray}
Here $a_{ij}\in W^{1,\infty}(\Omega)$ $(i,j=1,\cdots,d)$ is symmetric, positive definite and $0\leq c<\infty$. We denote $A=(a_{ij})_{d\times d}$ and $A^*$ its adjoint. Let 
\begin{eqnarray}
a(y,v)=\int_\Omega\sum\limits_{i,j=1}^da_{ij}\frac{\partial y}{\partial x_i}\frac{\partial v}{\partial x_j}+cyv,\quad \forall y,v\in H_0^1(\Omega).\nonumber
\end{eqnarray}
It is clear that $a(\cdot,\cdot)$ is a bounded bilinear form over $H_0^1(\Omega)$ and defines a norm $\|\cdot\|_{a,\Omega}=\sqrt{a(\cdot,\cdot)}$ which is equivalent to $\|\cdot\|_{1,\Omega}$.

The standard weak form of (\ref{elliptic}) reads as follows: Find $y\in H_0^1(\Omega)$ such that
\begin{eqnarray}
a(y,v)=(f,v)\quad \forall v\in H_0^1(\Omega).\label{elliptic_weak}
\end{eqnarray}
For each $f\in H^{-1}(\Omega)$ the above problem admits a unique solution by the well-known Lax-Milgram theorem. Since the elliptic equation (\ref{elliptic_weak}) is linear with respect to the right hand side $f$, we can define
a linear solution operator $S : L^2(\Omega) \rightarrow H_0^1(\Omega)$ such that $y = Sf$. 

Let $\mathcal{T}_h$ be a regular triangulation of $\Omega$ such that
$\bar\Omega=\cup_{T\in\mathcal{T}_h}\bar T$. We assume that $\mathcal{T}_h$ is shape regular in the sense that: there exists a constant $\gamma^*>0$ such that $\frac{h_T}{\rho_T}\leq \gamma^*$ for all $T\in \mathcal{T}_h$, where $h_T$ denotes the diameter of $T$ and $\rho_T$ is the diameter of the biggest ball contained in $T$. We set $h=\max_{T\in \mathcal{T}_h}h_T$. In this paper, we use $\mathcal{E}_h$ to denote the set of interior faces (edges or sides) of $\mathcal{T}_h$ and $\# \mathcal{T}_h$ to denote the number of elements of $\mathcal{T}_h$. 

On $\mathcal{T}_h$
we construct a family of nested finite element spaces $V_{h}$ consisting of  piecewise linear and continuous polynomials 
such that $V_h\subset C(\bar\Omega)\cap H_0^1(\Omega)$. We define the standard Galerkin projection operator $\mathcal{R}_h: H_0^1(\Omega)\rightarrow V_h$ by (\cite{Ciarlet})
\begin{eqnarray}
a(y-\mathcal{R}_hy,v_h)=0\quad\forall v_h\in V_h,\label{Galerkin_proj}
\end{eqnarray} 
which satisfies the following stability result
\begin{eqnarray}
\|\mathcal{R}_hy\|_{a,\Omega}\lesssim \|y\|_{a,\Omega}\quad \forall y\in H_0^1(\Omega).\label{Ritz_stab}
\end{eqnarray}

A standard finite element approximation to (\ref{elliptic_weak}) can then be formulated as: Find $y_h\in V_h$ such that
\begin{eqnarray}
a(y_h,v_h)=(f,v_h)\quad \forall v_h\in V_h.\label{elliptic_weak_h}
\end{eqnarray}
Similarly, we can define a discrete solution operator $S_h : L^2(\Omega) \rightarrow V_h$ such that $y_h = S_hf$. Thus, we have $y_h=\mathcal{R}_hy=\mathcal{R}_hSf$.

For the following purpose, we follow the idea of \cite{Dai} to introduce the quantity $\kappa (h)$ as follows
\begin{eqnarray}
\kappa (h)=\sup\limits_{f\in L^2(\Omega),\|f\|_{0,\Omega}=1}\inf\limits_{v_h\in V_h}\|Sf-v_h\|_{a,\Omega}.
\end{eqnarray}
We note that the quantity $\kappa(h)$ is determined by the regularity of $Sf$ which is further influenced by the property of domain $\Omega$.  Indeed, if the boundary of $\Omega$ is smooth, like $\mathcal{C}^1$, the additional regularity $Sf\in H^2(\Omega)$ holds and thus $\kappa(h)=O(h)$. This is still true for polygonal or polyhedral boundaries if the domain is convex. The regularity is reduced, however, in the vicinity of nonconvex portions of polygonal or polyhedral boundaries. Grisvard proved in \cite{Grisvard2} the precise regularity results (Theorem 2.4.3 for the two-dimensional case and Corollary 2.6.7 for the three-dimensional case): there exists an $\varepsilon \in (0,{1\over 2}]$, which depends on the shape of the domain, such that  $Sf\in H^{{3\over 2}+\varepsilon}(\Omega)$ for each $f\in L^2(\Omega)$. Obviously, $\kappa(h)\ll 1$ for $h\in (0,h_0)$ if $h_0\ll 1$.

The following results are standard and can be found in, e.g., \cite{Ciarlet,Dai}
\begin{Proposition}For each $f\in L^2(\Omega)$, there hold
\begin{eqnarray}
\|Sf-S_hf\|_{a,\Omega}\lesssim \kappa(h)\|f\|_{0,\Omega}\label{proj_error}
\end{eqnarray}
and
\begin{eqnarray}
\|Sf-S_hf\|_{0,\Omega}\lesssim \kappa(h)\|Sf-S_hf\|_{a,\Omega}.\label{L2_error}
\end{eqnarray}
\end{Proposition}

Now we are in the position to review the residual type a posteriori error estimator for the finite element approximation of elliptic boundary value problem. We define the element residual $\tilde r_T(y_h)$ and the jump residual $\tilde j_E(y_h)$ by
\begin{eqnarray}
\tilde r_T(y_h):&=& f-Ly_h=f+\nabla\cdot(A\nabla y_h)-cy_h\quad \mbox{in}\ \mathcal{T}_h,\label{element_res}\\
\tilde j_E(y_h):&=&[A\nabla y_h]_E\cdot n_E\quad\mbox{on}\ E\in \mathcal{E}_h,\label{jump}
\end{eqnarray}
where $[A\nabla y_h]_E\cdot n_E$ denotes the jump of $A\nabla y_h$ across the common side $E$ of elements $T^+$ and $T^-$, $n_E$ denotes the outward normal oriented to $T^-$. For each element $T\in \mathcal{T}_h$, we define the local error indicator $\tilde \eta_{h}(y_h,T)$ by
\begin{eqnarray}
\tilde\eta_{h}(y_h,T):= \Big(h_T^2\|\tilde r_T(y_h)\|_{0,T}^2+\sum\limits_{E\in\mathcal{E}_h,E\subset\partial T}h_E\|\tilde j_E(y_h)\|_{0,E}^2\Big)^{1\over 2}.\label{local_elliptic}
\end{eqnarray}
Then on a subset $\omega\subset \Omega$, we define the error estimator $\tilde\eta_{h}(y_h,\omega)$ by
\begin{eqnarray}
\tilde\eta_{h}(y_h,\omega):=\Big(\sum\limits_{T\in\mathcal{T}_h,T\subset\omega}\tilde\eta_{h}^2(y_h,T)\Big)^{1\over 2}.\label{estimator_elliptic}
\end{eqnarray}
Thus, $\tilde\eta_{h}(y_h,\Omega)$ constitutes the error estimator on $\Omega$ with respect to $\mathcal{T}_h$. 

For $f\in L^2(\Omega)$ we also need to define the data oscillation as (see \cite{Morin,Nochetto})
\begin{eqnarray}
\mbox{osc}(f,T):=\|h_T(f-\bar f_T)\|_{0,T},\quad \mbox{osc}(f,\mathcal{T}_h):=\Big(\sum\limits_{T\in\mathcal{T}_h}\mbox{osc}^2(f,T)\Big)^{1\over 2},\label{data_osc}
\end{eqnarray}
where $\bar f_T$ denotes the $L^2$-projection of $f$ onto piecewise constant space on $T$. It is easy to see that 
\begin{eqnarray}
\mbox{osc}(f_1+f_2,\mathcal{T}_h)\leq \mbox{osc}(f_1,\mathcal{T}_h)+\mbox{osc}(f_2,\mathcal{T}_h),\quad \forall f_1,f_2\in L^2(\Omega).\label{osc_linear}
\end{eqnarray}
For the above defined data oscillation we have the following lemma whose proof can be found in \cite[Lemma 2.4]{Dai}.
\begin{Lemma}\label{Lm:2.4}
There exists a constant $C_*$ which depends on $A$, the mesh regularity constant $\gamma^*$ and coefficient $c$, such that
\begin{eqnarray}
\mbox{osc}(Lv,\mathcal{T}_h)\leq C_*\|v\|_{a,\Omega},\quad \mbox{osc}(L^*v,\mathcal{T}_h)\leq C_*\|v\|_{a,\Omega}\quad\forall v\in V_h.\label{osc_bound}
\end{eqnarray}
\end{Lemma}

Now we can formulate the following global upper and lower bounds for the a posteriori error estimators of elliptic boundary value problems (see, e.g., \cite{Dorfler,Verfuth}):
\begin{eqnarray}
\|y-y_h\|^2_{a,\Omega}&\leq& \tilde C_1\tilde \eta^2_{h}(y_h,\Omega),\label{elliptic_upper}\\
\tilde C_2\tilde \eta_{h}^2(y_h,\Omega)&\leq& \|y-y_h\|^2_{a,\Omega}+\tilde C_3 \mbox{osc}^2(f-Ly_h,\mathcal{T}_h).\label{elliptic_lower}
\end{eqnarray}

For our following purpose we also need to study the adjoint equation of elliptic boundary value problem (\ref{elliptic}). For each $g\in L^2(\Omega)$, let $p\in H_0^1(\Omega)$ be the solution of the following adjoint equation
\begin{eqnarray}
a(v,p)=(g,v)\quad\forall v\in H_0^1(\Omega)\label{elliptic_weak_adjoint}
\end{eqnarray}
with  its finite element approximation
\begin{eqnarray}
a(v_h,p_h)=(g,v_h)\quad\forall v_h\in V_h.\label{elliptic_weak_adjoint_h}
\end{eqnarray}
We can also give the a posteriori global upper and lower error bounds:
\begin{eqnarray}
\|p-p_h\|^2_{a,\Omega}&\leq& \tilde C_1\tilde \eta^2_{h}(p_h,\Omega),\label{elliptic_upper_adjoint}\\
\tilde C_2\tilde \eta_{h}^2(p_h,\Omega)&\leq& \|p-p_h\|^2_{a,\Omega}+\tilde C_3 \mbox{osc}^2(g-L^*p_h,\mathcal{T}_h).\label{elliptic_lower_adjoint}
\end{eqnarray}

To analyse the adaptive finite element approximation for optimal control problem, we introduce a system of two source problems associated with the state and adjoint state equations, which is some trivial extension for the existing results of adaptive finite element approximation of scalar problem (see \cite{Cascon}). Specifically, we introduce the adaptive finite element algorithm to solve a system of elliptic boundary value problems (\ref{elliptic_weak}) and (\ref{elliptic_weak_adjoint}). 
There are different kinds of adaptive algorithms which differ from the marking strategies (see \cite{Mekchay,Morin,Nochetto}). Here we follow the D\"{o}rfler's marking introduced in \cite{Dorfler} which marks only the error estimator and avoids the marking for oscillation:
 \begin{Algorithm}\label{Alg:1.0}The D\"{o}rfler's marking strategy for BVPs
\begin{enumerate}
\item Given a parameter $0<\theta<1$;

\item Construct a minimal subset $\tilde{\mathcal{T}}_{h}\subset \mathcal{T}_h$ such that
\begin{equation*}
\sum\limits_{T\in\tilde{\mathcal{T}}_{h}}\big(\tilde \eta_h^2(y_h,T)+\tilde \eta_h^2(p_h,T)\big)\geq \theta\big(\tilde\eta_h^2(y_h,\Omega)+\tilde \eta_h^2(p_h,\Omega)\big).
\end{equation*}

\item Mark all the elements in $\tilde{\mathcal{T}}_{h}$.

\end{enumerate}
\end{Algorithm}

The adaptive algorithm for solving elliptic boundary value problems can then be described as follows (see \cite{Cascon}):
\begin{Algorithm}\label{Alg:1.1}
Adaptive finite element method for BVPs:
\begin{enumerate}
\item Given an initial mesh $\mathcal{T}_{h_0}$ with mesh size $h_0$ and construct the finite element space $V_{h_0}$.

\item Set $k=0$, solve (\ref{elliptic_weak_h}) and (\ref{elliptic_weak_adjoint_h}) to obtain $(y_{h_k},p_{h_k})\in V_{h_k}\times V_{h_k}$.

\item Compute the local error indicators $\tilde\eta_{h_k}(y_{h_k},T)$ and $\tilde\eta_{h_k}(p_{h_k},T)$ for each $T\in\mathcal{T}_{h_k}$.

\item Construct $\tilde{ \mathcal{T}}_{h_k}\subset\mathcal{T}_{h_k}$ by the marking Algorithm \ref{Alg:1.0}.

\item Refine $\tilde{ \mathcal{T}}_{h_k}$ to get a new conforming mesh $\mathcal{T}_{h_{k+1}}$.

\item Construct the finite element space $V_{h_{k+1}}$, solve (\ref{elliptic_weak_h}) and (\ref{elliptic_weak_adjoint_h}) to obtain $(y_{h_{k+1}},p_{h_{k+1}})\in V_{h_{k+1}}\times V_{h_{k+1}}$.

\item Set $k=k+1$ and go to Step (3).

\end{enumerate}
\end{Algorithm}

We denote $\mathbb{T}$ the class of all conforming refinements by bisection of $\mathcal{T}_{h_0}$ (see \cite{Cascon} for more details). Given a fixed number $b\geq 1$, for any $\mathcal{T}_{h_k}\in\mathbb{T}$ and $\mathcal{M}_{h_k}\subset \mathcal{T}_{h_k}$ of marked elements,
\begin{eqnarray}
\mathcal{T}_{h_{k+1}}=\mbox{REFINE}(\mathcal{T}_{h_k},\mathcal{M}_{h_k})\nonumber
\end{eqnarray}
outputs a conforming triangulation $\mathcal{T}_{h_{k+1}}\in \mathbb{T}$, where at least all elements of $\mathcal{M}_{h_k}$ are bisected $b$ times. We define $R_{\mathcal{T}_{h_k}\rightarrow \mathcal{T}_{h_{k+1}}}=\mathcal{T}_{h_k}\backslash(\mathcal{T}_{h_k}\cap \mathcal{T}_{h_{k+1}})$ as the set of refined elements satisfies $\mathcal{M}_{h_k}\subset R_{\mathcal{T}_{h_k}\rightarrow \mathcal{T}_{h_{k+1}}}$.

Then we can formulate the following standard result on the complexity of refinement, see \cite[Lemma 2.3]{Cascon} and \cite{Rob} for more details.
\begin{Lemma}\label{Lm:cardi}
Assume that $\mathcal{T}_{h_0}$ verifies condition (b) of Section 4 in \cite{Rob}. Let $\mathcal{T}_{h_k}$ ($k\geq 0$) be a sequence of conforming and nested triangulations of $\Omega$ generated by REFINE starting from the initial mesh $\mathcal{T}_{h_0}$. 
Assume that $\mathcal{T}_{h_{k+1}}$ is generated from  $\mathcal{T}_{h_k}$ by $\mathcal{T}_{h_{k+1}}=\mbox{REFINE}(\mathcal{T}_{h_k},\mathcal{M}_{h_k})$ with a subset $\mathcal{M}_{h_k}\subset \mathcal{T}_{h_k}$. Then there exists a constant $\hat C_0$ depending on $\mathcal{T}_{h_0}$ and $b$ such that
\begin{eqnarray}
\#\mathcal{T}_{h_{k+1}}-\#\mathcal{T}_{h_0}\leq\hat C_0\sum\limits_{i=0}^k\# \mathcal{M}_{h_{i}}\quad\forall k\geq 1.\label{cardi}
\end{eqnarray}
\end{Lemma}

We define 
\begin{eqnarray}
\|(y,p)\|_a^2=a(y,y)+a(p,p).\nonumber
\end{eqnarray}
The convergence of Algorithm \ref{Alg:1.1} based on the marking Algorithm \ref{Alg:1.0} is proven in \cite{Cascon} and now becomes a standard theory for the convergence analysis of AFEM for different kind of boundary value problems. The following Theorem \ref{Thm:elliptic_convergence}, Lemma \ref{Lm:distance} and Lemma \ref{Lm:decay} are extensions of corresponding results for single elliptic equation in \cite{Cascon} by some primary operations. We remark that in \cite{Zhou} the authors used the similar idea to prove the convergence of adaptive finite element computations for multiple eigenvalues.
\begin{Theorem}\label{Thm:elliptic_convergence}
Let $(y_{h_k},p_{h_k})\in V_{h_k}\times V_{h_k}$ be a sequence of finite element solutions of problems (\ref{elliptic_weak}) and (\ref{elliptic_weak_adjoint}) based on the adaptively refined mesh $\mathcal{T}_{h_k}$ produced by Algorithm \ref{Alg:1.1}. Then there exist constants $\tilde \gamma>0$ and $\tilde\beta\in (0,1)$, depending only on the shape regularity of meshes, the data and the parameters used in Algorithm \ref{Alg:1.1}, such that for any two consecutive iterates $k$ and $k+1$ we have
\begin{eqnarray}
&&\|(y-y_{h_{k+1}},p-p_{h_{k+1}})\|_{a}^2+\tilde\gamma\big(\tilde \eta^2_{h_{k+1}}(y_{h_{k+1}},\Omega)+\tilde \eta^2_{h_{k+1}}(p_{h_{k+1}},\Omega)\big)\nonumber\\
&\leq& \tilde\beta^2\Big(\|(y-y_{h_{k}},p-p_{h_{k}})\|_{a}^2+\tilde\gamma\big(\tilde \eta^2_{h_{k}}(y_{h_{k}},\Omega)+\tilde \eta^2_{h_{k}}(p_{h_{k}},\Omega)\big)\Big).\label{elliptic_convergence}
\end{eqnarray}
Here 
\begin{eqnarray}
\tilde \gamma :=\frac{1}{(1+\delta^{-1})C_*^2}\label{tilde_gamma}
\end{eqnarray}
with some constant $\delta\in (0,1)$.
\end{Theorem}

To prove the optimal complexity of the adaptive algorithm we need further results. The following lemma presents a localised upper bound estimate for the distance between two nested solutions of the elliptic boundary value problems (\ref{elliptic_weak}) and (\ref{elliptic_weak_adjoint}) (see \cite[Lemma 3.6]{Cascon} and \cite[Lemma 6.2]{Dai}).
\begin{Lemma}\label{Lm:distance}
Let $(y_{h_k},p_{h_k})\in V_{h_k}\times V_{h_k}$ and $(y_{h_{k+1}},p_{h_{k+1}})\in V_{h_{k+1}}\times V_{h_{k+1}}$ be the discrete solutions of problems (\ref{elliptic_weak}) and (\ref{elliptic_weak_adjoint}) over a mesh $\mathcal{T}_{h_k}$ and its refinement $\mathcal{T}_{h_{k+1}}$ with marked element $\mathcal{M}_{h_{k}}\subset \mathcal{T}_{h_k}$. Let $R_{\mathcal{T}_{h_k}\rightarrow \mathcal{T}_{h_{k+1}}}$ be the set of refined elements. Then the following localised upper bound is valid
\begin{eqnarray}
\|(y_{h_k}-y_{h_{k+1}},p_{h_k}-p_{h_{k+1}})\|_{a}^2\leq\tilde C_1\sum\limits_{T\in R_{\mathcal{T}_{h_k}\rightarrow \mathcal{T}_{h_{k+1}}}}\big(\tilde \eta_{h_k}^2(y_{h_k},T)+\tilde \eta_{h_k}^2(p_{h_k},T)\big).\label{distance}
\end{eqnarray}
\end{Lemma}
Consequently, we can show the optimality of the D\"{o}rfler's marking strategy in the following 
lemma (see \cite[Lemma 5.9]{Cascon} and \cite[Proposition 6.3]{Dai} for the proof).
\begin{Lemma}\label{Lm:decay}
Let $(y_{h_k},p_{h_k})\in V_{h_k}\times V_{h_k}$ and $(y_{h_{k+1}},p_{h_{k+1}})\in V_{h_{k+1}}\times V_{h_{k+1}}$ be the discrete solutions of problems (\ref{elliptic_weak}) and (\ref{elliptic_weak_adjoint}) over a mesh $\mathcal{T}_{h_k}$ and its refinement $\mathcal{T}_{h_{k+1}}$ with marked element $\mathcal{M}_{h_{k}}\subset \mathcal{T}_{h_k}$. Suppose that they satisfy the energy decrease property
\begin{eqnarray}
&&\|(y-y_{h_{k+1}},p-p_{h_{k+1}})\|_{a}^2+\tilde \gamma_0\big(\mbox{osc}^2(f-Ly_{h_{k+1}},\mathcal{T}_{h_{k+1}})+\mbox{osc}^2(g-L^*p_{h_{k+1}},\mathcal{T}_{h_{k+1}})\big)\nonumber\\
&\leq&\tilde\beta_0^2\Big(\|(y-y_{h_{k}},p-p_{h_{k}})\|_{a}^2+\tilde \gamma_0\big(\mbox{osc}^2(f-Ly_{h_k},\mathcal{T}_{h_k})+\mbox{osc}^2(g-L^*p_{h_k},\mathcal{T}_{h_k})\big)\Big)\label{bvp_decay}
\end{eqnarray}
with $\tilde \gamma_0>0$ a constant and $\tilde\beta_0^2\in (0,{1\over 2})$. Then the set $R_{\mathcal{T}_{h_k}\rightarrow \mathcal{T}_{h_{k+1}}}$ of marked elements satisfies the D\"{o}rfler property
\begin{eqnarray}
\sum\limits_{T\in R_{\mathcal{T}_{h_k}\rightarrow \mathcal{T}_{h_{k+1}}}}\big(\tilde\eta_{h_k}^2(y_{h_k},T)+\tilde\eta_{h_k}^2(p_{h_k},T)\big)\geq\tilde\theta \sum\limits_{T\in\mathcal{T}_{h_k}}\big(\tilde\eta_{h_k}^2(y_{h_k},T)+\tilde\eta_{h_k}^2(p_{h_k},T)\big)\label{BVP_Dorfler}
\end{eqnarray}
with $\tilde\theta=\frac{\tilde C_2(1-2\tilde\beta_0^2)}{\tilde C_0(\tilde C_1+(1+2C_*^2\tilde C_1)\tilde\gamma_0)}$, where $\tilde C_0=\max(1,\frac{\tilde C_3}{\tilde\gamma_0})$.
\end{Lemma}

\section{Adaptive finite element method for optimal control problem}
\setcounter{equation}{0}
In this section we consider the following elliptic optimal control problem:
\begin{eqnarray}
\min\limits_{u\in U_{ad}}\ \ J(y,u)={1\over
2}\|y-y_d\|_{0,\Omega}^2 +
\frac{\alpha}{2}\|u\|_{0,\Omega}^2\label{OCP}
\end{eqnarray}
subject to
\begin{equation}\label{OCP_state}
\left\{\begin{array}{llr}
Ly=u \quad&\mbox{in}\
\Omega, \\
 \ y=0  \quad &\mbox{on}\ \partial\Omega,
\end{array}\right.
\end{equation}
where $\alpha>0$ is a fixed parameter, $U_{ad}$ is the admissible control set with bilateral control constraints:
\begin{eqnarray}
U_{ad}:= \Big\{u\in L^2(\Omega), \quad a\leq u\leq b\ \mbox{a.e.}\ \mbox{in}\  \Omega\Big\},\nonumber
\end{eqnarray}
where $a,b\in \mathbb{R}$ and $a<b$. 
\begin{Remark}\label{control_operator}
We remark that all the theories presented below can be generalised to the case that the control acts on a subdomain $\omega\subset\Omega$. In this case the governing equation reads $Ly=Bu$ with the control operator $B:L^2(\omega)\rightarrow L^2(\Omega)$ an extension by zero operator from $\omega$ to $\Omega$.
\end{Remark}
With the solution operator $S$ of elliptic equation (\ref{OCP_state}) introduced in last section, we can formulate a reduced optimization problem
\begin{eqnarray}
\min\limits_{u\in U_{ad}}\ \ \hat J(u):=J(Su,u)={1\over
2}\|Su-y_d\|_{0,\Omega}^2 +
\frac{\alpha}{2}\|u\|_{0,\Omega}^2.\label{OCP_reduced}
\end{eqnarray}
Since the above optimization problem is linear and strictly convex, there exists a unique solution
 $u\in U_{ad}$ by standard argument (see \cite{Lions}). Moreover, the first order necessary and sufficient optimality condition can be stated as follows:
\begin{equation}\label{reduced_opt}
\hat J'(u)(v-u)=(\alpha u+S^*(Su-y_d),v-u)\geq 0, \
\ \ \ \ \forall v\in U_{ad},
\end{equation}
where $S^*$ is the adjoint of $S$ (\cite{Hinze09book}). Introducing the adjoint state $p:=S^*(Su-y_d)\in H_0^1(\Omega)$, we are led to the following optimality system:
\begin{equation}\label{OCP_OPT}
\left\{\begin{array}{llr}a(y,v)=(u,v),\ \ &\forall v\in H_0^1(\Omega),\\
a(w,p)=(y-y_d,w),\ \ &\forall w\in H_0^1(\Omega),\\
(\alpha u+p,v-u)\geq 0, \ \ &\forall v\in U_{ad}.
\end{array} \right.
\end{equation}
Hereafter, we call $u$, $y$ and $p$ the optimal control, state and adjoint state, respectively. From the last inequality of (\ref{OCP_OPT}) we have the pointwise representation of $u$ (see \cite{Lions}):
\begin{eqnarray}
u(x)=P_{[a,b]}\Big\{-\frac{1}{\alpha}p(x)\Big\},\label{projection}
\end{eqnarray}
where $P_{[a,b]}$ is the orthogonal projection operator from $L^2(\Omega)$ to
$U_{ad}$.

Next, let us consider the finite element approximation of
(\ref{OCP})-(\ref{OCP_state}). In this paper, we use the piecewise linear finite elements
 to approximate the state $y$, and variational discretization for the optimal
 control $u$ (see \cite{Hinze05COAP}). Based on the finite element space $V_h$,
we can define the finite dimensional approximation to the optimal
control problem (\ref{OCP})-(\ref{OCP_state}) as follows: Find
$(u_h,y_h)\in U_{ad}\times V_h$ such that
\begin{eqnarray}\label{OCP_h}
\min\limits_{u_h\in U_{ad}}J_h( y_h,u_h)={1\over
2}\|y_h-y_d\|_{0,\Omega}^2 +
\frac{\alpha}{2}\|u_h\|_{0,\Omega}^2
\end{eqnarray}
subject to
\begin{eqnarray}\label{OCP_state_h}
a(y_h,v_h)=(f+ u_h,v_h),\ \ \ \ \forall v_h\in V_h.
\end{eqnarray}
Similar to the continuous case we have $y_h=S_hu_h$. With this notation we can formulate a reduced discrete optimization problem
\begin{eqnarray}\label{OCP_Operator_h}
\min\limits_{u_h\in U_{ad}} \hat J_h(u_h):=J_h(S_hu_h,u_h)={1\over
2}\|S_hu_h-y_d\|_{0,\Omega}^2 +
\frac{\alpha}{2}\|u_h\|_{0,\Omega}^2.
\end{eqnarray}
We note that the above optimization problem can be solved by projected gradient method
or semi-smooth Newton method, see \cite{Hintermueller03SIOPT}, \cite{Hinze09book} and \cite{LiuYan08book} for more details.

Similar to the continuous problem (\ref{OCP})-(\ref{OCP_state}),
the above discretized optimization problem also admits a unique
solution ${u}_h\in U_{ad}$. Moreover, the first order necessary and sufficient
 optimality condition can be stated as follows:
\begin{eqnarray}\label{reduced_opt_h}
\hat J'_h( u_h)(v_h- u_h)=(\alpha u_h+S^*_h(S_h u_h-y_d),v_h-u_h)\geq 0, \ \forall v_h\in U_{ad},
\end{eqnarray}
where $S^*_h$ is the adjoint of $S_h$. Introducing the adjoint state $p_h:=S^*_h(S_hu_h-y_d)\in V_h$, the discretized first order necessary and sufficient optimality
condition is equivalent to:
\begin{equation}\label{OCP_OPT_h}
\left\{\begin{array}{llr}a(y_h,v_h)=( u_h,v_h),\ \ &\forall v_h\in V_h,\\
 a(w_h,p_h)=(y_h-y_d,w_h),\ \ &\forall w_h\in V_h,\\
 (\alpha u_h+ p_h,v_h-u_h)\geq 0, \ \ \ &\forall v_h\in U_{ad}.
\end{array} \right.
\end{equation}
Hereafter, we call $u_h$, $y_h$ and $p_h$ the discrete optimal control, state and adjoint state, respectively. Similar to the continuous case (\ref{projection}) we have
\begin{eqnarray}
u_h(x)=P_{U_{ad}}\Big\{-\frac{1}{\alpha}p_h(x)\Big\}.\label{ph_to_uh}
\end{eqnarray}
It should be noticed that $u_h$ is not generally a finite element function in $V_h$.

For convenience we define $y^h:=Su_h$ and $p^h:=S^*(S_hu_h-y_d)$. It is obvious that $y_h$ and $p_h$ are the standard Galerkin projections of $y^h$ and $p^h$, i.e., $y_h=\mathcal{R}_hy^h$ and $p_h=\mathcal{R}_hp^h$. The following equivalence property is established in \cite{Kohls}.
\begin{Theorem}\label{Thm:2.1}
Let $(u,y,p)\in U_{ad}\times H_0^1(\Omega)\times H_0^1(\Omega)$ and
 $( u_h, y_h, p_h)\in U_{ad}\times V_h\times V_h$ be the
 solutions of problems (\ref{OCP})-(\ref{OCP_state}) and
 (\ref{OCP_h})-(\ref{OCP_state_h}), respectively.
 Then the following an equivalence property holds:
\begin{eqnarray}\label{u_equivalence}
\|u-u_h\|_{0,\Omega}+\|y-y_h\|_{a,\Omega}
+\|p- p_h\|_{a,\Omega}\approx \|y^h- y_h\|_{a,\Omega}
+\|p^h- p_h\|_{a,\Omega}.
\end{eqnarray}
\end{Theorem}
\begin{proof}
For completeness we include a brief proof. Setting $v= u_h$ in (\ref{reduced_opt}) and $v_h = u$ in (\ref{reduced_opt_h}) we are led to
\begin{eqnarray}
(\alpha u+S^*(Su-y_d),u_h-u)\geq 0, \label{u_to_p1}\\
(\alpha u_h+S^*_h(S_hu_h-y_d),u-u_h)\geq 0. \label{u_to_p2}
\end{eqnarray}
Adding the above two inequalities, we conclude from (\ref{OCP_OPT}) and (\ref{OCP_OPT_h}) that
\begin{eqnarray}
&&\alpha\|u- u_h\|^2_{0,\Omega}\leq (S^*_h(S_hu_h-y_d)-S^*(Su-y_d),u- u_h)\nonumber\\
&=& (S^*_h(S_h u_h-y_d)-S^*(S_hu_h-y_d),u- u_h)+(S^*(S_hu_h-y_d)-S^*(Su-y_d),u-u_h)\nonumber\\
&=&(S^*_h(S_hu_h-y_d)-S^*(S_hu_h-y_d),u-u_h)+(S_hu_h-Su,Su-Su_h)\nonumber\\
&=&(S^*_h(S_h u_h-y_d)-S^*(S_hu_h-y_d),u-u_h)+(S_hu_h-Su,Su-S_hu_h)\nonumber\\
&&+(S_h u_h-Su,S_h u_h-Su_h).\label{u_rep}
\end{eqnarray}
It follows from the $\varepsilon$-Young inequality that
\begin{eqnarray}
\alpha\|u-u_h\|^2_{0,\Omega}\leq  C\|Su_h-S_h u_h\|^2_{a,\Omega}+C\|S^*(S_hu_h-y_d)-S^*_h(S_hu_h-y_d)\|_{a,\Omega}^2.\label{u_error}
\end{eqnarray}
Moreover, we have
\begin{eqnarray}
\|y-y_h\|_{a,\Omega}&\leq& \|y-Su_h\|_{a,\Omega}+\|Su_h-y_h\|_{a,\Omega}\nonumber\\
&\leq& C\|u-u_h\|_{0,\Omega}+\|Su_h-y_h\|_{a,\Omega}\nonumber
\end{eqnarray}
and
\begin{eqnarray}
\|p- p_h\|_{a,\Omega}&\leq& \|p-S^*(S_hu_h-y_d)\|_{a,\Omega}+\|S^*(S_hu_h-y_d)- p_h\|_{a,\Omega}\nonumber\\
&\leq& \|Su-S_hu_h\|_{0,\Omega}+\|S^*(S_hu_h-y_d)-p_h\|_{a,\Omega}\nonumber\\
&\leq& C\|u-u_h\|_{0,\Omega}+\|S^*(S_hu_h-y_d)-p_h\|_{a,\Omega}+\|S u_h-y_h\|_{a,\Omega}.\nonumber
\end{eqnarray}
Combining the above estimates we prove the upper bound. 

Now we prove the lower bound. Note that
\begin{eqnarray}
\|S u_h-S_hu_h\|_{a,\Omega}&\leq&\|Su_h-Su\|_{a,\Omega}+\|Su-S_hu_h\|_{a,\Omega}\nonumber\\
&\leq&C\| u- u_h\|_{0,\Omega}+\| y- y_h\|_{a,\Omega}.\label{y_error}
\end{eqnarray}
Similarly, we can derive that
\begin{eqnarray}
&&\|S^*(S_hu_h-y_d)-S^*_h(S_hu_h-y_d)\|_{a,\Omega}\nonumber\\
&\leq&\|S^*(S_hu_h-y_d)-S^*(Su-y_d)\|_{a,\Omega}+\|S^*(Su-y_d)-S^*_h(S_hu_h-y_d)\|_{a,\Omega}\nonumber\\
&\leq&\|S_hu_h-Su\|_{0,\Omega}+\|p- p_h\|_{a,\Omega}\nonumber\\
&=&\|y-y_h\|_{a,\Omega}+\| p- p_h\|_{a,\Omega}.\label{p_error}
\end{eqnarray}
Thus, we can conclude from the above estimates the lower bound. This completes the proof.
\end{proof}

Next, we will prove a compact equivalence property which shows the certain relationship between the finite element optimal control approximation and the associated finite element boundary value approximation. 
\begin{Theorem}\label{Thm:2.2}
Let $h\in (0,h_0)$, $(u,y,p)\in U_{ad}\times H_0^1(\Omega)\times H_0^1(\Omega)$ and
 $(u_h,y_h,p_h)\in U_{ad}\times V_h\times V_h$ be the
 solutions of problems (\ref{OCP})-(\ref{OCP_state}) and
 (\ref{OCP_h})-(\ref{OCP_state_h}), respectively.
 Then the following equivalence properties hold
\begin{eqnarray}
\|y-y_h\|_{a,\Omega}&=& \|y^h-y_h\|_{a,\Omega}+O(\kappa(h))\Big(\|y-y_h\|_{a,\Omega}
+\|p- p_h\|_{a,\Omega}\Big),\label{equivalence_1}\\
\|p-p_h\|_{a,\Omega}
&=& \|p^h- p_h\|_{a,\Omega} +O(\kappa(h))\Big(\|y-y_h\|_{a,\Omega}
+\|p- p_h\|_{a,\Omega}\Big)\label{equivalence_2}
\end{eqnarray}
provided $h_0\ll 1$.
\end{Theorem}
\begin{proof}
It is obvious that 
\begin{eqnarray}\label{split}
y-y_h=y^h- y_h+y-y^h,\quad p-p_h=p^h- p_h+p-p^h.
\end{eqnarray}
Moreover, it follows from the stability results of elliptic equation that
\begin{eqnarray}
\|y-y^h\|_{a,\Omega}\leq C\|u-u_h\|_{0,\Omega},\quad \|p- p^h\|_{a,\Omega}\leq  C\|y-y_h\|_{0,\Omega}.\label{stab}
\end{eqnarray}
In the following we estimate $\|y-y_h\|_{0,\Omega}$. Let $\psi\in H_0^1(\Omega)$ be the solution of the following auxiliary problem
\begin{equation}\label{elliptic_auxi}
\left\{\begin{array}{llr}
L^*\psi= y-y_h\quad&\mbox{in}\ \Omega, \\
 \ \psi=0  \quad &\mbox{on}\ \partial\Omega.
\end{array}\right.
\end{equation}
 Let $\psi_h\in V_h$ be the finite element approximation of $\psi$. Then we can conclude from (\ref{proj_error}) and the standard duality argument (see, e.g., \cite{Ciarlet}) that
\begin{eqnarray}
\|y-y_h\|_{0,\Omega}^2&=&a(y-y_h,\psi)\nonumber\\
&=&a(y-y_h,\psi-\psi_h)+a(y-y_h,\psi_h)\nonumber\\
&=&a(y-y_h,\psi-\psi_h)+(u-u_h,\psi_h-\psi)+(u-u_h,\psi)\nonumber\\
&\leq&C\Big(\kappa(h)\|y-y_h\|_{a,\Omega}+\|u-u_h\|_{0,\Omega}\Big)\|y-y_h\|_{0,\Omega},\nonumber
\end{eqnarray}
which in turn implies 
\begin{eqnarray}
\|y- y_h\|_{0,\Omega}\leq  C\kappa(h)\|y-y_h\|_{a,\Omega}+C\|u-u_h\|_{0,\Omega}.\label{stab1}
\end{eqnarray}
Considering (\ref{stab}) we have
\begin{eqnarray}
\|p- p^h\|_{a,\Omega}\leq C\kappa(h)\|y-y_h\|_{a,\Omega}+C\|u-u_h\|_{0,\Omega}.\label{stab_p}
\end{eqnarray}
It remains to estimate $\|u-u_h\|_{0,\Omega}$. Note that it follows from (\ref{u_to_p1}) and (\ref{u_to_p2}) that
\begin{eqnarray}
\alpha\|u- u_h\|^2_{0,\Omega}&\leq& (S^*_h(S_hu_h-y_d)-S^*(Su-y_d),u- u_h)\nonumber\\
&=& (S^*_h(S_h u_h-y_d)-S^*_h(S_hu-y_d),u- u_h)\nonumber\\
&&+(S^*_h(S_hu-y_d)-S^*(Su-y_d),u-u_h)\nonumber\\
&=&(S_h(u_h-u),S_h(u-u_h))+(S^*_h(S_hu-y_d)-S^*(Su-y_d),u-u_h)\nonumber\\
&\leq&(S^*_h(S_hu-y_d)-S^*(Su-y_d),u-u_h),\nonumber
\end{eqnarray}
which yields 
\begin{eqnarray}
\|u-u_h\|_{0,\Omega}\leq C\|S^*_h(S_hu-y_d)-S^*(Su-y_d)\|_{0,\Omega}.\label{u_error}
\end{eqnarray}
Let $\phi\in H_0^1(\Omega)$ be the solution of the following auxiliary problem
\begin{equation}\label{elliptic_auxi}
\left\{\begin{array}{llr}
L\phi= S^*_h(S_hu-y_d)-S^*(Su-y_d)\quad&\mbox{in}\ \Omega, \\
 \ \phi=0  \quad &\mbox{on}\ \partial\Omega.
\end{array}\right.
\end{equation}
Then from the standard duality argument we have
\begin{eqnarray}
&&\|S^*_h(S_hu-y_d)-S^*(Su-y_d)\|_{0,\Omega}^2=a(\phi,S^*_h(S_hu-y_d)-S^*(Su-y_d))\nonumber\\
&=&a(\phi-\phi_h,S^*_h(S_hu-y_d)-S^*(Su-y_d))+a(\phi_h,S^*_h(S_hu-y_d)-S^*(Su-y_d))\nonumber\\
&=&a(\phi-\phi_h,S^*_h(S_hu-y_d)-S^*(Su-y_d))+(\phi_h,S_hu-Su)\nonumber\\
&=&a(\phi-\phi_h,S^*_h(S_hu-y_d)-S^*(Su-y_d))+(\phi_h-\phi,S_hu-Su)+(\phi,S_hu-Su),\label{u_1}
\end{eqnarray}
where $\phi_h\in V_h$ is the finite element approximation of $\phi$. We can conclude from (\ref{proj_error})-(\ref{L2_error}) that
\begin{eqnarray}
&&a(\phi-\phi_h,S^*_h(S_hu-y_d)-S^*(Su-y_d))\nonumber\\
&\leq& C\kappa(h)\|S^*_h(S_hu-y_d)-S^*(Su-y_d)\|_{0,\Omega}\|S^*_h(S_hu-y_d)-S^*(Su-y_d)\|_{a,\Omega}\label{u_2}
\end{eqnarray}
and 
\begin{eqnarray}
(\phi_h-\phi,S_hu-Su)\leq C\kappa^2(h)\|S^*_h(S_hu-y_d)-S^*(Su-y_d)\|_{0,\Omega}\|S_hu-Su\|_{a,\Omega},\label{u_3}\\
(\phi,S_hu-Su)\leq C\kappa(h)\|S^*_h(S_hu-y_d)-S^*(Su-y_d)\|_{0,\Omega}\|S_hu-Su\|_{a,\Omega}.\label{u_4}
\end{eqnarray}
Then we are able to derive that
\begin{eqnarray}
&&\|S^*_h(S_hu-y_d)-S^*(Su-y_d)\|_{0,\Omega}\nonumber\\
&\leq&C\kappa(h)(\|S^*_h(S_hu-y_d)-S^*(Su-y_d)\|_{a,\Omega}+\|S_hu-Su\|_{a,\Omega}).\label{u_44}
\end{eqnarray}
Combining (\ref{u_error}) and (\ref{u_44}) we are led to
\begin{eqnarray}
\|u-u_h\|_{0,\Omega}
 &\lesssim&\kappa(h)(\|S^*_h(S_hu-y_d)-S^*(Su-y_d)\|_{a,\Omega}+\|S_hu-Su\|_{a,\Omega})\nonumber\\
 &\lesssim&\kappa(h)(\|p_h-p\|_{a,\Omega}+\|S^*_h(S_hu-y_d)-S^*_h(S_hu_h-y_d)\|_{a,\Omega}+\|S_hu-Su\|_{a,\Omega})\nonumber\\
 &\lesssim&\kappa(h)(\|p_h-p\|_{a,\Omega}+\|S_hu-S_hu_h\|_{a,\Omega}+\|S_hu-Su\|_{a,\Omega})\nonumber\\
 &\lesssim&\kappa(h)(\|p_h-p\|_{a,\Omega}+\|S_hu_h-Su\|_{a,\Omega}+\|S_hu-S_hu_h\|_{a,\Omega})\nonumber\\
 &\lesssim&\kappa(h)(\|p_h-p\|_{a,\Omega}+\|y_h-y\|_{a,\Omega}+\|u-u_h\|_{0,\Omega}).\label{u_5}
\end{eqnarray}
If $h_0\ll 1$ then $\kappa(h)\ll 1$ for all $h\in (0,h_0)$, and we arrive at
\begin{eqnarray}
\|u-u_h\|_{0,\Omega}\lesssim\kappa(h)(\|p_h-p\|_{a,\Omega}+\|y_h-y\|_{a,\Omega}).\label{lift}
\end{eqnarray}
Inserting the above estimate into (\ref{stab}) and (\ref{stab_p}), we can conclude from (\ref{split}) the desired results  (\ref{equivalence_1})-(\ref{equivalence_2}). This completes the proof.
\end{proof}

Now we are in the position to consider the adaptive finite element method for optimal control problem  (\ref{OCP})-(\ref{OCP_state}). At first we will derive a posteriori error estimates for above optimal control problems. To begin with, we firstly introduce some notations. Similar to the definitions of (\ref{element_res}) and (\ref{jump}) we define the element residuals $r_{y,T}(y_h)$, $r_{p,T}(p_h)$ and the jump residuals $j_{y,E}(y_h)$, $j_{p,E}(p_h)$ by
\begin{eqnarray}
r_{y,T}(y_h):&=& u_h-Ly_h=u_h+\nabla\cdot(A\nabla y_h)-cy_h\quad \mbox{in}\ \mathcal{T}_h,\label{element_res_y}\\
r_{p,T}(p_h):&=& y_h-y_d-L^*p_h=y_h-y_d+\nabla\cdot(A^*\nabla p_h)-cp_h\quad \mbox{in}\ \mathcal{T}_h,\label{element_res_p}\\
j_{y,E}(y_h):&=&[A\nabla y_h]_E\cdot n_E\quad\mbox{on}\ E\in \mathcal{E}_h,\label{jump_y}\\
j_{p,E}(p_h):&=&[A^*\nabla p_h]_E\cdot n_E\quad\mbox{on}\ E\in \mathcal{E}_h.\label{jump_p}
\end{eqnarray}
For each element $T\in \mathcal{T}_h$, we define the local error indicators $\eta_{y,h}(y_h,T)$ and $\eta_{p,h}(p_h,T)$ by
\begin{eqnarray}
\eta_{y,h}(y_h,T):= \Big(h_T^2\|r_{y,T}(y_h)\|_{0,T}^2+\sum\limits_{E\in\mathcal{E}_h,E\subset\partial T}h_E\|j_{y,E}(y_h)\|_{0,E}^2\Big)^{1\over 2},\label{local_y}\\
\eta_{p,h}(p_h,T):= \Big(h_T^2\|r_{p,T}(p_h)\|_{0,T}^2+\sum\limits_{E\in\mathcal{E}_h,E\subset\partial T}h_E\|j_{p,E}(p_h)\|_{0,E}^2\Big)^{1\over 2}.\label{local_p}
\end{eqnarray}
Then on a subset $\omega\subset \Omega$, we define the error estimators $\eta_{y,h} (y_h,\omega)$ and $\eta_{p,h}(p_h,\omega)$ by
\begin{eqnarray}
\eta_{y,h}(y_h,\omega):=\Big(\sum\limits_{T\in\mathcal{T}_h,T\subset\omega}\eta_{y,h}^2(y_h,T)\Big)^{1\over 2},\label{estimator_y} \\
 \eta_{p,h}(p_h,\omega):=\Big(\sum\limits_{T\in\mathcal{T}_h,T\subset\omega}\eta_{p,h}^2( p_h,T)\Big)^{1\over 2}.\label{estimator_p}
\end{eqnarray}
Thus, $\eta_{y,h}(y_h,\Omega)$ and $\eta_{p,h}(p_h,\Omega)$ constitute the error estimators for the state equation and the adjoint state equation on $\Omega$ with respect to $\mathcal{T}_h$. 

Note that $S_hu_h$ and $S^*_h(S_hu_h-y_d)$ are the standard Galerkin projections of $Su_h$ and $S^*(S_hu_h-y_d)$, respectively.  Similar to (\ref{elliptic_upper})-(\ref{elliptic_lower}), standard a posterior error estimates for elliptic boundary value problem give the following upper bounds (see, e.g., \cite{Verfuth}) which show the reliability of the error estimators.
\begin{Lemma}\label{Lm: upper_bound}
Let $S$ and $S_h$ be the continuous and discrete solution operators defined above. Then the following a posteriori error estimates hold
\begin{eqnarray}
\|S u_h-S_hu_h\|_{a,\Omega}^2\leq \tilde C_1\eta_{y,h}^2(y_h,\Omega),\label{state_estimator}\\
\|S^*(S_hu_h-y_d)-S^*_h(S_hu_h-y_d)\|^2_{a,\Omega}\leq \tilde C_1\eta_{p,h}^2(p_h,\Omega).\label{adjoint_estimator}
\end{eqnarray}
\end{Lemma}
Then we can also derive the following global a posteriori error lower bounds, i.e., the global efficiency of the error estimators.
\begin{Lemma}\label{Lm: lower_bound}
Let $S$ and $S_h$ be the continuous and discrete solution operators defined above. Then the following a posteriori error lower bounds hold
\begin{eqnarray}
\tilde C_2\eta_{y,h}^2(y_h, \Omega)&\leq& \|S u_h-S_hu_h\|^2_{a,\Omega}+ \tilde C_3\mbox{osc}^2(u_h-Ly_h,\mathcal{T}_h),\label{state_lower}\\
\tilde C_2\eta_{p,h}^2( p_h,\Omega)&\leq& \|S^*(S_h u_h-y_d)-S^*_h(S_hu_h-y_d)\|^2_{a,\Omega}\nonumber\\
&&+ \tilde C_3\mbox{osc}^2(y_h-y_d-L^*p_h,\mathcal{T}_h).\label{adjoint_lower}
\end{eqnarray}
\end{Lemma}

Let $h_0\in (0,1)$ be the mesh size of the initial mesh $\mathcal{T}_{h_0}$ and define
\begin{eqnarray}
\tilde\kappa(h_0):=\sup\limits_{h\in (0,h_0]}\kappa(h).\nonumber
\end{eqnarray}
It is obvious that $\tilde\kappa(h_0)\ll 1$ if $h_0\ll 1$. For ease of exposition we also define the following quantities: 
\begin{eqnarray}
\eta_h^2((y_h,p_h),T)&=&\eta_{y,h}^2(y_{h},T)+\eta_{p,h}^2(p_{h},T),\nonumber\\
\mbox{osc}^2((y_h,p_h),T)&=&\mbox{osc}^2(u_h-Ly_h,T)+\mbox{osc}^2(y_h-y_d-L^*p_h,T),\nonumber
\end{eqnarray} 
and the straightforward modifications for $\eta_h^2((y_h,p_h),\Omega)$ and $\mbox{osc}^2((y_h,p_h),\mathcal{T}_h)$.

Now we state the following a posteriori error estimates for the finite element approximation
of the optimal control problem.
\begin{Theorem}\label{Thm:2.3}
Let $h\in (0,h_0)$. Assume that $(u,y,p)\in U_{ad}\times H_0^1(\Omega)\times H_0^1(\Omega)$ and $(u_h,y_h,p_h)\in U_{ad}\times V_h\times V_h$ are the
 solutions of problems (\ref{OCP})-(\ref{OCP_state}) and
 (\ref{OCP_h})-(\ref{OCP_state_h}), respectively.
 Then there exist positive constants $C_1$, $C_2$ and $C_3$, independent of the mesh size $h$, such that
\begin{eqnarray}\label{u_upper}
\|(y-y_h,p- p_h)\|_{a}^2\leq C_1\eta_{h}^2( (y_h,p_h),\Omega)
\end{eqnarray}
and
\begin{eqnarray}\label{u_lower}
C_2\eta_{h}^2((y_h,p_h),\Omega)\leq\|(y-y_h,p-p_h)\|_{a}^2+C_3\mbox{osc}^2((y_h,p_h),\mathcal{T}_h)
\end{eqnarray}
provided $h_0\ll 1$.
\end{Theorem}
\begin{proof}
Note that $y^h = Su_h$, $y_h=S_hu_h$, $p^h=S^*(S_h u_h-y_d)$ and $p_h=S^*_h(S_h u_h-y_d)$. From the estimates (\ref{equivalence_1})-(\ref{equivalence_2}), Lemmas \ref{Lm: upper_bound} and \ref{Lm: lower_bound} we have
\begin{eqnarray}
\|(y-y_h,p-p_h)\|_{a}^2&\leq& 2(\|y^h-y_h\|_{a,\Omega}^2+\|p^h-p_h\|_{a,\Omega}^2)+\hat C_1\kappa^2(h)\|(y-y_h,p- p_h)\|_{a}^2\nonumber\\
&\leq&2\tilde C_1\eta_{h}^2((y_h,p_h),\Omega)+\hat C_1\tilde\kappa^2(h_0)\|(y-y_h,p- p_h)\|_{a}^2\nonumber
\end{eqnarray}
and 
\begin{eqnarray}
\tilde C_2\eta_{h}^2((y_h,p_h),\Omega)&\leq& (\|y^h-y_h\|_{a,\Omega}^2+\|p^h-p_h\|_{a,\Omega}^2) +\tilde C_3\mbox{osc}^2((y_h,p_h),\mathcal{T}_h)\nonumber\\
&\leq&\|(y-y_h,p-p_h)\|_{a,\Omega}^2 +\tilde C_3\mbox{osc}^2((y_h,p_h),\mathcal{T}_h)\nonumber\\
&&+\hat C_2\tilde\kappa^2(h_0)\|(y-y_h,p- p_h)\|_{a}^2.\nonumber
\end{eqnarray}
We obtain the desired results by choosing 
\begin{eqnarray}
C_1=\frac{2\tilde C_1}{1-\hat C_1\tilde\kappa^2(h_0)},\quad C_2=\frac{\tilde C_2}{1+\hat C_2\tilde\kappa^2(h_0)},\quad C_3=\frac{\tilde C_3}{1+\hat C_2\tilde\kappa^2(h_0)}.\label{C123}
\end{eqnarray}
\end{proof}

The adaptive finite element procedure consists of the following loops
$$\mbox{SOLVE}\rightarrow \mbox{ESTIMATE}\rightarrow \mbox{MARK}\rightarrow \mbox{REFINE}.$$
The ESTIMATE step is based on the a posteriori error estimators presented in Theorem \ref{Thm:2.3}, while the step REFINE can be done by using iterative or recursive bisection of elements with the minimal refinement condition (see \cite{Stevenson,Verfuth}). Due to \cite{Cascon}, the procedure REFINE here is not required to satisfy the interior node property of \cite{Morin}. Note that there are two error estimators $\eta_{y,h}(y_h,T)$ and $\eta_{p,h}(p_h,T)$ contributed to the state approximation and adjoint state approximation, respectively. We use the sum of the two estimators as our indicators for the marking strategy. The marking algorithm based on the D\"{o}rfler's strategy for optimal control problems can be described as follows
\begin{Algorithm}\label{Alg:3.0}The D\"{o}rfler's marking strategy for OCPs
\begin{enumerate}
\item Given a parameter $0<\theta<1$;

\item Construct a minimal subset $\tilde{\mathcal{T}}_{h}\subset \mathcal{T}_h$ such that
\begin{equation*}
\sum\limits_{T\in\tilde{\mathcal{T}}_{h}}\eta_{h}^2((y_h,p_h),T)\geq \theta\eta_{h}^2((y_h,p_h),\Omega).
\end{equation*}

\item Mark all the elements in $\tilde{\mathcal{T}}_{h}$.

\end{enumerate}
\end{Algorithm}

Then we can present the adaptive finite element algorithm for the optimal
control problem (\ref{OCP_h})-(\ref{OCP_state_h})  as follows:
\begin{Algorithm}\label{Alg:3.1}
Adaptive finite element algorithm for OCPs:
\begin{enumerate}
\item Given an initial mesh $\mathcal{T}_{h_0}$ with mesh size $h_0$ and construct the finite element space $V_{h_0}$.

\item Set $k=0$ and solve the optimal control problem (\ref{OCP_h})-(\ref{OCP_state_h}) to obtain 
$(u_{h_k},y_{h_k},p_{h_k})\in U_{ad}\times V_{h_k}\times V_{h_k}$.

\item Compute the local error indicator $\eta_{h_k}((y_{h_k},p_{h_k}),T)$.

\item Construct $\tilde{ \mathcal{T}}_{h_k}\subset\mathcal{T}_{h_k}$ by the marking Algorithm \ref{Alg:3.0}.

\item Refine $\tilde{ \mathcal{T}}_{h_k}$ to get a new conforming mesh $\mathcal{T}_{h_{k+1}}$ by procedure REFINE.

\item Construct the finite element space $V_{h_{k+1}}$ and solve the optimal control problem (\ref{OCP_h})-(\ref{OCP_state_h}) to obtain 
$(u_{h_{k+1}},y_{h_{k+1}},p_{h_{k+1}})\in U_{ad}\times V_{h_{k+1}}\times V_{h_{k+1}}$.

\item Set $k=k+1$ and go to Step (3).

\end{enumerate}
\end{Algorithm}

\section{Convergence of AFEM for optimal control problem}
\setcounter{equation}{0}

In this section we intend to prove the convergence of the adaptive Algorithm \ref{Alg:3.1}. The proof uses some ideas of \cite{Dai,He} and some results of \cite{Cascon}. Following Theorem \ref{Thm:2.2}, we may firstly establish some relationships between the two level approximations, which will be used in our analysis for both convergence and optimal complexity.
\begin{Theorem}\label{Thm:4.1}
Let $h,H\in (0,h_0)$ and $(u,y,p)\in U_{ad}\times H_0^1(\Omega)\times H_0^1(\Omega)$ be the
 solution of problem (\ref{OCP})-(\ref{OCP_state}). Assume that $(u_h, y_h, p_h)\in U_{ad}\times V_h\times V_h$  and $(u_H, y_H, p_H)\in U_{ad}\times V_H\times V_H$
 are the solutions of problem (\ref{OCP_h})-(\ref{OCP_state_h}), respectively. Define $y^H:=Su_H$ and $p^H:=S^*(S_Hu_H-y_d)$. Then the following properties hold
\begin{eqnarray}
\|y-y_h\|_{a,\Omega}&=& \|y^H-\mathcal{R}_h y^H\|_{a,\Omega}
+O(\tilde\kappa(h_0))\big(\|y-y_h\|_{a,\Omega}+\|y-y_H\|_{a,\Omega}\nonumber\\
&&+\|p-p_h\|_{a,\Omega}+\|p-p_H\|_{a,\Omega}\big),\label{two_level_equ_y}\\
\|p-p_h\|_{a,\Omega}&=& \|p^H-\mathcal{R}_h p^H\|_{a,\Omega}+O(\tilde\kappa(h_0))\big(\|y-y_h\|_{a,\Omega}+\|y-y_H\|_{a,\Omega}\nonumber\\
&&+\|p-p_h\|_{a,\Omega}+\|p-p_H\|_{a,\Omega}\big)\label{two_level_equ_p},\\
\mbox{osc}(u_h-Ly_h,\mathcal{T}_h)&=& \mbox{osc}(u_H-L\mathcal{R}_hy^H,\mathcal{T}_h)+O(\tilde\kappa(h_0))\big(\|y-y_h\|_{a,\Omega}+\|p-p_h\|_{a,\Omega}\nonumber\\
&&+\|y-y_H\|_{a,\Omega}+\|p-p_H\|_{a,\Omega}\big)\label{two_level_data_y},\\
\mbox{osc}(y_h-y_d-L^*p_h,\mathcal{T}_h)&=& \mbox{osc}(y_H-y_d-L^*\mathcal{R}_hp^H,\mathcal{T}_h)+O(\tilde\kappa(h_0))\big(\|y-y_h\|_{a,\Omega}\nonumber\\
&&+\|p-p_h\|_{a,\Omega}+\|y-y_H\|_{a,\Omega}+\|p-p_H\|_{a,\Omega}\big)\label{two_level_data_p}
\end{eqnarray}
and
\begin{eqnarray}
\eta_{y,h}(y_h,\Omega)&=& \tilde\eta_{h}(\mathcal{R}_hy^H,\Omega)+O(\tilde \kappa(h_0))\big(\|y-y_h\|_{a,\Omega}+\|y-y_H\|_{a,\Omega}\nonumber\\
&&+\|p-p_h\|_{a,\Omega}+\|p-p_H\|_{a,\Omega}\big),\label{two_level_estimator_y}\\
\eta_{p,h}(p_h,\Omega)&=& \tilde\eta_{h}(\mathcal{R}_hp^H,\Omega)+O(\tilde \kappa(h_0))\big(\|y-y_h\|_{a,\Omega}+\|y-y_H\|_{a,\Omega}\nonumber\\
&&+\|p-p_h\|_{a,\Omega}+\|p-p_H\|_{a,\Omega}\big)\label{two_level_estimator_p}
\end{eqnarray}
provided $h_0\ll 1$.
\end{Theorem}
\begin{proof}
Note that 
\begin{eqnarray}
y-y_h=y^H-\mathcal{R}_hy^H+\mathcal{R}_h(y^H-y^h)+y-y^H\label{x_1}
\end{eqnarray}
and
\begin{eqnarray}
p-p_h=p^H-\mathcal{R}_hp^H+\mathcal{R}_h(p^H-p^h)+p-p^H.\label{x_2}
\end{eqnarray}
On the other hand, it follows from (\ref{Ritz_stab}) that
\begin{eqnarray}
\|\mathcal{R}_h(y^H-y^h)+y-y^H\|_{a,\Omega}&\lesssim& \|y^H-y^h\|_{a,\Omega}+\|y-y^H\|_{a,\Omega}\nonumber\\
&\lesssim&\|y-y^h\|_{a,\Omega}+\|y-y^H\|_{a,\Omega}\nonumber\\
&\lesssim&\|u-u_h\|_{0,\Omega}+\|u-u_H\|_{0,\Omega}\label{x_3}
\end{eqnarray}
and
\begin{eqnarray}
\|\mathcal{R}_h(p^H-p^h)+p-p^H\|_{a,\Omega}&\lesssim& \|p^H-p^h\|_{a,\Omega}+\|p-p^H\|_{1,\Omega}\nonumber\\
&\lesssim&\|y-y_h\|_{0,\Omega}+\|y-y_H\|_{0,\Omega}\nonumber\\
&\lesssim&\|u-u_h\|_{0,\Omega}+\kappa(h)\|y-y_h\|_{a,\Omega}\nonumber\\
&&+\|u-u_H\|_{0,\Omega}+\kappa(H)\|y-y_H\|_{a,\Omega},\label{x_4}
\end{eqnarray}
where in the last inequality we used (\ref{stab1}). It follows from (\ref{lift}) that
\begin{eqnarray}
&&\|\mathcal{R}_h(y^H-y^h)+y-y^H\|_{a,\Omega}+\|\mathcal{R}_h(p^H-p^h)+p-p^H\|_{a,\Omega}\nonumber\\
&\lesssim& \kappa(h)\Big(\|y-y_h\|_{a,\Omega}+\|p-p_h\|_{a,\Omega}\Big)+\kappa(H)\Big(\|y-y_H\|_{a,\Omega}+\|p-p_H\|_{a,\Omega}\Big)\nonumber\\
&\lesssim& \tilde\kappa(h_0)\Big(\|y-y_h\|_{a,\Omega}+\|p-p_h\|_{a,\Omega}+\|y-y_H\|_{a,\Omega}+\|p-p_H\|_{a,\Omega}\Big)\label{x_5}
\end{eqnarray}
provided $h_0\ll 1$. This combining with (\ref{x_1})-(\ref{x_2}) yields (\ref{two_level_equ_y}) and (\ref{two_level_equ_p}).

Then we prove (\ref{two_level_data_y})-(\ref{two_level_data_p}). Note that 
\begin{eqnarray}
u_h-Ly_h=u_H-L\mathcal{R}_hy^H+L\mathcal{R}_h(y^H-y^h)+(u_h-u_H),\label{y_1}\\
y_h-y_d-L^*p_h=y_H-y_d-L^*\mathcal{R}_hp^H+L^*\mathcal{R}_h(p^H-p^h)+(y_h-y_H)\label{y_2}.
\end{eqnarray}
From Lemma \ref{Lm:2.4} we have 
\begin{eqnarray}
\mbox{osc}(L\mathcal{R}_h(y^H-y^h),\mathcal{T}_h)\lesssim\|\mathcal{R}_h(y^H-y^h)\|_{a,\Omega},\nonumber\\
\mbox{osc}(L^*\mathcal{R}_h(p^H-p^h),\mathcal{T}_h)\lesssim\|\mathcal{R}_h(p^H-p^h)\|_{a,\Omega},\nonumber
\end{eqnarray}
which together with (\ref{x_5}) imply
\begin{eqnarray}
&&\mbox{osc}(L\mathcal{R}_h(y^H-y^h),\mathcal{T}_h)+
\mbox{osc}(L^*\mathcal{R}_h(p^H-p^h),\mathcal{T}_h)\nonumber\\
 &\lesssim&\tilde\kappa(h_0)\Big(\|y-y_h\|_{a,\Omega}+\|p-p_h\|_{a,\Omega}+\|y-y_H\|_{a,\Omega}+\|p-p_H\|_{a,\Omega}\Big).\label{y_3}
\end{eqnarray}
Moreover, since $\bar f_T$ is the $L^2$-projection of $f$ onto piecewise polynomials on $T$, there holds 
\begin{eqnarray}
\mbox{osc}(f,\mathcal{T}_h)=\Big(\sum\limits_{T\in\mathcal{T}_h}\|h_T(f-\bar f_T)\|_{0,T}^2\Big)^{1\over 2}\lesssim\|f\|_{0,\Omega}.\nonumber
\end{eqnarray}
In view of (\ref{stab1}) we thus have
\begin{eqnarray}
\mbox{osc}(u_h-u_H,\mathcal{T}_h)\lesssim\|u_h-u_H\|_{0,\Omega}&\lesssim&\|u-u_h\|_{0,\Omega}+\|u-u_H\|_{0,\Omega},\nonumber\\
\mbox{osc}(y_h-y_H,\mathcal{T}_h)\lesssim\|y_h-y_H\|_{0,\Omega}&\lesssim&\|u-u_h\|_{0,\Omega}+\|u-u_H\|_{0,\Omega}\nonumber\\
&&+\kappa(H)\|y-y_H\|_{a,\Omega}+\kappa(h)\|y-y_h\|_{a,\Omega},\nonumber
\end{eqnarray}
which together with (\ref{lift}) yield
\begin{eqnarray}
\mbox{osc}(u_h-u_H,\mathcal{T}_h)&\lesssim&\tilde\kappa(h_0)\Big(\|y-y_h\|_{a,\Omega}+\|p-p_h\|_{a,\Omega}\nonumber\\
&&+\|y-y_H\|_{a,\Omega}+\|p-p_H\|_{a,\Omega}\Big),\label{y_4}\\
\mbox{osc}(y_h-y_H,\mathcal{T}_h)&\lesssim&\tilde\kappa(h_0)\Big(\|y-y_h\|_{a,\Omega}+\|p-p_h\|_{a,\Omega}\nonumber\\
&&+\|y-y_H\|_{a,\Omega}+\|p-p_H\|_{a,\Omega}\Big).\label{y_5}
\end{eqnarray}
We can conclude the desired results (\ref{two_level_data_y})-(\ref{two_level_data_p}) from the definition of the data oscillation and (\ref{y_1})-(\ref{y_5}).

Now it remains to prove (\ref{two_level_estimator_y}) and (\ref{two_level_estimator_p}). From the definition of $y^H$ and $y^h$ we know that $y^h-y^H$ is the solution of elliptic boundary value problem with right hand side $u_h-u_H$. It follows from (\ref{elliptic_lower}) and (\ref{x_3}) that
\begin{eqnarray}
\tilde\eta_{h}(\mathcal{R}_h(y^h-y^H),\Omega)&\lesssim& \|(y^h-y^H)-\mathcal{R}_h(y^h-y^H)\|_{a,\Omega}\nonumber\\
&&+\mbox{osc}(u_h-u_H-L\mathcal{R}_h(y^h-y^H),\mathcal{T}_h)\nonumber\\
&\lesssim&\|u-u_h\|_{0,\Omega}+\|u-u_H\|_{0,\Omega}\nonumber\\
&&+\mbox{osc}(u_h-u_H-L\mathcal{R}_h(y^h-y^H),\mathcal{T}_h).\label{y_6}
\end{eqnarray}
From (\ref{osc_linear}), (\ref{lift}), (\ref{y_3}) and (\ref{y_4}) we are led to 
\begin{eqnarray}
 \mbox{osc}(u_h-u_H-L\mathcal{R}_h(y^h-y^H),\mathcal{T}_h)&\lesssim&\tilde\kappa(h_0)\Big(\|y-y_h\|_{a,\Omega}+\|p-p_h\|_{a,\Omega}\nonumber\\
&&+\|y-y_H\|_{a,\Omega}+\|p-p_H\|_{a,\Omega}\Big).\label{y_7}
\end{eqnarray}
Note that
\begin{eqnarray}
\eta_{y,h}(y_h,\Omega)=\tilde\eta_{h}(\mathcal{R}_hy^h,\Omega)=\tilde\eta_{h}(\mathcal{R}_hy^H+\mathcal{R}_h(y^h-y^H),\Omega).\nonumber
\end{eqnarray}
This combining with (\ref{y_6}) and (\ref{y_7}) gives 
\begin{eqnarray}
\eta_{y,h}(y_h,\Omega)&=& \tilde\eta_{h}(\mathcal{R}_hy^H,\Omega)+\tilde \kappa(h_0)\Big(\|y-y_h\|_{a,\Omega}+\|y-y_H\|_{a,\Omega}\nonumber\\
&&+\|p-p_h\|_{a,\Omega}+\|p-p_H\|_{a,\Omega}\Big),\nonumber
 \end{eqnarray}
 which proves (\ref{two_level_estimator_y}). Similarly we can prove (\ref{two_level_estimator_p}). Thus, we complete the proof of the theorem.
\end{proof}
Now we are ready to prove the error reduction for the sum of the energy errors and the scaled error estimators of the state $y$ and the adjoint state $p$, between two consecutive adaptive loops.
\begin{Theorem}\label{Thm:4.2}
Let $(u,y,p)\in U_{ad}\times H_0^1(\Omega)\times H_0^1(\Omega)$ be the
 solution of problem (\ref{OCP})-(\ref{OCP_state}) and $(u_{h_k},y_{h_k}, p_{h_k})\in U_{ad}\times V_{h_k}\times V_{h_k}$  be a sequence of solutions to
 problem (\ref{OCP_h})-(\ref{OCP_state_h}) produced by Algorithm \ref{Alg:3.1}. Then there exist constants $\gamma>0$ and $\beta\in (0,1)$ depending only on the shape regularity of meshes and the parameter $\theta$ used by Algorithm \ref{Alg:3.0}, such that for any two consecutive iterates $k$ and $k+1$, we have
\begin{eqnarray}\label{convergence}
&&\|(y-y_{h_{k+1}},p-p_{h_{k+1}})\|_{a}^2+\gamma\eta_{h_{k+1}}^2((y_{h_{k+1}},p_{h_{k+1}}),\Omega)\nonumber\\
&\leq& \beta^2\Big(\|(y-y_{h_{k}},p-p_{h_{k}})\|_{a}^2+\gamma\eta_{h_{k}}^2((y_{h_{k}},p_{h_{k}}),\Omega)\Big)
\end{eqnarray}
provided $h_0\ll 1$. Therefore, Algorithm \ref{Alg:3.1} converges with a linear rate $\beta$, namely, the $k$-th iterate solution $(u_{h_k},y_{h_k},p_{h_k})$ of Algorithm \ref{Alg:3.1} satisfies
\begin{eqnarray}\label{convergence_rate}
\|(y-y_{h_{k}},p-p_{h_{k}})\|_{a}^2+\gamma\eta_{h_{k}}^2((y_{h_{k}},p_{h_{k}}),\Omega)\leq C_0\beta^{2k},
\end{eqnarray}
where $C_0=\|(y-y_{h_{0}},p-p_{h_{0}})\|_{a}^2+\gamma\eta_{h_{0}}^2((y_{h_{0}},p_{h_{0}}),\Omega)$.
\end{Theorem}
\begin{proof}
For convenience, we use $(u_H,y_H,p_H)$ and $(u_h,y_h,p_h)$ to denote $(u_{h_k},y_{h_k},p_{h_k})$ and $(u_{h_{k+1}}, y_{h_{k+1}},p_{h_{k+1}})$, respectively. So it suffices to prove that for $(u_H,y_H,p_H)$ and $(u_h,y_h,p_h)$, there holds
\begin{eqnarray}\label{convergence_1}
&&\|(y-y_{h},p-p_{h})\|_{a}^2+\gamma\eta_{h}^2((y_{h},p_{h}),\Omega)\nonumber\\
&\leq& \beta^2\Big(\|(y-y_{H},p-p_{H})\|_{a}^2+\gamma\eta_{H}^2((y_{H},p_{H}),\Omega)\Big).
\end{eqnarray}

Recall that $y^H:=Su_H$, $y^h:=Su_h$ and $p^H:=S^*(S_Hu_H-y_d)$, $p^h:=S^*(S_hu_h-y_d)$. It follows from Algorithm \ref{Alg:3.0} that the D\"{o}rfler's marking strategy in Algorithm \ref{Alg:1.0} is satisfied for $(y^H,p^H)$. So we conclude from Theorem \ref{Thm:elliptic_convergence} that there exist constants $\tilde \gamma$ and $\tilde \beta\in (0,1)$ satisfying 
\begin{eqnarray}
&&\|(y^H-\mathcal{R}_hy^H,p^H-\mathcal{R}_hp^H)\|_{a}^2+\tilde\gamma\big(\tilde\eta_h^2(\mathcal{R}_hy^H,\Omega)+\tilde\eta_h^2(\mathcal{R}_hp^H,\Omega)\big)\nonumber\\
&\leq& \tilde\beta^2\Big(\|(y^H-\mathcal{R}_Hy^H,p^H-\mathcal{R}_Hp^H)\|_{a}^2+\tilde\gamma\big(\tilde\eta_H^2(\mathcal{R}_Hy^H,\Omega)+\tilde\eta_H^2(\mathcal{R}_Hp^H,\Omega)\big)\Big).\label{e_1}
\end{eqnarray}
Note that $\mathcal{R}_Hy^H=y_H$ and $\mathcal{R}_Hp^H=p_H$, we thus have
\begin{eqnarray}
&&\|(y^H-\mathcal{R}_hy^H,p^H-\mathcal{R}_hp^H)\|_{a}^2+\tilde\gamma\big(\tilde\eta_h^2(\mathcal{R}_hy^H,\Omega)+\tilde\eta_h^2(\mathcal{R}_hp^H,\Omega)\big)\nonumber\\
&\leq& \tilde\beta^2\Big(\|(y^H-y_H,p^H-p_H)\|_{a}^2+\tilde\gamma\big(\eta_{y,H}^2(y_H,\Omega)+\eta_{p,H}^2(p_H,\Omega)\big)\Big).\label{e_3}
\end{eqnarray}
We conclude from (\ref{two_level_equ_y})-(\ref{two_level_equ_p}) and (\ref{two_level_estimator_y})-(\ref{two_level_estimator_p}) that there exists a constant $\hat C_4>0$ such that
\begin{eqnarray}
&&\|(y-y_{h},p-p_{h})\|_{a}^2+\tilde\gamma\eta_{h}^2((y_{h},p_{h}),\Omega)\nonumber\\
&\leq&(1+\delta_1)\|(y^H-\mathcal{R}_hy^{H},p^H-\mathcal{R}_hp^{H})\|_{a}^2+(1+\delta_1)\tilde\gamma\big(\tilde\eta_{h}^2(\mathcal{R}_hy^{H},\Omega)+\tilde\eta_{h}^2(\mathcal{R}_hp^{H},\Omega)\big)\nonumber\\
&&+\hat C_4(1+\delta_1^{-1})\tilde \kappa^2(h_0)\Big(\|(y-y_h,p-p_h)\|^2_{a}+\|(y-y_H,p-p_H)\|^2_{a}\Big)\nonumber\\
&&+\hat C_4(1+\delta_1^{-1})\tilde \kappa^2(h_0)\tilde\gamma\Big(\|(y-y_h,p-p_h)\|^2_{a}+\|(y-y_H,p-p_H)\|^2_{a}\Big),\nonumber
\end{eqnarray}
where the $\delta_1$-Young inequality is used and $\delta_1\in(0,1)$ satisfies
\begin{eqnarray}
(1+\delta_1)\tilde\beta^2<1.\label{delta_1}
\end{eqnarray}
Thus, there exists a positive constant $\hat C_5$ depending on $\hat C_4$ and $\tilde\gamma$ such that
\begin{eqnarray}
&&\|(y-y_{h},p-p_{h})\|_{a}^2+\tilde\gamma\eta_{h}^2((y_{h},p_{h}),\Omega)\nonumber\\
&\leq&(1+\delta_1)\Big(\|(y^H-\mathcal{R}_hy^{H},p^H-\mathcal{R}_hp^{H})\|_{a}^2+\tilde\gamma\big(\tilde\eta_{h}^2(\mathcal{R}_hy^{H},\Omega)+\tilde\eta_{h}^2(\mathcal{R}_hp^{H},\Omega)\big)\Big)\nonumber\\
&&+\hat C_5\delta_1^{-1}\tilde \kappa^2\Big(h_0)(\|(y-y_h,p-p_h)\|^2_{a,\Omega}+\|(y-y_H,p-p_H)\|^2_{a,\Omega}\Big).\label{e_5}
\end{eqnarray}
It follows from (\ref{e_3}) and (\ref{e_5}) that
\begin{eqnarray}
&&\|(y-y_{h},p-p_{h})\|_{a}^2+\tilde\gamma\eta_{h}^2((y_{h},p_{h}),\Omega)\nonumber\\
&\leq&(1+\delta_1)\tilde\beta^2\Big(\|(y^H-y_{H},p^H-p_{H})\|_{a}^2+\tilde\gamma\eta_{H}^2((y_{H},p_{H}),\Omega)\Big)\nonumber\\
&&+\hat C_5\delta_1^{-1}\tilde \kappa^2(h_0)\Big(\|(y-y_h,p-p_h)\|^2_{a}+\|(y-y_H,p-p_H)\|^2_{a}\Big).\label{e_6}
\end{eqnarray}
Then using Theorem \ref{Thm:2.2} we arrive at
\begin{eqnarray}
&&\|(y-y_{h},p-p_{h})\|_{a}^2+\tilde\gamma\eta_{h}^2((y_{h},p_{h}),\Omega)\nonumber\\
&\leq&(1+\delta_1)\tilde\beta^2\Big((1+\hat C_6\tilde\kappa(h_0))\|(y-y_{H},p-p_{H})\|_{a}^2+\tilde\gamma\eta_{H}^2((y_{H},p_{H}),\Omega)\Big)\nonumber\\
&&+\hat C_5\delta_1^{-1}\tilde \kappa^2(h_0)\Big(\|(y-y_h,p-p_h)\|^2_{a}+\|(y-y_H,p-p_H)\|^2_{a}\Big),\nonumber
\end{eqnarray}
and thus
\begin{eqnarray}
&&\|(y-y_{h},p-p_{h})\|_{a}^2+\tilde\gamma\eta_{h}^2((y_{h},p_{h}),\Omega)\nonumber\\
&\leq&(1+\delta_1)\tilde\beta^2\Big(\|(y-y_{H},p-p_{H})\|_{a}^2+\tilde\gamma\eta_{H}^2((y_{H},p_{H}),\Omega)\Big)\nonumber\\
&&+C_4\tilde \kappa(h_0)\|(y-y_H,p-p_H)\|^2_{a}+C_4\delta_1^{-1}\tilde \kappa^2(h_0)\|(y-y_h,p-p_h)\|^2_{a},\label{e_7}
\end{eqnarray}
where $C_4$ is a positive constant depending on $\hat C_5$ and $\hat C_6$ when $h_0\ll 1$. So we can derive
\begin{eqnarray}
&&(1-C_4\delta_1^{-1}\tilde \kappa^2(h_0))\|(y-y_{h},p-p_{h})\|_{a}^2+\tilde\gamma\eta_{h}^2((y_{h},p_{h}),\Omega)\nonumber\\
&\leq&\Big((1+\delta_1)\tilde\beta^2+C_4\tilde \kappa(h_0)\Big)\|(y-y_{H},p-p_{H})\|_{a}^2+(1+\delta_1)\tilde\beta^2\tilde\gamma\eta_{H}^2((y_{H},p_{H}),\Omega),\label{e_8}
\end{eqnarray}
or equivalently,
\begin{eqnarray}
&&\|(y-y_{h},p-p_{h})\|_{a}^2+\frac{\tilde\gamma}{1-C_4\delta_1^{-1}\tilde \kappa^2(h_0)}\eta_{h}^2((y_{h},p_{h}),\Omega)\nonumber\\
&\leq&\frac{(1+\delta_1)\tilde\beta^2+C_4\tilde \kappa(h_0)}{1-C_4\delta_1^{-1}\tilde \kappa^2(h_0)}\|(y-y_{H},p-p_{H})\|_{a}^2+\frac{(1+\delta_1)\tilde\beta^2\tilde\gamma}{1-C_4\delta_1^{-1}\tilde \kappa^2(h_0)}\eta_{H}^2((y_{H},p_{H}),\Omega).\label{e_9}
\end{eqnarray}
Since $\tilde \kappa(h_0)\ll 1$ provided that $h_0\ll 1$, we can define the constant $\beta$ as
\begin{eqnarray}
\beta:=\Big(\frac{(1+\delta_1)\tilde\beta^2+C_4\tilde \kappa(h_0)}{1-C_4\delta_1^{-1}\tilde \kappa^2(h_0)}\Big)^{1\over 2},\label{beta}
\end{eqnarray}
which satisfies $\beta\in (0,1)$ if $h_0\ll 1$. Then
\begin{eqnarray}
&&\|(y-y_{h},p-p_{h})\|_{a}^2+\frac{\tilde\gamma}{1-C_4\delta_1^{-1}\tilde \kappa^2(h_0)}\eta_{h}^2((y_{h},p_{h}),\Omega)\nonumber\\
&\leq&\beta^2\Big(\|(y-y_{H},p-p_{H})\|_{a}^2+\frac{(1+\delta_1)\tilde\beta^2\tilde\gamma}{(1+\delta_1)\tilde\beta^2+C_4\tilde \kappa(h_0)}\eta_{H}^2((y_{H},p_{H}),\Omega)\Big).\label{e_10}
\end{eqnarray}
Now we choose
\begin{eqnarray}
\gamma:=\frac{\tilde \gamma}{1-C_4\delta_1^{-1}\tilde \kappa^2(h_0)}, \label{gamma}
\end{eqnarray}
it is obvious that
\begin{eqnarray}
\frac{(1+\delta_1)\tilde\beta^2\tilde\gamma}{(1+\delta_1)\tilde\beta^2+C_4\tilde \kappa(h_0)}=\frac{(1+\delta_1)\tilde\beta^2(1-C_4\delta_1^{-1}\tilde \kappa^2(h_0))\gamma}{(1+\delta_1)\tilde\beta^2+C_4\tilde \kappa(h_0)}< (1-C_4\delta_1^{-1}\tilde \kappa^2(h_0))\gamma<\gamma.\nonumber
\end{eqnarray}
Then we obtain (\ref{convergence_1}), this completes the proof.
\end{proof}
\begin{Remark}\label{Re:initial_mesh}
We remark that the requirement $h_0\ll 1$ on the initial mesh $\mathcal{T}_{h_0}$ is not restrictive for the convergence analysis of AFEM for nonlinear problems, such as optimal control problems studied in this paper, see, e.g., \cite{Gaevskaya}. For similar requirement we refer to \cite{Zhou,Dai} for the convergence analysis of adaptive finite element eigenvalue computations and \cite{Mekchay} for the adaptive finite element computations for nonsymmetric boundary value problems, we should also mention \cite{He} for the adaptive finite element  method of a semilinear elliptic equation.  
\end{Remark}
\begin{Remark}\label{Re:OCP_marking}
In adaptive Algorithm \ref{Alg:3.1} we use the sum of the error estimators $\eta_{y,h}(y_h,T)$ contributed to the state approximation and $\eta_{p,h}(p_h,T)$ contributed to the adjoint state approximation as indicator to select the subset $\tilde{\mathcal{T}}_h$ for refinement. This marking strategy enables us to prove the convergence and quasi-optimality (see Section 5) of AFEM for optimal control problems. We remark that it is also possible to use the separate marking for the contributions of  $\eta_{y,h}(y_h,T)$ and $\eta_{p,h}(p_h,T)$ as follows:
\begin{itemize}
\item Construct a minimal subset $\tilde{\mathcal{T}}_{h,1}\subset \mathcal{T}_h$ such that
$\sum\limits_{T\in\tilde{\mathcal{T}}_{h,1}} \eta_{y,h}^2(y_h,T)\geq \theta\eta_{y,h}^2(y_h,\Omega)$.

\item Construct another minimal subset $\tilde{\mathcal{T}}_{h,2}\subset \mathcal{T}_h$ such that
$\sum\limits_{T\in\tilde{\mathcal{T}}_{h,2}} \eta_{p,h}^2(p_h,T)\geq \theta\eta_{p,h}^2(p_h,\Omega)$.

\item Set $\tilde{\mathcal{T}}_{h}:=\tilde{\mathcal{T}}_{h,1}\cup \tilde{\mathcal{T}}_{h,2}$ and mark all the elements in $\tilde{\mathcal{T}}_{h}$. 
\end{itemize}
With this marking strategy we can also prove the convergence of AFEM for optimal control problems by using the results of \cite{Cascon,Dai} for single boundary value problem. To be more specific, the error reduction (\ref{e_1}) can be derived separately for the state and adjoint state approximations. However, the resulting over-refinement for this marking strategy prevents us to prove the quasi-optimality of the adaptive algorithm. 
\end{Remark}

\section{Complexity of AFEM for optimal control problem}
\setcounter{equation}{0}
In this section we intend to analyse the complexity of adaptive finite element algorithm for optimal control problems based on the known results of the complexity for elliptic boundary value problems. The proof uses some ideas of \cite{Dai,He} and some results of \cite{Cascon}. 

Similar to \cite{Cascon} and \cite{Dai}, for our purpose to analyse the complexity of AFEM for optimal control problems we need to introduce a function approximation class as follows
\begin{eqnarray}
\mathcal{A}_\gamma^s:=\Big\{(y,p,y_d)\in H_0^1(\Omega)\times H_0^1(\Omega)\times L^2(\Omega):\quad |(y,p,y_d)|_{s,\gamma}<+\infty\Big\},\nonumber
\end{eqnarray}
where $\gamma>0$ is some constant and
\begin{eqnarray}
|(y,p,y_d)|_{s,\gamma}=\sup\limits_{\varepsilon>0}\varepsilon\inf\limits_{\mathcal{T}\subset\mathcal{T}_{h_0}:\ \inf(\|(y-y_{\mathcal{T}},p-p_{\mathcal{T}})\|_{a}^2+(\gamma+1)\mbox{osc}^2((y_{\mathcal{T}},p_{\mathcal{T}}),{\mathcal{T}}))^{1/ 2}\leq\varepsilon}(\#{\mathcal{T}}-\#{\mathcal{T}_{h_0}})^s.\nonumber
\end{eqnarray}
Here $\mathcal{T}\subset\mathcal{T}_{h_0}$ means $\mathcal{T}$ is a refinement of $\mathcal{T}_{h_0}$, $y_{\mathcal{T}}$ and $p_{\mathcal{T}}$ are elements of the finite element space corresponding to the partition $\mathcal{T}$. It is seen from the definition that $\mathcal{A}_\gamma^s=\mathcal{A}_1^s$ for all $\gamma>0$, thus we use $\mathcal{A}^s$ throughout the paper with corresponding norm $|\cdot|_{s}$. So $\mathcal{A}^s$ is the class of functions that can be approximated with a given tolerance $\varepsilon$ by continuous peicewise linear polynomial functions over a partition $\mathcal{T}$ with number of degrees of freedom $\#{\mathcal{T}}-\#{\mathcal{T}_{h_0}}\lesssim \varepsilon^{-1/s}|v|_s^{1/s}$.

Now we are in the position to prepare for the proof of optimal complexity of Algorithm \ref{Alg:3.1} for the optimal control problem (\ref{OCP})-(\ref{OCP_state}). At first, we define $y^{h_k}:=Su_{h_k}$ and $p^{h_k}:=S^*(S_{h_k}u_{h_k}-y_d)$. Then we have the following result.
\begin{Lemma}\label{Lm:OCP_decay}
Let $(u_{h_k},y_{h_k},p_{h_k})\in U_{ad}\times V_{h_k}\times V_{h_k}$ and $(u_{h_{k+1}},y_{h_{k+1}},p_{h_{k+1}})\in U_{ad}\times V_{h_{k+1}}\times V_{h_{k+1}}$ be discrete solutions of   problem (\ref{OCP_h})-(\ref{OCP_state_h}) over  mesh $\mathcal{T}_{h_k}$ and its refinement $\mathcal{T}_{h_{k+1}}$ with marked element $\mathcal{M}_{h_{k}}$. Suppose they satisfy the following property
\begin{eqnarray}
&&\|(y-y_{h_{k+1}},p-p_{h_{k+1}})\|_{a}^2+ \gamma_*\mbox{osc}^2((y_{h_{k+1}},p_{h_{k+1}}),\mathcal{T}_{h_{k+1}})\nonumber\\
&\leq&\beta_*^2\Big(\|(y-y_{h_k},p-p_{h_k})\|_{a}^2+\gamma_*\mbox{osc}^2((y_{h_k},p_{h_k}),\mathcal{T}_{h_k})\Big)\label{OCP_decay}
\end{eqnarray}
with $\gamma_*$  and $\beta_*$ some positive constants. Then for the associated state and adjoint state approximations we have
\begin{eqnarray}
&&\|(y^{h_k}-\mathcal{R}_{h_{k+1}}y^{h_k},p^{h_k}-\mathcal{R}_{h_{k+1}}p^{h_k})\|_{a}^2+ \tilde\gamma_*\mbox{osc}^2((\mathcal{R}_{h_{k+1}}y^{h_k},\mathcal{R}_{h_{k+1}}p^{h_k}),\mathcal{T}_{h_{k+1}})\nonumber\\
&\leq&\tilde \beta_*^2\Big(\|(y^{h_k}-\mathcal{R}_{h_k}y^{h_k},p^{h_k}-\mathcal{R}_{h_k}p^{h_k})\|_{a}^2+\tilde\gamma_*\mbox{osc}^2((y_{h_k},p_{h_k}),\mathcal{T}_{h_k})\Big)\label{OCP_BVP_decay}
\end{eqnarray}
with 
\begin{eqnarray}
\tilde\beta_*:=\Big(\frac{(1+\delta_1)\beta^2_*+C_5\tilde \kappa(h_0)}{1-C_5\delta_1^{-1}\tilde \kappa^2(h_0)}\Big)^{1\over 2},\quad \tilde\gamma_*:=\frac{\gamma_*}{1-C_5\delta_1^{-1}\tilde \kappa^2(h_0)},\nonumber
\end{eqnarray}
where $C_5$ is some constant depending on $C_*$, $\hat C_5$ and $\hat C_6$. $\hat C_5$, $\hat C_6$ and $\delta_1\in (0,1)$ are some constants as in the proof of Thoerem \ref{Thm:4.2}.
\end{Lemma}
\begin{proof}
The proof follows the similar procedure as in the proof of Theorem \ref{Thm:4.2} when (\ref{two_level_estimator_y})-(\ref{two_level_estimator_p}) are replaced by (\ref{two_level_data_y})-(\ref{two_level_data_p}). Specifically, in the proof of Theorem \ref{Thm:4.2} we use (\ref{e_1}), Theorem \ref{Thm:2.2} and Theorem \ref{Thm:4.1} to prove (\ref{convergence_1}). Conversely, here we need to prove (\ref{e_1}) from (\ref{convergence_1}), Theorem \ref{Thm:2.2} and Theorem \ref{Thm:4.1}.
\end{proof}
Next, we are able to derive a result similar to Lemma \ref{Lm:decay} concerning the optimality of D\"{o}rfler's marking strategy for the optimal control problems.
\begin{Corollary}\label{Cor:OCP_decay}
Let $(u_{h_k},y_{h_k},p_{h_k})\in U_{ad}\times V_{h_k}\times V_{h_k}$ and $(u_{h_{k+1}},y_{h_{k+1}},p_{h_{k+1}})\in U_{ad}\times V_{h_{k+1}}\times V_{h_{k+1}}$ be discrete solutions of   problem (\ref{OCP_h})-(\ref{OCP_state_h}) over  mesh $\mathcal{T}_{h_k}$ and its refinement $\mathcal{T}_{h_{k+1}}$ with marked element $\mathcal{M}_{h_k}$. Suppose they satisfy the following property
\begin{eqnarray}
&&\|(y-y_{h_{k+1}},p-p_{h_{k+1}})\|_{a}^2+ \gamma_*\mbox{osc}^2((y_{h_{k+1}},p_{h_{k+1}}),\mathcal{T}_{h_{k+1}})\nonumber\\
&\leq&\beta_*^2(\|(y-y_{h_k},p-p_{h_k})\|_{a}^2+\gamma_*\mbox{osc}^2((y_{h_k},p_{h_k}),\mathcal{T}_{h_k}))\nonumber
\end{eqnarray}
with constants $\gamma_*>0$  and $\beta_*\in (0,\sqrt{1\over 2})$. Then the set $R_{\mathcal{T}_{h_k}\rightarrow\mathcal{T}_{h_{k+1}}}$ of refined elements satisfies the D\"{o}rfler property
\begin{eqnarray}
\sum\limits_{T\in R_{\mathcal{T}_{h_k}\rightarrow\mathcal{T}_{h_{k+1}}}}\eta_{h_k}^2((y_{h_k},p_{h_k}),T)\geq\hat\theta \sum\limits_{T\in\mathcal{T}_{h_k}}\eta_{h_k}^2((y_{h_k},p_{h_k}),T)\label{BVP_Dorfler}
\end{eqnarray}
with $\hat\theta=\frac{\tilde C_2(1-2\tilde\beta_*^2)}{\tilde C_0(\tilde C_1+(1+2C_*^2\tilde C_1)\tilde\gamma_*)}$ and $\tilde C_0=\max(1,\frac{\tilde C_3}{\tilde\gamma_*})$.
\end{Corollary}
\begin{proof}
From Lemma \ref{Lm:OCP_decay} we can conclude  (\ref{OCP_BVP_decay}) from (\ref{OCP_decay}). Note that $y_{h_k}=\mathcal{R}_{h_k}y^{h_k}$  and $p_{h_k}=\mathcal{R}_{h_k}p^{h_k}$. By the lower bounds in Lemma \ref{Lm: lower_bound} we have 
\begin{eqnarray}
(1-2\tilde\beta_*^2)\tilde C_2\eta_{h_k}^2((y_{h_k},p_{h_k}),\Omega)&\leq& (1-2\tilde\beta_*^2)\Big(\|(y^{h_k}-y_{h_k},p^{h_k}-p_{h_k})\|_{a}^2 +\tilde C_3\mbox{osc}^2((y_{h_k},p_{h_k}),\mathcal{T}_{h_k})\Big)\nonumber\\
&=&(1-2\tilde\beta_*^2)\Big(\|(y^{h_k}-y_{h_k},p^{h_k}-p_{h_k})\|_{a}^2 +\frac{\tilde C_3}{\tilde \gamma_*}\tilde\gamma_*\mbox{osc}^2((y_{h_k},p_{h_k}),\mathcal{T}_{h_k})\Big)\nonumber\\
&\leq&\tilde C_0(1-2\tilde\beta_*^2)\Big(\|(y^{h_k}-y_{h_k},p^{h_k}-p_{h_k})\|_{a}^2 +\tilde\gamma_*\mbox{osc}^2((y_{h_k},p_{h_k}),\mathcal{T}_{h_k})\Big).\nonumber
\end{eqnarray}
Thus, it follows from (\ref{OCP_BVP_decay}) that
\begin{eqnarray}
&&\frac{\tilde C_2}{\tilde C_0}(1-2\tilde\beta_*^2)\sum\limits_{T\in\mathcal{T}_{h_k}}\eta_{h_k}^2((y_{h_k},p_{h_k}),T)\nonumber\\
&\leq&(1-2\tilde\beta_*^2)\Big(\|(y^{h_k}-y_{h_k},p^{h_k}-p_{h_k})\|_{a}^2 +\tilde\gamma_*\mbox{osc}^2((y_{h_k},p_{h_k}),\mathcal{T}_{h_k})\Big)\nonumber\\
&=&\|(y^{h_k}-y_{h_k},p^{h_k}-p_{h_k})\|_{a}^2 +\tilde\gamma_*\mbox{osc}^2((y_{h_k},p_{h_k}),\mathcal{T}_{h_k})\nonumber\\
&&-2\tilde\beta_*^2\Big(\|(y^{h_k}-y_{h_k},p^{h_k}-p_{h_k})\|_{a}^2 +\tilde\gamma_*\mbox{osc}^2((y_{h_k},p_{h_k}),\mathcal{T}_{h_k})\Big)\nonumber\\
&\leq&\|(y^{h_k}-y_{h_k},p^{h_k}-p_{h_k})\|_{a}^2 +\tilde\gamma_*\mbox{osc}^2((y_{h_k},p_{h_k}),\mathcal{T}_{h_k})\nonumber\\
&&-2\Big(\|(y^{h_k}-\mathcal{R}_{h_{k+1}}y^{h_k},p^{h_k}-\mathcal{R}_{h_{k+1}}p^{h_k})\|_{a}^2 +\tilde\gamma_*\mbox{osc}^2((\mathcal{R}_{h_{k+1}}y^{h_k},\mathcal{R}_{h_{k+1}}p^{h_k}),\mathcal{T}_{h_{k+1}})\Big)\nonumber\\
&\leq&\|(y^{h_k}-y_{h_k},p^{h_k}-p_{h_k})\|_{a}^2-\|(y^{h_k}-\mathcal{R}_{h_{k+1}}y^{h_k},p^{h_k}-\mathcal{R}_{h_{k+1}}p^{h_k})\|_{a}^2\nonumber\\
&&+\tilde\gamma_*\Big(\mbox{osc}^2((y_{h_k},p_{h_k}),\mathcal{T}_{h_k})-2\mbox{osc}^2((\mathcal{R}_{h_{k+1}}y^{h_k},\mathcal{R}_{h_{k+1}}p^{h_k}),\mathcal{T}_{h_{k+1}})\Big).\label{T_1}
\end{eqnarray}
Note that $y_{h_k}$ and $\mathcal{R}_{h_{k+1}}y^{h_k}$ are the Galerkin projections of $y^{h_k}$ on $V_{h_k}$ and $V_{h_{k+1}}$, respectively. From the standard Galerkin orthogonality we have
\begin{eqnarray}
&&\|(y^{h_k}-y_{h_k},p^{h_k}-p_{h_k})\|_{a}^2-\|(y^{h_k}-\mathcal{R}_{h_{k+1}}y^{h_k},p^{h_k}-\mathcal{R}_{h_{k+1}}p^{h_k})\|_{a}^2\nonumber\\
&=&\|(y_{h_k}-\mathcal{R}_{h_{k+1}}y^{h_k},p_{h_k}-\mathcal{R}_{h_{k+1}}p^{h_k})\|_{a}^2.\label{T_2}
\end{eqnarray}
By (\ref{osc_bound}), the triangle and the Young inequalities we have
\begin{eqnarray}
\mbox{osc}^2((y_{h_k},p_{h_k}),T)
&\leq& 2 \mbox{osc}^2((\mathcal{R}_{h_{k+1}}y^{h_k},\mathcal{R}_{h_{k+1}}p^{h_k}),T)\nonumber\\
&&+2C_*^2\|(y_{h_k}-\mathcal{R}_{h_{k+1}}y^{h_k},p_{h_k}-\mathcal{R}_{h_{k+1}}p^{h_k})\|_{a}^2,\nonumber
\end{eqnarray}
which together with the dominance of the indicator over oscillation (see \cite[Remark 2.1]{Cascon})
\begin{eqnarray}
\mbox{osc}^2(u_{h_k}-Ly_{h_k},T)&\leq&\eta_{y,h_k}^2(y_{h_k},T),\label{domi_1}\\
\mbox{osc}^2(y_{h_k}-y_d-L^*p_{h_k},T)&\leq&\eta_{p,h_k}^2(p_{h_k},T)\label{domi_2}
\end{eqnarray}
implies
\begin{eqnarray}
&&\mbox{osc}^2((y_{h_k},p_{h_k}),\mathcal{T}_{h_k})-2\mbox{osc}^2((\mathcal{R}_{h_{k+1}}y^{h_k},\mathcal{R}_{h_{k+1}}p^{h_k}),\mathcal{T}_{h_{k+1}})\nonumber\\
&\leq&\sum\limits_{T\in R_{\mathcal{T}_{h_k}\rightarrow\mathcal{T}_{h_{k+1}}}}\mbox{osc}^2((y_{h_k},p_{h_k}),T)+\mbox{osc}^2((y_{h_k},p_{h_k}),\mathcal{T}_{h_k}\cap \mathcal{T}_{h_{k+1}})\nonumber\\
&&-2\mbox{osc}^2((\mathcal{R}_{h_{k+1}}y^{h_k},\mathcal{R}_{h_{k+1}}p^{h_k}),\mathcal{T}_{h_k}\cap\mathcal{T}_{h_{k+1}})\nonumber\\
&\leq&\sum\limits_{T\in R_{\mathcal{T}_{h_k}\rightarrow\mathcal{T}_{h_{k+1}}}}\eta_{h_k}^2((y_{h_k},p_{h_k}),T)+2C_*^2\|(y_{h_k}-\mathcal{R}_{h_{k+1}}y^{h_k},p_{h_k}-\mathcal{R}_{h_{k+1}}p^{h_k})\|_{a}^2\nonumber\\
&\leq&(1+2C_*^2\tilde C_1)\sum\limits_{T\in R_{\mathcal{T}_{h_k}\rightarrow\mathcal{T}_{h_{k+1}}}}\eta_{h_k}^2((y_{h_k},p_{h_k}),T),\label{T_3}
\end{eqnarray}
where we used (\ref{distance}) in the last inequality. Combining (\ref{T_1})-(\ref{T_3}) and (\ref{distance})  we obtain
\begin{eqnarray}
&&\frac{\tilde C_2}{\tilde C_0}(1-2\tilde\beta_*^2)\sum\limits_{T\in\mathcal{T}_{h_k}}\eta_{h_k}^2((y_{h_k},p_{h_k}),T)\nonumber\\
&\leq&(\tilde C_1+(1+2C_*^2\tilde C_1)\tilde\gamma_*)\sum\limits_{T\in R_{\mathcal{T}_{h_k}\rightarrow\mathcal{T}_{h_{k+1}}}}\eta_{h_k}^2((y_{h_k},p_{h_k}),T).\label{T_4}
\end{eqnarray}
By choosing 
\begin{eqnarray}
\hat\theta:=\frac{\frac{\tilde C_2}{\tilde C_0}(1-2\tilde\beta_*^2)}{\tilde C_1+(1+2C_*^2\tilde C_1)\tilde\gamma_*}=\frac{\tilde C_2(1-2\tilde\beta_*^2)}{\tilde C_0(\tilde C_1+(1+2C_*^2\tilde C_1)\tilde\gamma_*)}\nonumber
\end{eqnarray}
we complete the proof.
\end{proof}
\begin{Lemma}
Let $(y,p,y_d)\in \mathcal{A}^s$ and $\mathcal{T}_{h_k}$ ($k\geq 0$) be a sequence of meshes generated by Algorithm \ref{Alg:3.1} starting from the initial mesh $\mathcal{T}_{h_0}$.  Let $\mathcal{T}_{h_{k+1}}=\mbox{REFINE}(\mathcal{T}_{h_k},\mathcal{M}_{h_k})$ where $\mathcal{M}_{h_k}$ is produced by Algorithm \ref{Alg:3.0} with $\theta$ satisfying $\theta\in (0,\frac{ C_2\gamma}{C_3(C_1+(1+2C_*^2 C_1)\gamma)})$. Then
\begin{eqnarray}
\#\mathcal{M}_{h_{k}}
\leq C_5\Big(\|(y-y_{h_k},p-p_{h_k})\|_{a}^2+ \gamma\mbox{osc}^2((y_{h_k},p_{h_k}),\mathcal{T}_{h_k})\Big)^{-{1\over 2s}}|(y,p,y_d)|_s^{1\over s},\label{upper_DOF}
\end{eqnarray}
where the constant $C_5$ depends on the discrepancy between $\theta$ and $\frac{ C_2\gamma}{C_3(C_1+(1+2C_*^2 C_1)\gamma)}$.
\end{Lemma}
\begin{proof}
Let $\rho,\rho_1\in (0,1)$ satisfy $\rho_1\in (0,\rho)$ and 
\begin{eqnarray}
\theta<\frac{ C_2\gamma}{C_3(C_1+(1+2C_*^2 C_1)\gamma)}(1-\rho^2).\nonumber
\end{eqnarray}
Choose $\delta_1\in (0,1)$ to satisfy (\ref{delta_1}) and
\begin{eqnarray}
(1+\delta_1)^2\rho_1^2\leq \rho^2,\label{T_5}
\end{eqnarray}
 which implies
\begin{eqnarray}
(1+\delta_1)\rho_1^2<1.\label{T_6}
\end{eqnarray}

Set 
\begin{eqnarray}
\varepsilon = \frac{1}{\sqrt{2}}\rho_1\Big(\|(y-y_{h_k},p-p_{h_k})\|_{a}^2+ \gamma\mbox{osc}^2((y_{h_k},p_{h_k}),\mathcal{T}_{h_k})\Big)^{1\over 2}\nonumber
\end{eqnarray}
and let $\mathcal{T}_{h_\varepsilon}$  be a refinement of $\mathcal{T}_{h_0}$ with minimal degrees of freedom satisfying 
\begin{eqnarray}
\|(y-y_{h_\varepsilon},p-p_{h_\varepsilon})\|_{a}^2+ (\gamma+1)\mbox{osc}^2((y_{h_\varepsilon},p_{h_\varepsilon}),\mathcal{T}_{h_\varepsilon})\leq\varepsilon^2.\label{T_7}
\end{eqnarray}
We can conclude from the definition of $\mathcal{A}^s$ that
\begin{eqnarray}
\#\mathcal{T}_{h_\varepsilon}-\#\mathcal{T}_{h_0}\lesssim\varepsilon^{-{1\over s}} |(y,p,y_d)|_s^{1\over s}.\nonumber
\end{eqnarray}
Let $\mathcal{T}_{h_*}:=\mathcal{T}_{h_\varepsilon}\oplus\mathcal{T}_{h_k}$ be the smallest common refinement of $\mathcal{T}_{h_\varepsilon}$ and $\mathcal{T}_{h_k}$. Let $V_{h_\varepsilon}\subset H_0^1(\Omega)$ and $V_{h_*}\subset H_0^1(\Omega)$ be the finite element spaces defined on $\mathcal{T}_{h_\varepsilon}$ and $\mathcal{T}_{h_*}$, respectively. Assume that $(u_{h_\varepsilon},y_{h_\varepsilon},p_{h_\varepsilon})\in U_{ad}\times V_{h_\varepsilon}\times V_{h_\varepsilon}$ is the solution of  problem (\ref{OCP_h})-(\ref{OCP_state_h}). 

Define $y^{h_\varepsilon}:=Su_{h_\varepsilon}$ and $p^{h_\varepsilon}:=S^*(S_{h_\varepsilon}u_{h_\varepsilon}-y_d)$. From the definition of oscillation we can conclude from Lemma \ref{Lm:2.4} that
\begin{eqnarray}
\mbox{osc}(u_{h_\varepsilon}-L\mathcal{R}_{h_*}y^{h_\varepsilon},\mathcal{T}_{h_*})&\leq&\mbox{osc}(u_{h_\varepsilon}-L\mathcal{R}_{h_\varepsilon}y^{h_\varepsilon},\mathcal{T}_{h_*})+\mbox{osc}(L(\mathcal{R}_{h_*}-\mathcal{R}_{h_\varepsilon})y^{h_\varepsilon},\mathcal{T}_{h_*})\nonumber\\
&\leq&\mbox{osc}(u_{h_\varepsilon}-L\mathcal{R}_{h_\varepsilon}y^{h_\varepsilon},\mathcal{T}_{h_*})+C_*\|(\mathcal{R}_{h_*}-\mathcal{R}_{h_\varepsilon})y^{h_\varepsilon}\|_{a}\nonumber
\end{eqnarray}
and 
\begin{eqnarray}
\mbox{osc}(y_{h_\varepsilon}-y_d-L^*\mathcal{R}_{h_*}p^{h_\varepsilon},\mathcal{T}_{h_*})&\leq&\mbox{osc}(y_{h_\varepsilon}-y_d-L^*\mathcal{R}_{h_\varepsilon}p^{h_\varepsilon},\mathcal{T}_{h_*})+\mbox{osc}(L^*(\mathcal{R}_{h_*}-\mathcal{R}_{h_\varepsilon})p^{h_\varepsilon},\mathcal{T}_{h_*})\nonumber\\
&\leq&\mbox{osc}(y_{h_\varepsilon}-y_d-L^*\mathcal{R}_{h_\varepsilon}p^{h_\varepsilon},\mathcal{T}_{h_*})+C_*\|(\mathcal{R}_{h_*}-\mathcal{R}_{h_\varepsilon})p^{h_\varepsilon})\|_{a}.\nonumber
\end{eqnarray}
Then from the Young's inequality we have
\begin{eqnarray}
\mbox{osc}^2((\mathcal{R}_{h_*}y^{h_\varepsilon},\mathcal{R}_{h_*}p^{h_\varepsilon}),\mathcal{T}_{h_*})&\leq&2\mbox{osc}^2((\mathcal{R}_{h_\varepsilon}y^{h_\varepsilon},\mathcal{R}_{h_\varepsilon}p^{h_\varepsilon}),\mathcal{T}_{h_*})\nonumber\\
&&+2C_*^2\|((\mathcal{R}_{h_*}-\mathcal{R}_{h_\varepsilon})y^{h_\varepsilon},(\mathcal{R}_{h_*}-\mathcal{R}_{h_\varepsilon})p^{h_\varepsilon}))\|_{a}^2.\nonumber
\end{eqnarray}
Due to the orthogonality 
\begin{eqnarray}
\|(y^{h_\varepsilon}-\mathcal{R}_{h_*}y^{h_\varepsilon},p^{h_\varepsilon}-\mathcal{R}_{h_*}p^{h_\varepsilon})\|_{a}^2&=&\|(y^{h_\varepsilon}-\mathcal{R}_{h_\varepsilon}y^{h_\varepsilon},p^{h_\varepsilon}-\mathcal{R}_{h_\varepsilon}p^{h_\varepsilon})\|_{a}^2\nonumber\\
&&-\|((\mathcal{R}_{h_*}-\mathcal{R}_{h_\varepsilon})y^{h_\varepsilon},(\mathcal{R}_{h_*}-\mathcal{R}_{h_\varepsilon})p^{h_\varepsilon}))\|_{a}^2,\nonumber
\end{eqnarray}
we arrive at
\begin{eqnarray}
&&\|(y^{h_\varepsilon}-\mathcal{R}_{h_*}y^{h_\varepsilon},p^{h_\varepsilon}-\mathcal{R}_{h_*}p^{h_\varepsilon})\|_{a}^2+\frac{1}{2C_*^2}\mbox{osc}^2((\mathcal{R}_{h_*}y^{h_\varepsilon},\mathcal{R}_{h_*}p^{h_\varepsilon}),\mathcal{T}_{h_*})\nonumber\\
&\leq&\|(y^{h_\varepsilon}-\mathcal{R}_{h_\varepsilon}y^{h_\varepsilon},p^{h_\varepsilon}-\mathcal{R}_{h_\varepsilon}p^{h_\varepsilon})\|_{a}^2+\frac{1}{C_*^2}\mbox{osc}^2((\mathcal{R}_{h_\varepsilon}y^{h_\varepsilon},\mathcal{R}_{h_\varepsilon}p^{h_\varepsilon}),\mathcal{T}_{h_*}).\nonumber
\end{eqnarray}
From Theorem \ref{Thm:elliptic_convergence} we can see that $\tilde\gamma\leq \frac{1}{2C_*^2}$, which implies
\begin{eqnarray}
&&\|(y^{h_\varepsilon}-\mathcal{R}_{h_*}y^{h_\varepsilon},p^{h_\varepsilon}-\mathcal{R}_{h_*}p^{h_\varepsilon})\|_{a}^2+\tilde\gamma\mbox{osc}^2((\mathcal{R}_{h_*}y^{h_\varepsilon},\mathcal{R}_{h_*}p^{h_\varepsilon}),\mathcal{T}_{h_*})\nonumber\\
&\leq&\|(y^{h_\varepsilon}-\mathcal{R}_{h_\varepsilon}y^{h_\varepsilon},p^{h_\varepsilon}-\mathcal{R}_{h_\varepsilon}p^{h_\varepsilon})\|_{a}^2+\frac{1}{C_*^2}\mbox{osc}^2((\mathcal{R}_{h_\varepsilon}y^{h_\varepsilon},\mathcal{R}_{h_\varepsilon}p^{h_\varepsilon}),\mathcal{T}_{h_*})\nonumber\\
&\leq&\|(y^{h_\varepsilon}-\mathcal{R}_{h_\varepsilon}y^{h_\varepsilon},p^{h_\varepsilon}-\mathcal{R}_{h_\varepsilon}p^{h_\varepsilon})\|_{a}^2+(\tilde\gamma+\sigma)\mbox{osc}^2((\mathcal{R}_{h_\varepsilon}y^{h_\varepsilon},\mathcal{R}_{h_\varepsilon}p^{h_\varepsilon}),\mathcal{T}_{h_*})\nonumber
\end{eqnarray}
with $\sigma = \frac{1}{C_*^2}-\tilde\gamma\in (0,1)$. Following the similar procedure as in the proof of Theorem \ref{Thm:4.2} when (\ref{two_level_estimator_y})-(\ref{two_level_estimator_p}) are replaced by (\ref{two_level_data_y})-(\ref{two_level_data_p}), we are led to
\begin{eqnarray}
&&\|(y-y_{h_*},p-p_{h_*})\|_{a}^2+ \gamma\mbox{osc}^2((y_{h_*},p_{h_*}),\mathcal{T}_{h_*})\nonumber\\
&\leq&\beta_0^2\Big(\|(y-y_{h_\varepsilon},p-p_{h_\varepsilon})\|_{a}^2+ (\gamma+\sigma)\mbox{osc}^2((y_{h_\varepsilon},p_{h_\varepsilon}),\mathcal{T}_{h_\varepsilon})\Big)\nonumber\\
&\leq&\beta_0^2\Big(\|(y-y_{h_\varepsilon},p-p_{h_\varepsilon})\|_{a}^2+ (\gamma+1)\mbox{osc}^2((y_{h_\varepsilon},p_{h_\varepsilon}),\mathcal{T}_{h_\varepsilon})\Big),\label{T_8}
\end{eqnarray}
where 
\begin{eqnarray}
\beta_0:=\Big(\frac{(1+\delta_1)+C_4\tilde \kappa(h_0)}{1-C_4\delta_1^{-1}\tilde \kappa^2(h_0)}\Big)^{1\over 2}\nonumber
\end{eqnarray}
and $C_4$ is the constant appeared in the proof of Theorem \ref{Thm:4.2}. Thus, by (\ref{T_7}) and (\ref{T_8}) it follows
\begin{eqnarray}
&&\|(y-y_{h_*},p-p_{h_*})\|_{a}^2+ \gamma\mbox{osc}^2((y_{h_*},p_{h_*}),\mathcal{T}_{h_*})\nonumber\\
&\leq&\beta_1^2\Big(\|(y-y_{h_k},p-p_{h_k})\|_{a}^2+ \gamma\mbox{osc}^2((y_{h_k},p_{h_k}),\mathcal{T}_{h_k})\Big)
\end{eqnarray}
with $\beta_1=\frac{1}{\sqrt{2}}\beta_0\rho_1$.

In view of (\ref{T_6}) we have $\beta_1^2\in(0,{1\over 2})$ provided $h_0\ll 1$. It follows from Corollary \ref{Cor:OCP_decay} that
\begin{eqnarray}
\sum\limits_{T\in R_{\mathcal{T}_{h_k}\rightarrow\mathcal{T}_{h_*}}}\eta_{h_k}^2((y_{h_k},p_{h_k}),T)\geq\theta_1 \sum\limits_{T\in\mathcal{T}_{h_k}}\eta_{h_k}^2((y_{h_k},p_{h_k}),T),
\end{eqnarray}
where $\theta_1=\frac{\tilde C_2(1-2\tilde\beta_1^2)}{\tilde C_5(\tilde C_1+(1+2C_*^2\tilde C_1)\tilde\gamma_1)}$, $\tilde\gamma_1=\frac{\gamma}{1-C_5\delta_1^{-1}\tilde \kappa^2(h_0)}$, $\tilde C_5=\max(1,\frac{\tilde C_3}{\tilde\gamma_1})$ and
\begin{eqnarray}
\tilde\beta_1=\Big(\frac{(1+\delta_1)\beta_1^2+C_5\tilde \kappa(h_0)}{1-C_5\delta_1^{-1}\tilde \kappa^2(h_0)}\Big)^{1\over 2}.\nonumber
\end{eqnarray}
It follows from the definition of $\gamma$ in (\ref{tilde_gamma}) and $\tilde \gamma$ in (\ref{gamma}) that $\tilde\gamma_1<1$, which together with $\tilde C_3>1$ (see \cite{Dai}) implies $\tilde C_5=\frac{\tilde C_3}{\tilde\gamma_1}$. Since $h_0\ll 1$, we obtain that $\tilde\gamma_1>\gamma$ and $\tilde\beta_1\in (0,\frac{1}{\sqrt{2}}\rho)$ from (\ref{T_5}). It is easy to see from (\ref{C123}) and $\tilde\gamma_1>\gamma$ that
\begin{eqnarray}
\theta_1&=&\frac{\tilde C_2(1-2\tilde\beta_1^2)}{\frac{\tilde C_3}{\tilde\gamma_1}(\tilde C_1+(1+2C_*^2\tilde C_1)\tilde\gamma_1)}\geq\frac{\tilde C_2}{\tilde C_3(\frac{\tilde C_1}{\tilde\gamma_1}+1+2C_*^2\tilde C_1)} (1-\rho^2)\nonumber\\
&=&\frac{ C_2(1+\hat C_2\tilde\kappa^2(h_0))}{C_3(1+\hat C_2\tilde\kappa^2(h_0))(\frac{C_1(1-\hat C_1\tilde\kappa^2(h_0))}{2\tilde\gamma_1}+1+C_*^2 C_1(1-\hat C_1\tilde\kappa^2(h_0)))} (1-\rho^2)\nonumber\\
&\geq&\frac{ C_2}{C_3(\frac{C_1}{\gamma}+1+2C_*^2 C_1)} (1-\rho^2)=\frac{ C_2\gamma}{C_3(C_1+(1+2C_*^2 C_1)\gamma)} (1-\rho^2)>\theta,
\end{eqnarray}
provided $h_0\ll 1$. This implies
\begin{eqnarray}
\sum\limits_{T\in R_{\mathcal{T}_{h_k}\rightarrow\mathcal{T}_{h_*}}}\eta_{h_k}^2((y_{h_k},p_{h_k}),T)>\theta \sum\limits_{T\in\mathcal{T}_{h_k}}\eta_{h_k}^2((y_{h_k},p_{h_k}),T).\nonumber
\end{eqnarray}

Note that Algorithm \ref{Alg:3.0} selects a minimal set $\mathcal{M}_{h_k}=\tilde{\mathcal{T}}_{h_k}$ satisfying
\begin{eqnarray}
\sum\limits_{T\in \mathcal{M}_{h_k}}\eta_{h_k}^2((y_{h_k},p_{h_k}),T)\geq\theta \sum\limits_{T\in\mathcal{T}_{h_k}}\eta_{h_k}^2((y_{h_k},p_{h_k}),T).\nonumber
\end{eqnarray}
Thus,
\begin{eqnarray}
\#\mathcal{M}_{h_k}&\leq& \#R_{\mathcal{T}_{h_k}\rightarrow\mathcal{T}_{h_*}}\leq \#\mathcal{T}_{h_*}-\#\mathcal{T}_{h_k}\leq \#\mathcal{T}_{h_\varepsilon}-\#\mathcal{T}_{h_0}\nonumber\\
&\leq&(\frac{1}{\sqrt{2}}\rho_1)^{-{1\over s}}\Big(\|(y-y_{h_k},p-p_{h_k})\|_{a}^2+ \gamma\mbox{osc}^2((y_{h_k},p_{h_k}),\mathcal{T}_{h_k})\Big)^{-{1\over 2s}} |(y,p,y_d)|_s^{1\over s},\nonumber
\end{eqnarray}
which is the desired result with an explicit dependance on the discrepancy between $\theta$ and $\frac{ C_2\gamma}{C_3(C_1+(1+2C_*^2 C_1)\gamma)}$.
\end{proof}

We are now ready to prove that Algorithm \ref{Alg:3.1} possesses optimal complexity for the state and adjoint state approximations.
\begin{Theorem}\label{Thm:complexity}
Let $(u,y,p)\in U_{ad}\times H_0^1(\Omega)\times H_0^1(\Omega)$ be the solution of problem (\ref{OCP})-(\ref{OCP_state}) and $(u_{h_n},y_{h_n},p_{h_n})\in U_{ad}\times V_{h_n}\times V_{h_n}$ be a sequence of solutions of problem (\ref{OCP_h})-(\ref{OCP_state_h}) corresponding to a sequence of finite element spaces $V_{h_n}$ with partitions $\mathcal{T}_{h_n}$ produced by Algorithm \ref{Alg:3.1}. Then the $n$-th
iterate solution $(y_{h_n},p_{h_n})$ of Algorithm \ref{Alg:3.1} satisfies the optimal bound
\begin{eqnarray}
\|(y-y_{h_n},p-p_{h_n})\|_{a}^2+ \gamma\mbox{osc}^2((y_{h_n},p_{h_n}),\mathcal{T}_{h_n})
\lesssim (\#\mathcal{T}_{h_n}-\#\mathcal{T}_{h_0})^{-2s},\label{optimality}
\end{eqnarray}
where the hidden constant depends on the exact solution $(u,y,p)$ and the discrepancy between $\theta$ and $\frac{ C_2\gamma}{C_3(C_1+(1+2C_*^2 C_1)\gamma)}$.
\end{Theorem}
\begin{proof}
It follows from (\ref{cardi}) and (\ref{upper_DOF}) that
\begin{eqnarray}
&&\#\mathcal{T}_{h_n}-\#\mathcal{T}_{h_0}\lesssim\sum\limits_{k=0}^{n-1}\#\mathcal{M}_{h_k}\nonumber\\
&\lesssim&\sum\limits_{k=0}^{n-1}\Big(\|(y-y_{h_k},p-p_{h_k})\|_{a}^2+ \gamma\mbox{osc}^2((y_{h_k},p_{h_k}),\mathcal{T}_{h_k})\Big)^{-{1\over 2s}} |(y,p,y_d)|_s^{1\over s}.\label{Z_1}
\end{eqnarray}
From the lower bound (\ref{u_lower}) we have
\begin{eqnarray}
\|(y-y_{h_k},p-p_{h_k})\|_{a}^2+ \gamma\eta^2_{h_k}((y_{h_k},p_{h_k}),\Omega)\leq C_6\Big(\|(y-y_{h_k},p-p_{h_k})\|_{a}^2+ \gamma\mbox{osc}^2((y_{h_k},p_{h_k}),\mathcal{T}_{h_k})\Big),\nonumber
\end{eqnarray}
where $C_6=\max(1+\frac{\gamma}{C_2},\frac{ C_3}{C_2})$. Then we arrive at
\begin{eqnarray}
\#\mathcal{T}_{h_n}-\#\mathcal{T}_{h_0}
\lesssim\sum\limits_{k=0}^{n-1}\Big(\|(y-y_{h_k},p-p_{h_k})\|_{a}^2+ \gamma\eta_{h_k}^2((y_{h_k},p_{h_k}),\Omega)\Big)^{-{1\over 2s}} |(y,p,y_d)|_s^{1\over s}.\label{Z_2}
\end{eqnarray}
Due to (\ref{convergence}), we obtain for $0\leq k< n$ that
\begin{eqnarray}
\|(y-y_{h_n},p-p_{h_n})\|_{a}^2+ \gamma\eta_{h_n}^2((y_{h_n},p_{h_n}),\Omega)\leq \beta^{2(n-k)}\Big(\|(y-y_{h_k},p-p_{h_k})\|_{a}^2+ \gamma\eta_{h_k}^2((y_{h_k},p_{h_k}),\Omega)\Big).\nonumber
\end{eqnarray}
Thus, 
\begin{eqnarray}
\#\mathcal{T}_{h_n}-\#\mathcal{T}_{h_0}
&\lesssim&\Big(\|(y-y_{h_n},p-p_{h_n})\|_{a}^2+ \gamma\eta_{h_n}^2((y_{h_n},p_{h_n}),\Omega)\Big)^{-{1\over 2s}} |(y,p,y_d)|_s^{1\over s}\sum\limits_{k=0}^{n-1}\beta^{\frac{n-k}{s}}\nonumber\\
&\lesssim&\Big(\|(y-y_{h_n},p-p_{h_n})\|_{a}^2+ \gamma\eta_{h_n}^2((y_{h_n},p_{h_n}),\Omega)\Big)^{-{1\over 2s}} |(y,p,y_d)|_s^{1\over s},\label{Z_3}
\end{eqnarray}
where the last inequality holds due to the fact that $\beta<1$. 

From (\ref{domi_1})-(\ref{domi_2}) we have
\begin{eqnarray}
 \mbox{osc}^2((y_{h_n},p_{h_n}),\mathcal{T}_{h_n})\leq\eta_{h_n}^2((y_{h_n},p_{h_n}),\Omega),\nonumber
\end{eqnarray}
which together with (\ref{Z_3}) yields
\begin{eqnarray}
\#\mathcal{T}_{h_n}-\#\mathcal{T}_{h_0}
\lesssim\Big(\|(y-y_{h_n},p-p_{h_n})\|_{a}^2+ \gamma\mbox{osc}^2((y_{h_n},p_{h_n}),\mathcal{T}_{h_n})\Big)^{-{1\over 2s}},\label{Z_4}
\end{eqnarray}
 this completes the proof.
\end{proof}
\begin{Remark}\label{Re:control_conv}
From (\ref{lift}) and the equivalence property (\ref{u_equivalence}) we can conclude that Theorem \ref{Thm:4.2} also implies the convergence of $\|u-u_{h_k}\|_{0,\Omega}$, namely, for the $n$-th iterate solution $u_{h_n}$ of Algorithm \ref{Alg:3.1} there holds
\begin{eqnarray}
\|u-u_{h_n}\|_{0,\Omega}^2\lesssim \beta^{2n}. 
\end{eqnarray}
We remark that the control variable can also be included into the complexity analysis of AFEM for optimal control problems to obtain
\begin{eqnarray}
\|u-u_{h_n}\|_{0,\Omega}^2\lesssim (\#\mathcal{T}_{h_n}-\#\mathcal{T}_{h_0})^{-2s}. 
\end{eqnarray}
However, the above results are sub-optimal for the optimal control as illustrated by the numerical results in Section 6.  To prove the optimality of AFEM for control variable it seems that we need to work with AFEM based on $L^2$-norm error estimators, we refer to \cite{Hinze05COAP} for optimal a priori error estimate. We expect that the results in \cite{Demlow} will enable us to prove the optimal convergence of AFEM for the optimal control $u$, this will be postponed to future work.
\end{Remark}

\section{Numerical experiments}
\setcounter{equation}{0}
In this section we carry out some numerical tests in two dimensions to support our theoretical results obtained in this paper. We take the elliptic operator $L$ as $-\Delta$ with homogeneous Dirichlet boundary condition for all the examples.

\begin{Example}\label{Exm:1}
This example is taken from \cite{Apel}. The domain $\Omega$ can be described in polar coordinates by
\begin{eqnarray}
\Omega=\{(r,\vartheta),\ \ 1<r<1,\ \ \ 0<\vartheta<{3\over 2}\pi\}.\nonumber
\end{eqnarray}
We take the exact solutions as
\begin{eqnarray}
y(r,\vartheta)&=&(r^\lambda-r^{\nu_1})\sin(\lambda\vartheta),\nonumber\\
p(r,\vartheta)&=&\alpha(r^\lambda-r^{\nu_2})\sin(\lambda\vartheta),\nonumber\\
u(r,\vartheta)&=&P_{U_{ad}}(-\frac{p}{\alpha})\nonumber
\end{eqnarray}
with $\lambda={2\over 3}$ and $\nu_1=\nu_2={5\over 2}$. We set $\alpha = 0.1$, $a=-0.3$ and $b=1$.  We assume the additional right hand side $f$ for the state equation.
\end{Example}

We give the numerical results for the optimal control approximation by Algorithm \ref{Alg:3.1} with parameter $\theta = 0.4$ and $\theta = 0.5$. Figure \ref{fig:1} shows the profiles of the numerically computed optimal state and adjoint state. We present in Figure \ref{fig:2} the triangulations by Algorithm \ref{Alg:3.1} after 8 and 10 adaptive iterations. We can see that the meshes are concentrated on the reentrant corner where the singularities located. 
\begin{figure}[ht]
\centering
\includegraphics[width=7cm,height=7cm]{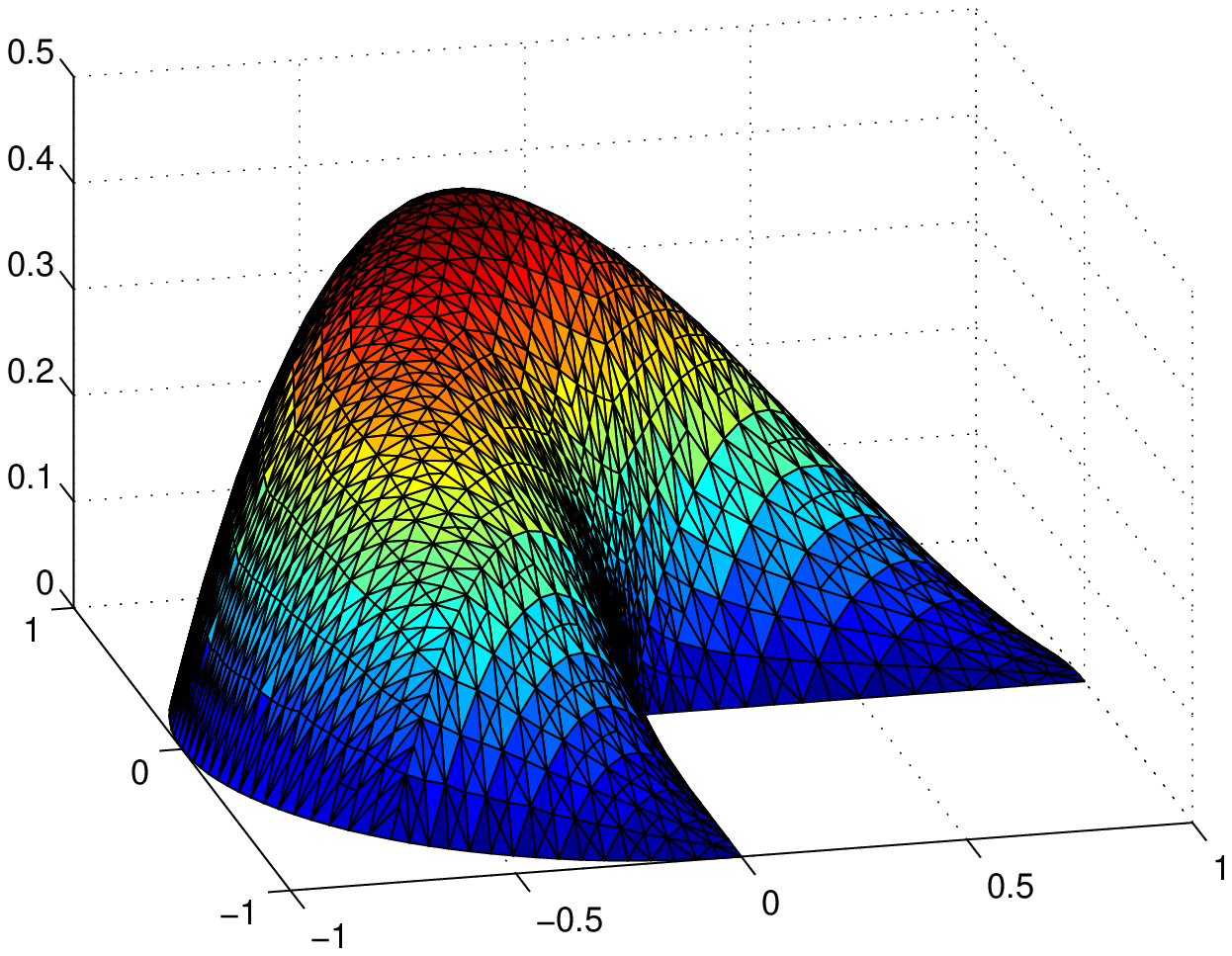}\hspace{0.3cm}
\includegraphics[width=7cm,height=7cm]{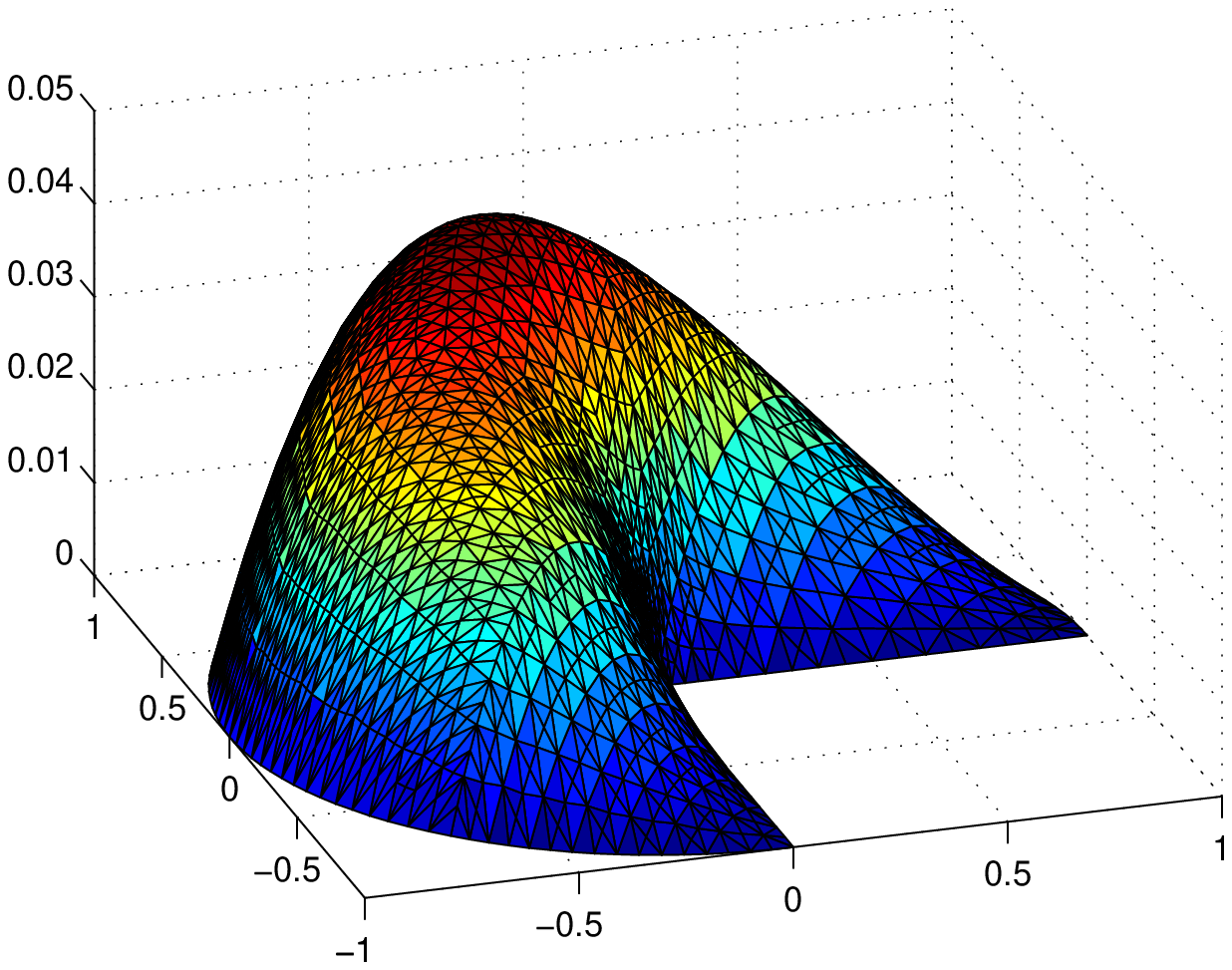}
\caption{The profiles of the discretised optimal state $y_h$ (left) and adjoint state $p_h$ (right)
for Example \ref{Exm:1} on adaptively refined mesh.}
\label{fig:1}
\end{figure}

\begin{figure}[ht]
\centering
\includegraphics[width=7cm,height=6cm]{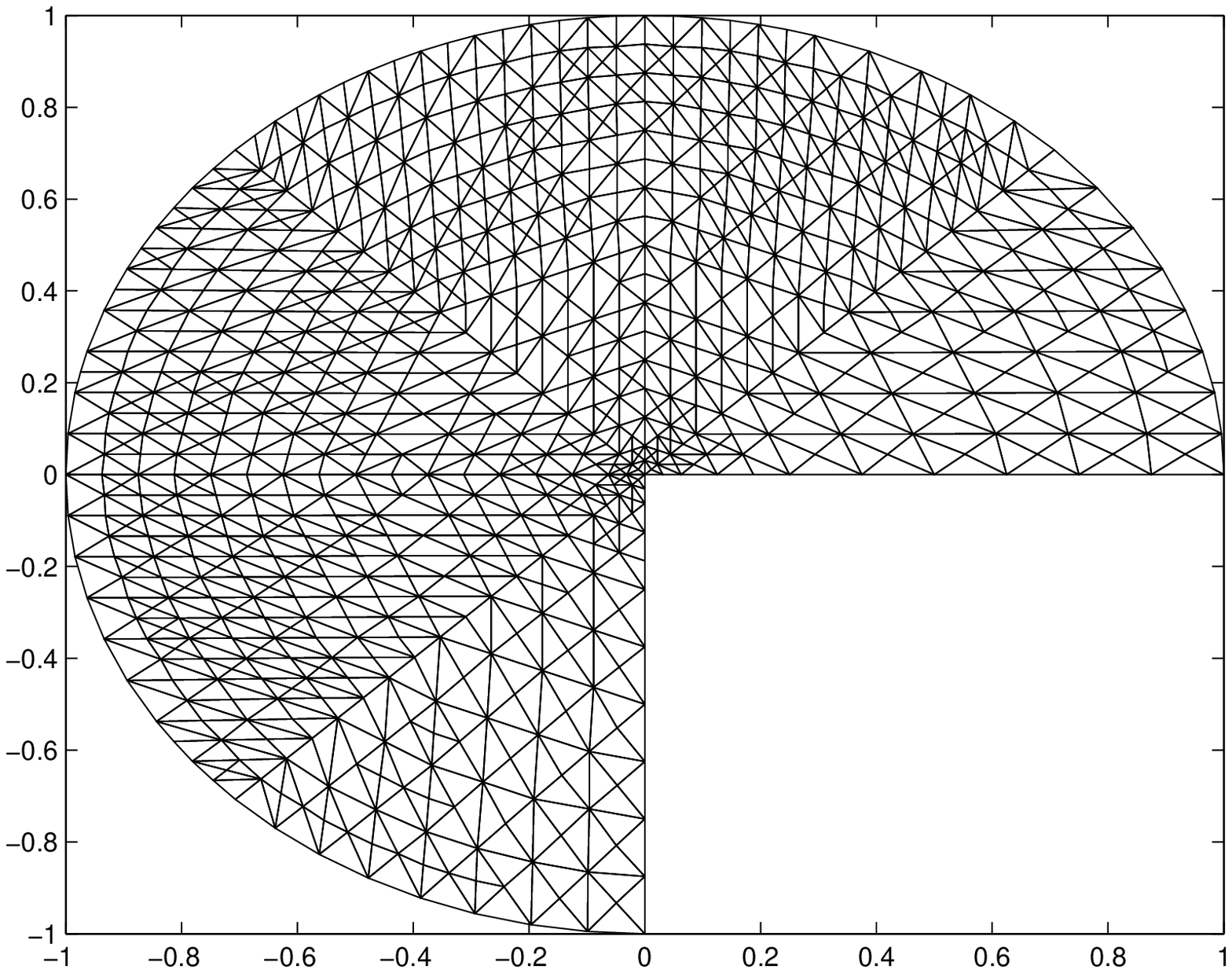}\hspace{0.05cm}
\includegraphics[width=7cm,height=6cm]{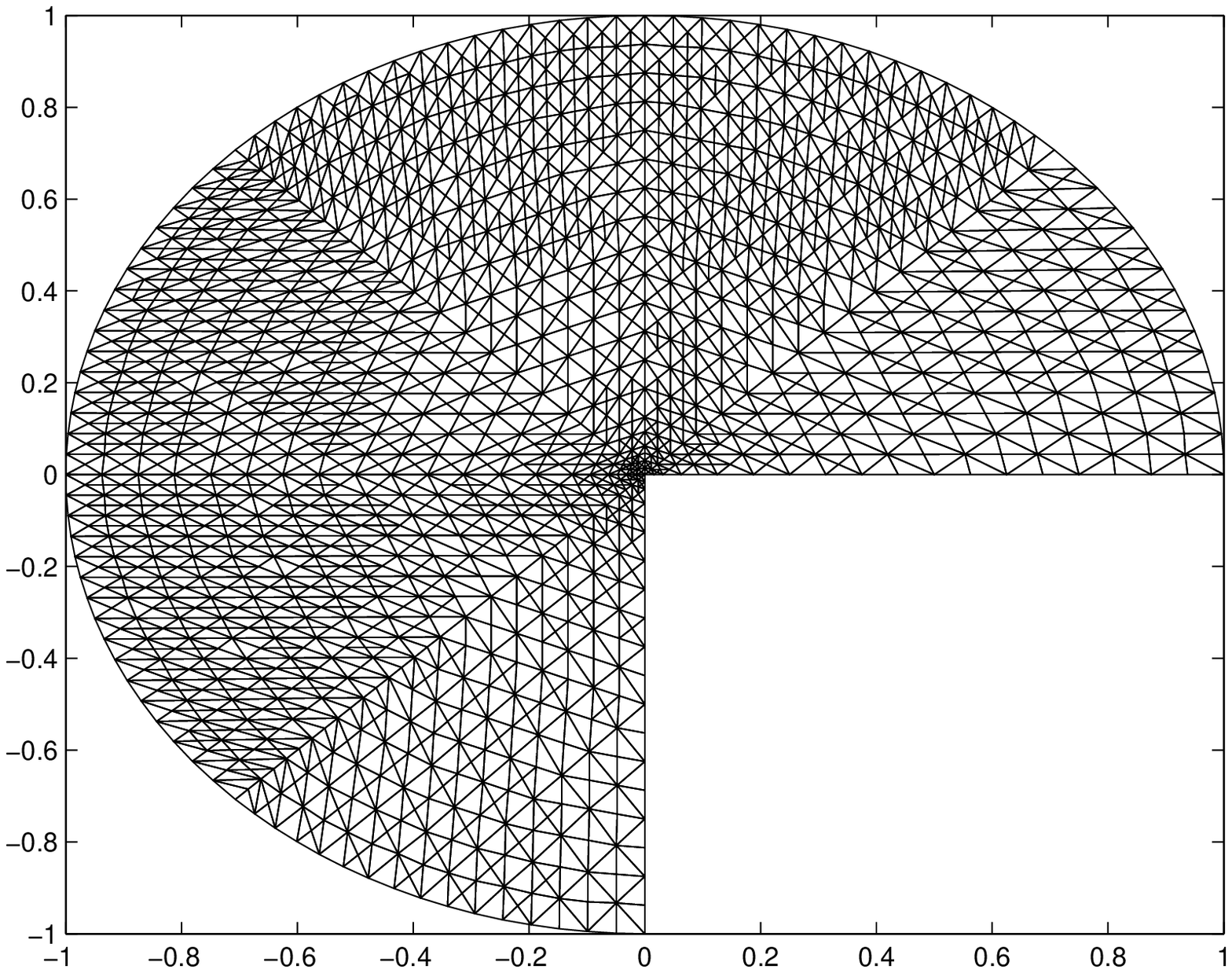}
\caption{The meshes after 8 (left) and 10 (right) adaptive iterations 
for Example \ref{Exm:1} generated by Algorithm \ref{Alg:3.1} with $\theta =0.4$.}
\label{fig:2}
\end{figure}

To illustrate the efficiency of adaptive finite element method for solving optimal control problems, we show in the left plot of Figure \ref{fig:3} the error histories of the optimal control, state and adjoint state with uniform refinement. We can only observe the reduced orders of convergence which are less than one for the energy norms of the state and adjoint state, and less than two for the $L^2$-norm of the control. In the right plot of Figure \ref{fig:3} we present the convergence behaviours of the optimal control, state and adjoint state, as well as the error estimators $\eta_{y,h}(y_h,\Omega)$ and $\eta_{p,y}(p_h,\Omega)$ for the state and adjoint state equations with adaptive refinement. In Figure \ref{fig:4} we present the convergence of the error $\|(y-y_h,p-p_h)\|_a$ and error indicator $\eta_h((y_h,p_h),\Omega)$ with $\theta = 0.4$ and $\theta = 0.5$, respectively. It is shown from Figure \ref{fig:4} that the error $\|(y-y_h,p-p_h)\|_a$ is proportional to the a posteriori error estimators, which implies the efficiency of the a posteriori error estimators given in Section 3. Moreover, we can also observe that the convergence order of error $\|(y-y_h,p-p_h)\|_a$ is approximately parallel to the line with slope $-1/2$ which is the optimal convergence rate we can expect by using linear finite elements, this coincides with our theory in Section 5. For the error $\|u-u_h\|_{0,\Omega}$ we can observe the reduction with slope $-1$, which is better than the results presented in Remark \ref{Re:control_conv}, and strongly suggests that the convergence rate for the optimal control is not optimal.

\begin{figure}[ht]
\centering
\includegraphics[width=7cm,height=7cm]{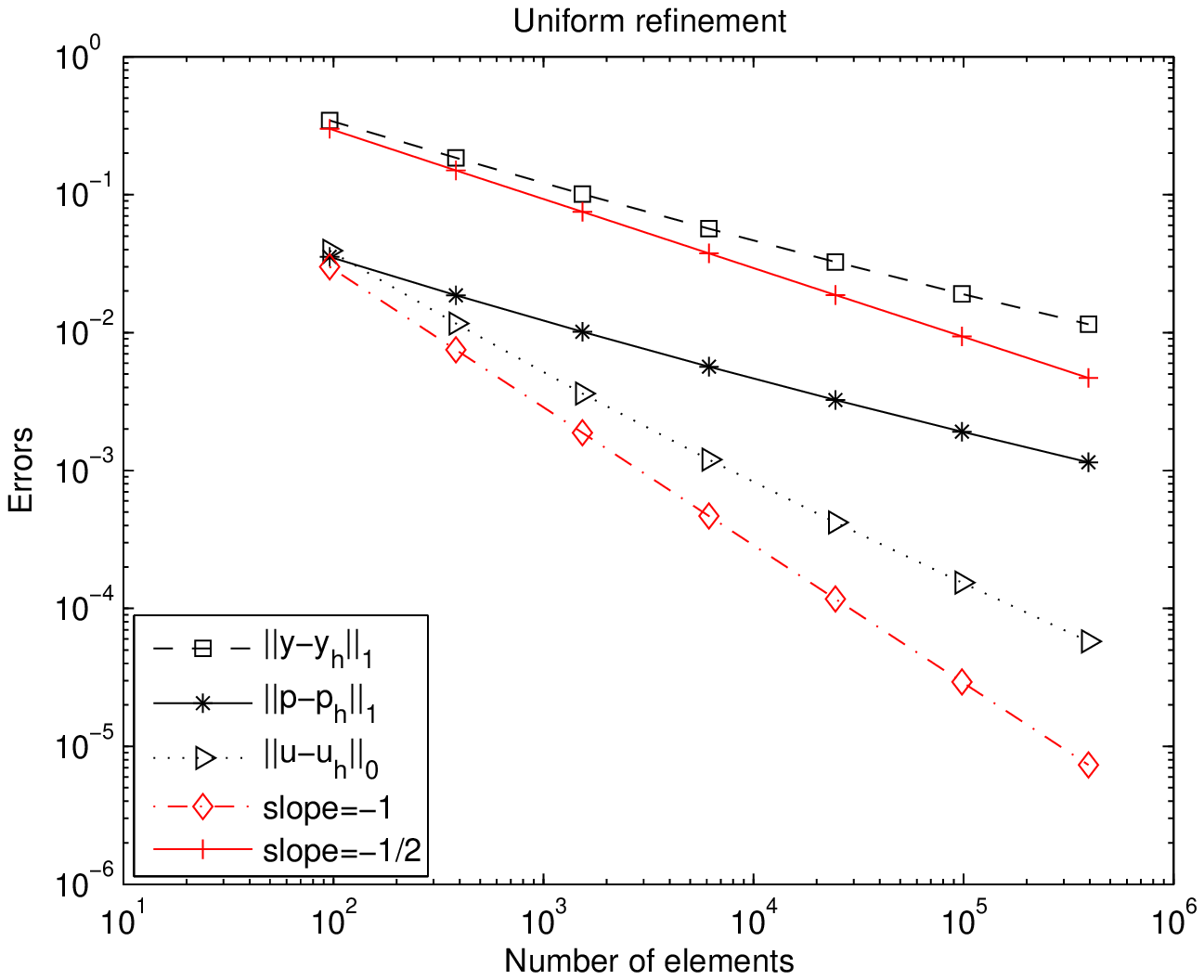}\hspace{0.3cm}
\includegraphics[width=7cm,height=7cm]{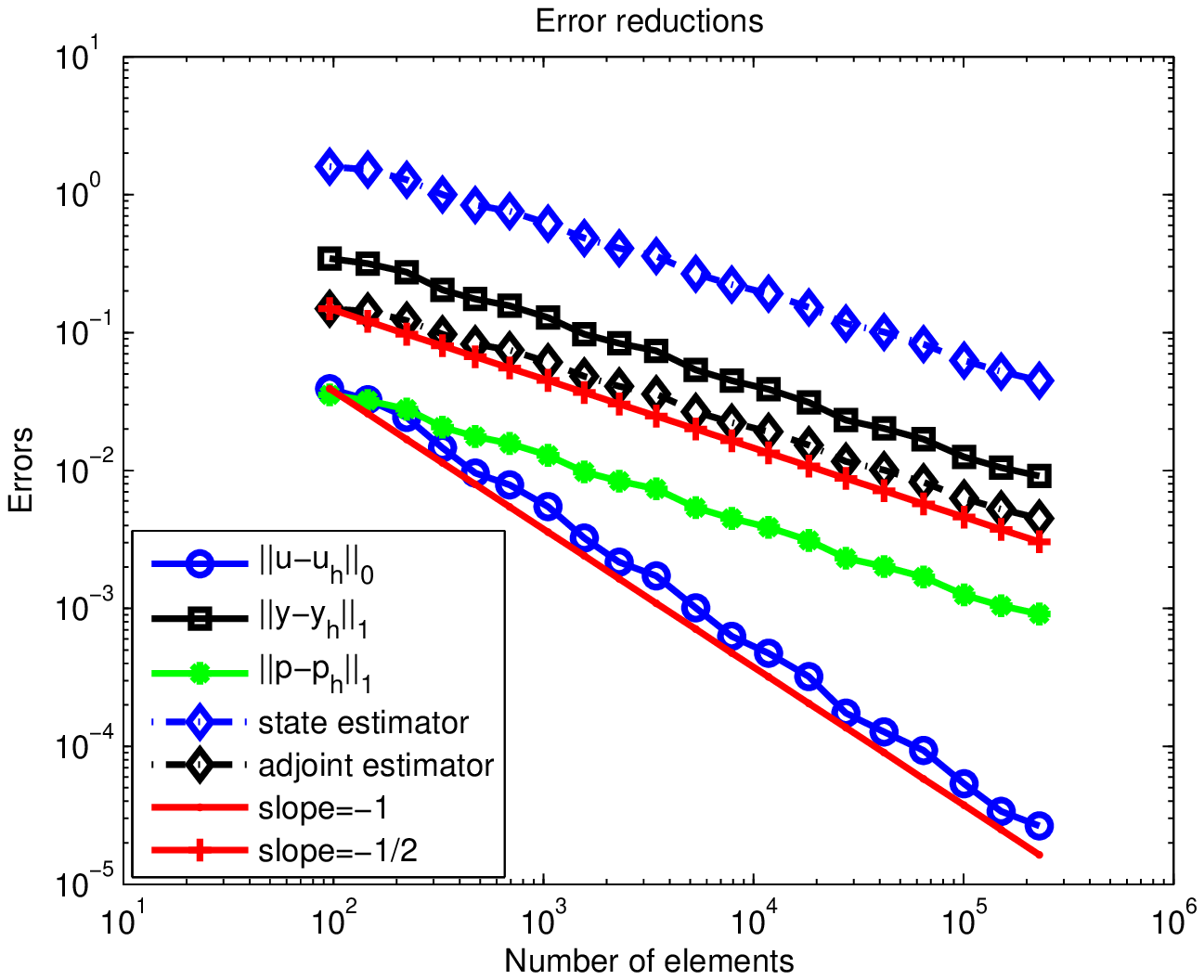}
\caption{The convergence history of the optimal control, state and adjoint state on uniformly refined meshes (left), and the convergence of the errors and estimators on adaptively refined meshes (right) for Example \ref{Exm:1} generated by Algorithm \ref{Alg:3.1}.}
\label{fig:3}
\end{figure}

\begin{figure}[ht]
\centering
\includegraphics[width=7cm,height=7cm]{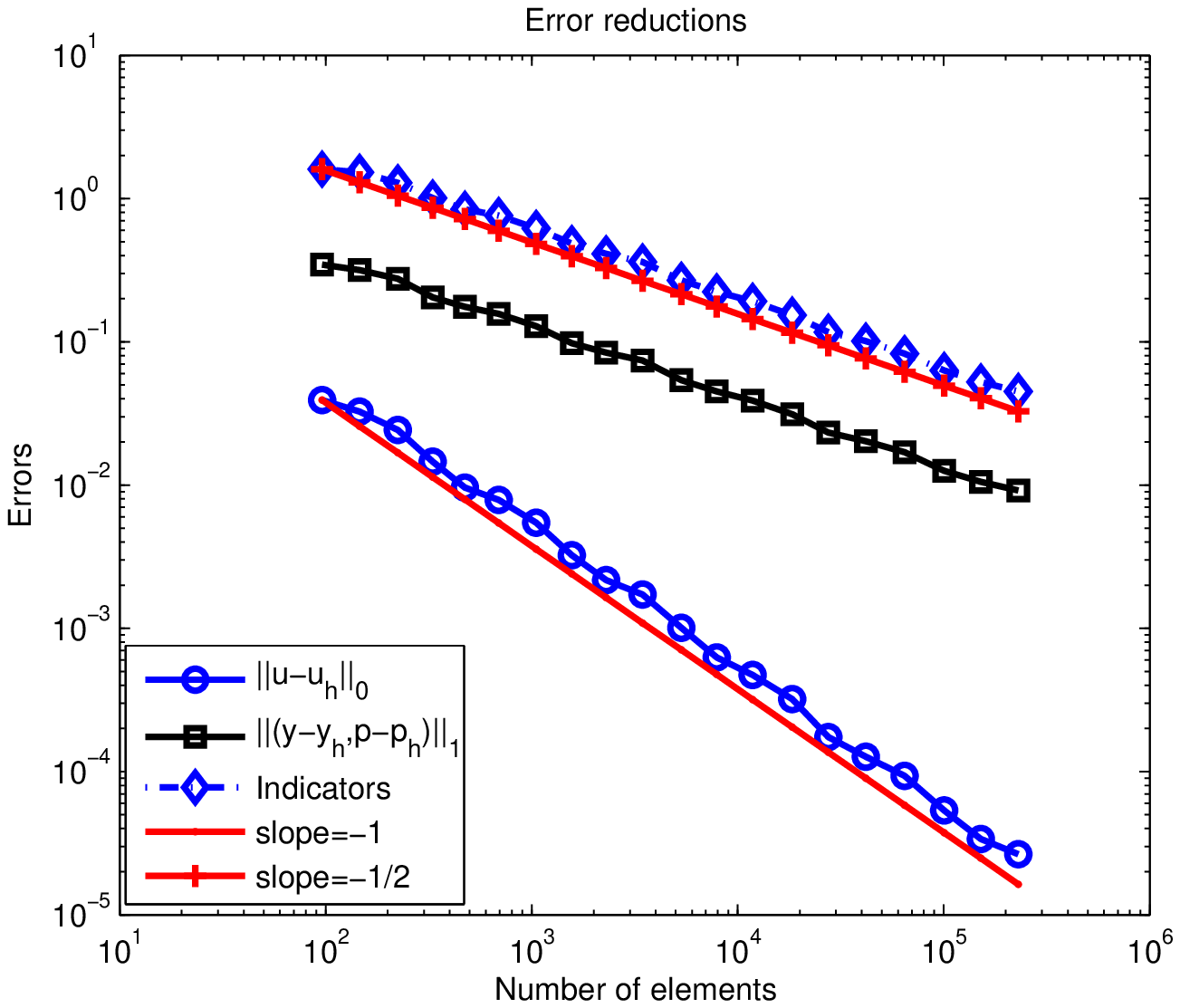}\hspace{0.3cm}
\includegraphics[width=7cm,height=7cm]{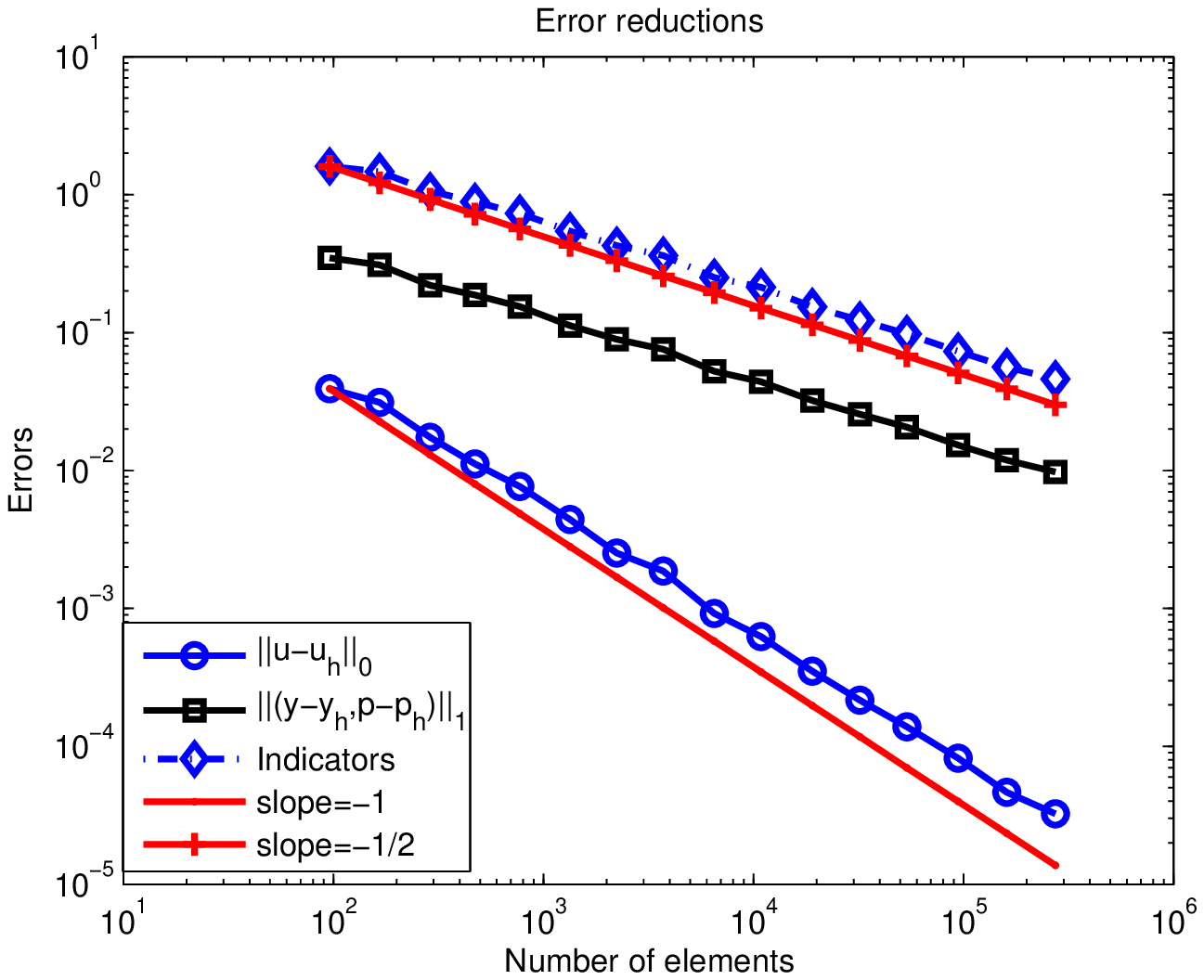}
\caption{The convergence history of the optimal control, the state and adjoint state and error indicator on adaptively refined meshes with $\theta=0.4$ (left) and $\theta=0.5$ (right) for Example \ref{Exm:1} generated by Algorithm \ref{Alg:3.1}.}
\label{fig:4}
\end{figure}

\begin{Example}\label{Exm:2}
In the second example we consider an optimal control problem without explicit solutions. we set $\Omega=(-1,1)^2$, $\alpha = 10^{-3}$, $a=-10$ and $b=10$. The desired state $y_d$ is chosen as $10$, $1$, $-10$ and $-1$ in the first, second, third and fourth quadrant, respectively.
\end{Example}

Similar to the above example Figure \ref{fig:5} shows the profiles of the numerically computed optimal state and adjoint state. We present in the left plot of Figure \ref{fig:6} the triangulation generated by Algorithm \ref{Alg:3.1} after 8 adaptive iteration with parameter $\theta=0.5$. Since there are no explicit solutions we can not show the convergence of the error  $\|(y-y_h,p-p_h)\|_a$ as in Example \ref{Exm:1}. Instead we show in the right plot of Figure \ref{fig:6} the convergence of the error indicator $\eta_h((y_h,p_h),\Omega)$, the error estimators $\eta_{y,h}(y_h,\Omega)$ and $\eta_{p,y}(p_h,\Omega)$ for the state and adjoint state equations. We can observe the error reduction with slope $-1/2$.

\begin{figure}[ht]
\centering
\includegraphics[width=7cm,height=7cm]{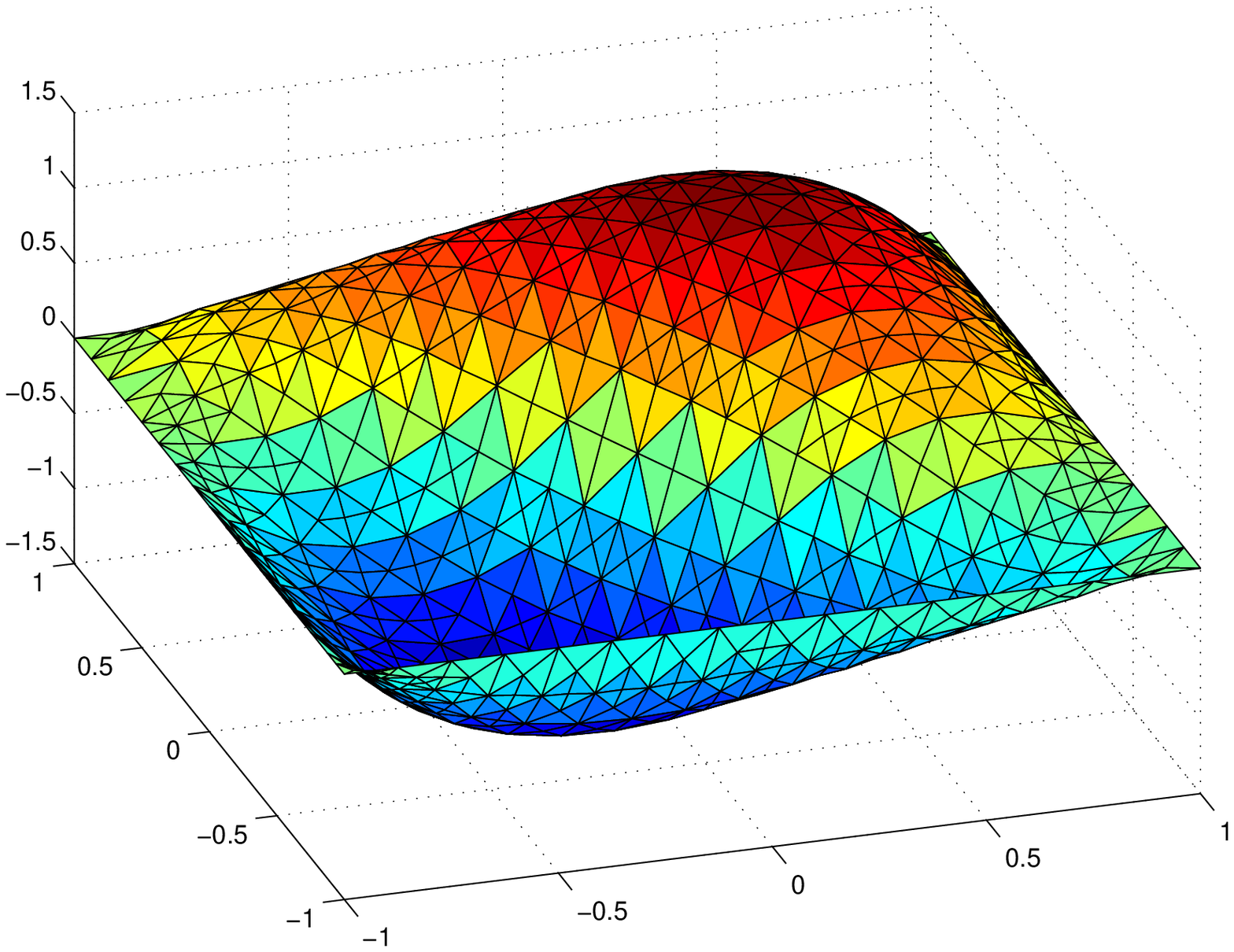}\hspace{0.3cm}
\includegraphics[width=7cm,height=7cm]{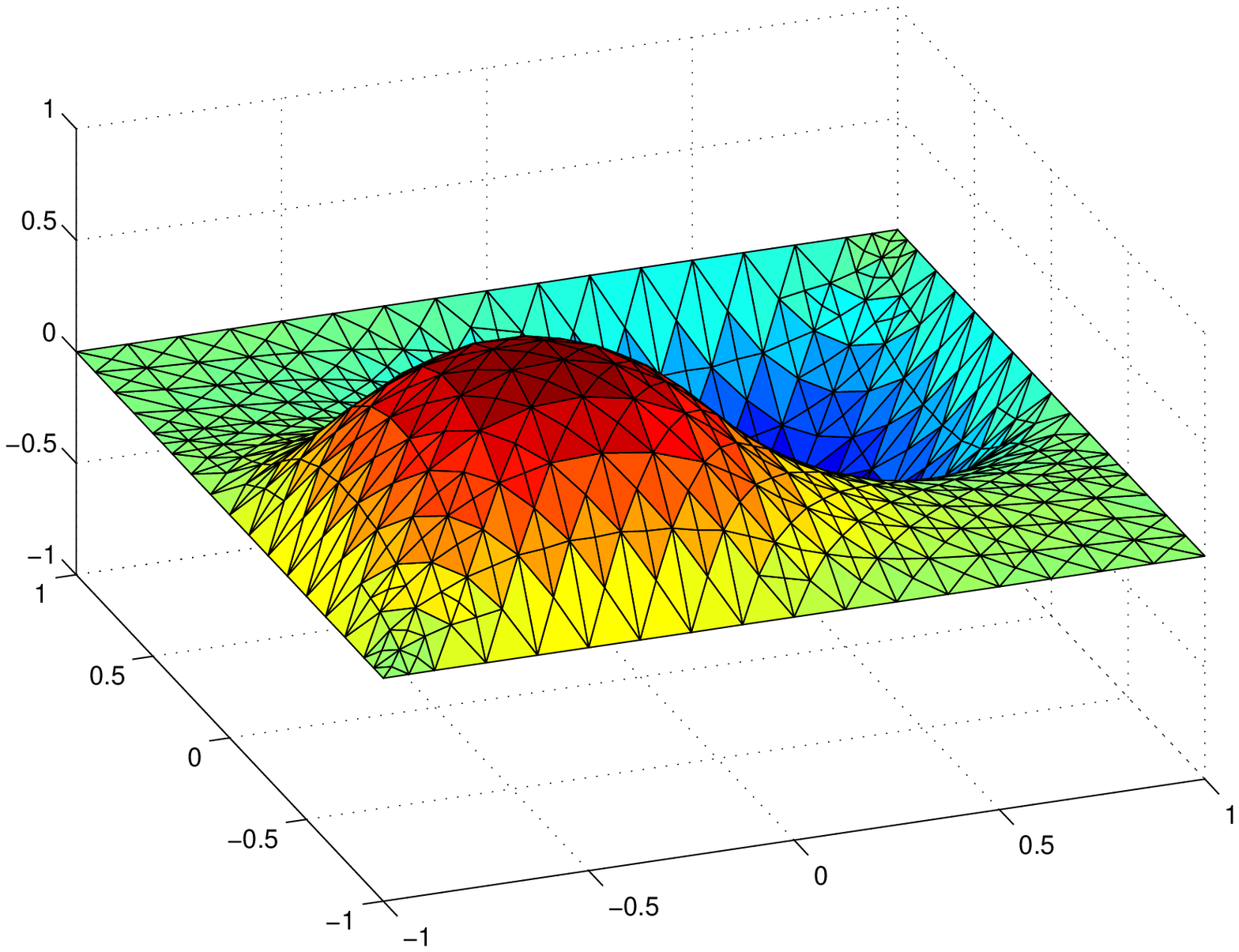}
\caption{The profiles of the discretised optimal state $y_h$ (left) and adjoint state $p_h$ (right)
for Example \ref{Exm:2} on adaptively refined mesh.}
\label{fig:5}
\end{figure}

\begin{figure}[ht]
\centering
\includegraphics[width=7.5cm,height=6cm]{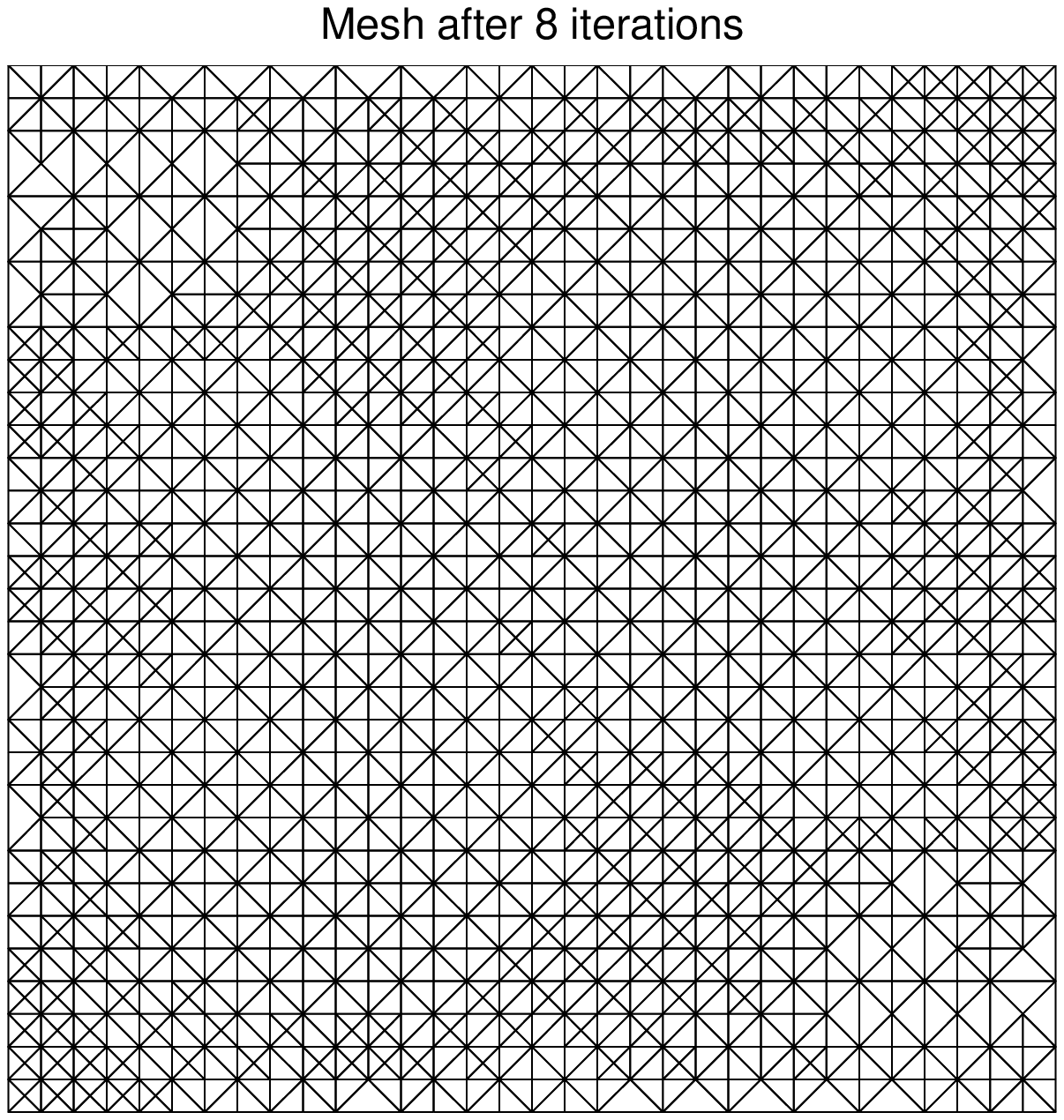}\hspace{0.05cm}
\includegraphics[width=7cm,height=6cm]{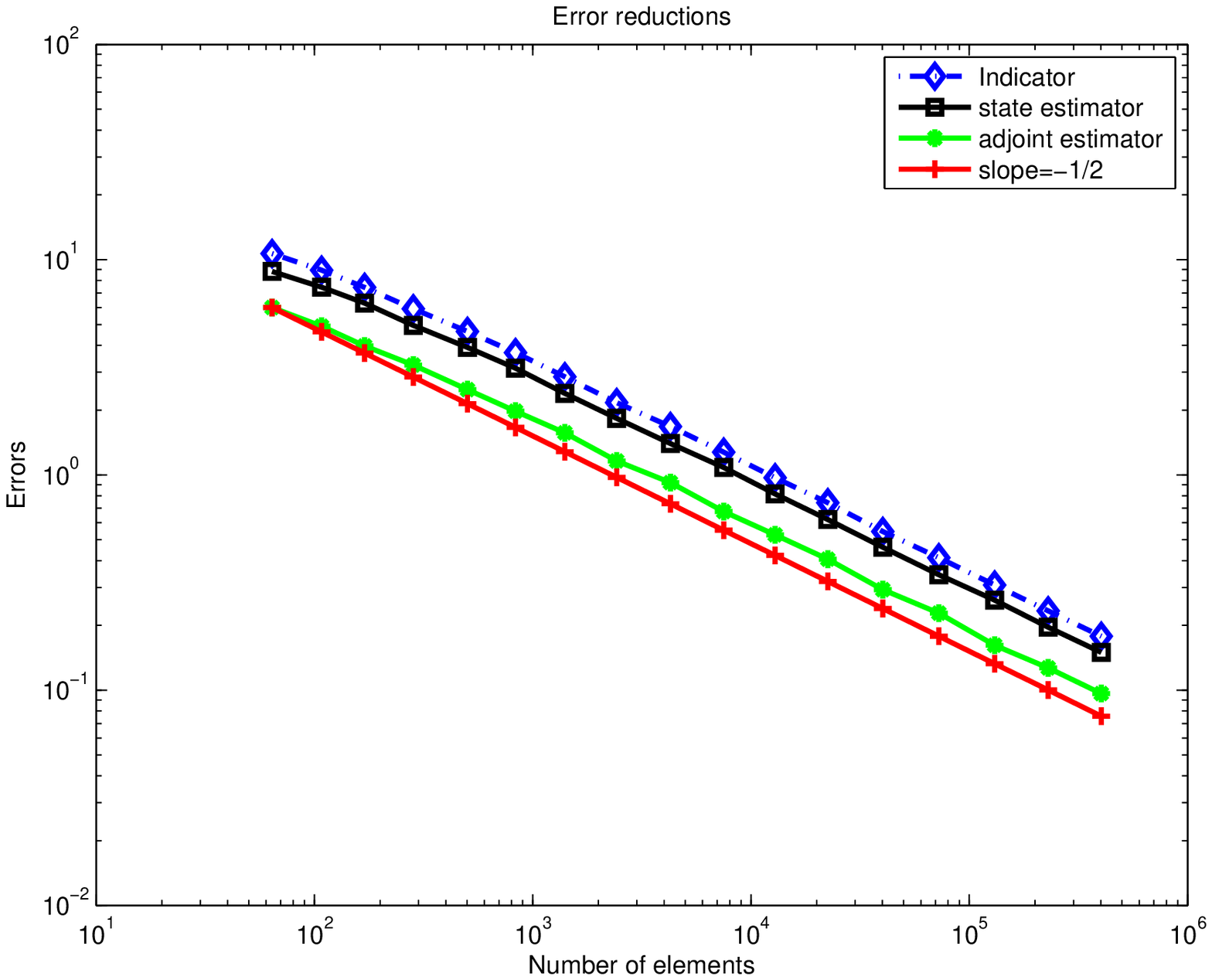}
\caption{The mesh (left) after 8 adaptive iterations and the convergence history of the error estimators on adaptively refined meshes (right) with $\theta=0.5$
for Example \ref{Exm:2} generated by Algorithm \ref{Alg:3.1}.}
\label{fig:6}
\end{figure}

\begin{Example}\label{Exm:3}
In the third example we also consider an optimal control problem without explicit solutions defined on domain $\Omega=(-1,1)\times (-1,1)\backslash [0,1)\times (x_1,0]$. We set $\alpha = 10^{-2}$, $a=0$ and $b=8$. We take the desired state $y_d=2$.
\end{Example}

We show in Figure \ref{fig:7} the profiles of the numerically computed optimal state and adjoint state, singularities for both the state and adjoint state can be observed around the reentrant corner. We present in the left plot of Figure \ref{fig:8} the triangulation generated by Algorithm \ref{Alg:3.1} after 8 adaptive iteration with parameter $\theta=0.5$ which is locally refined around the corner. Since there are no explicit solutions we also show in the right plot of Figure \ref{fig:8} the convergence of the error indicator $\eta_h((y_h,p_h),\Omega)$, the error estimators $\eta_{y,h}(y_h,\Omega)$ and $\eta_{p,y}(p_h,\Omega)$ for the state and adjoint state equations. We can also observe the error reduction with slope $-1/2$.

\begin{figure}[ht]
\centering
\includegraphics[width=7cm,height=7cm]{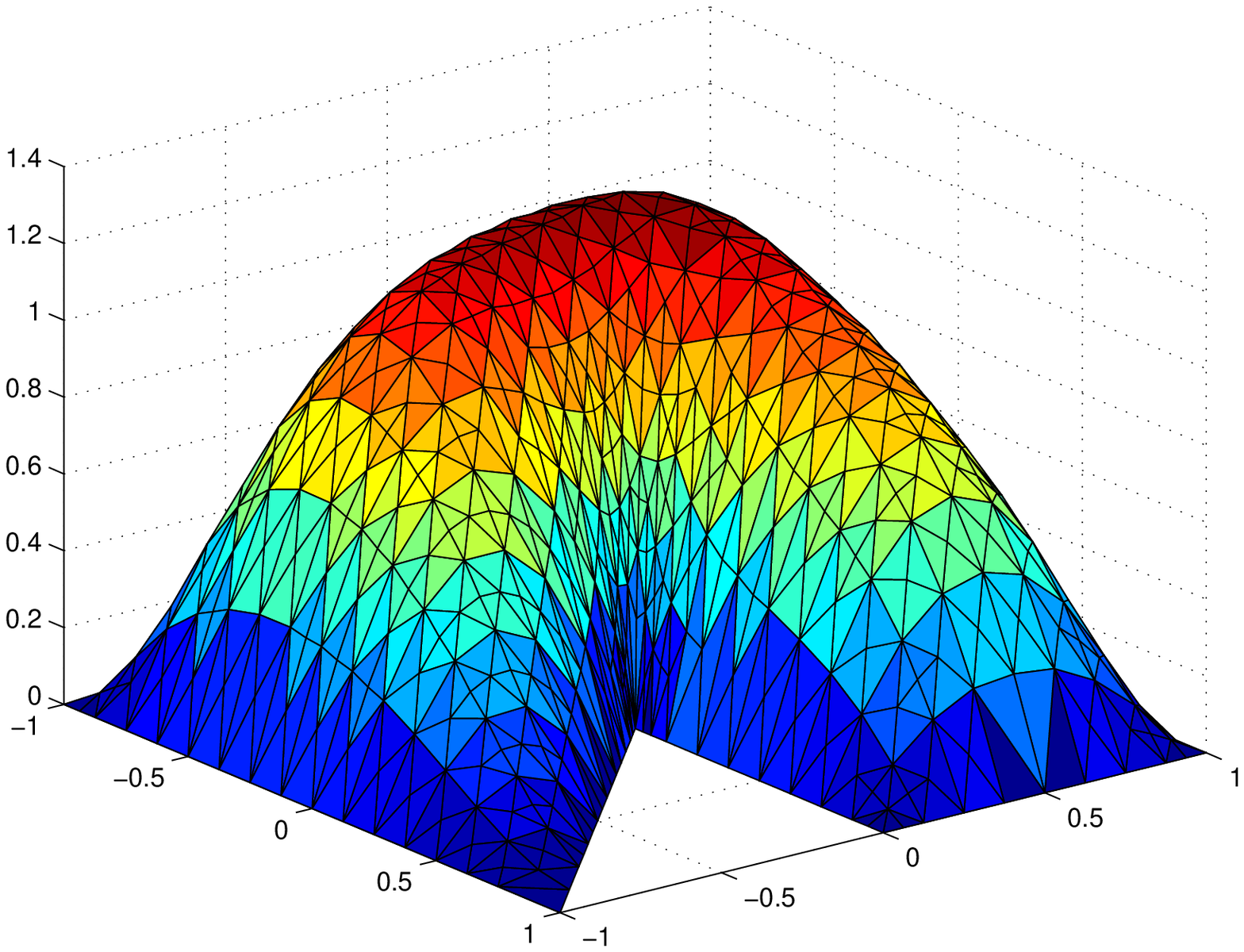}\hspace{0.3cm}
\includegraphics[width=7cm,height=7cm]{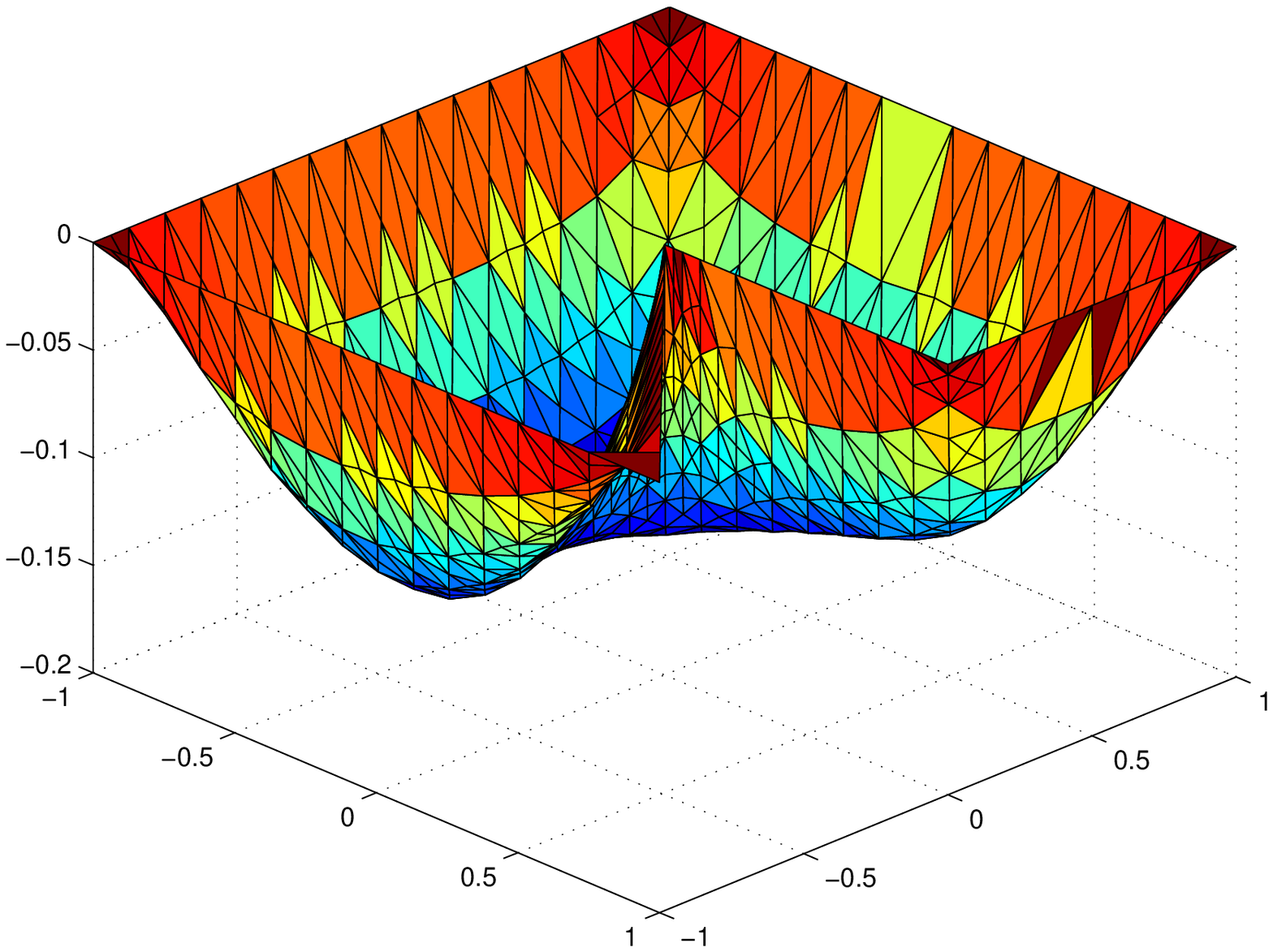}
\caption{The profiles of the discretised optimal state $y_h$ (left) and adjoint state $p_h$ (right)
for Example \ref{Exm:3} on adaptively refined mesh.}
\label{fig:7}
\end{figure}

\begin{figure}[ht]
\centering
\includegraphics[width=7.5cm,height=6cm]{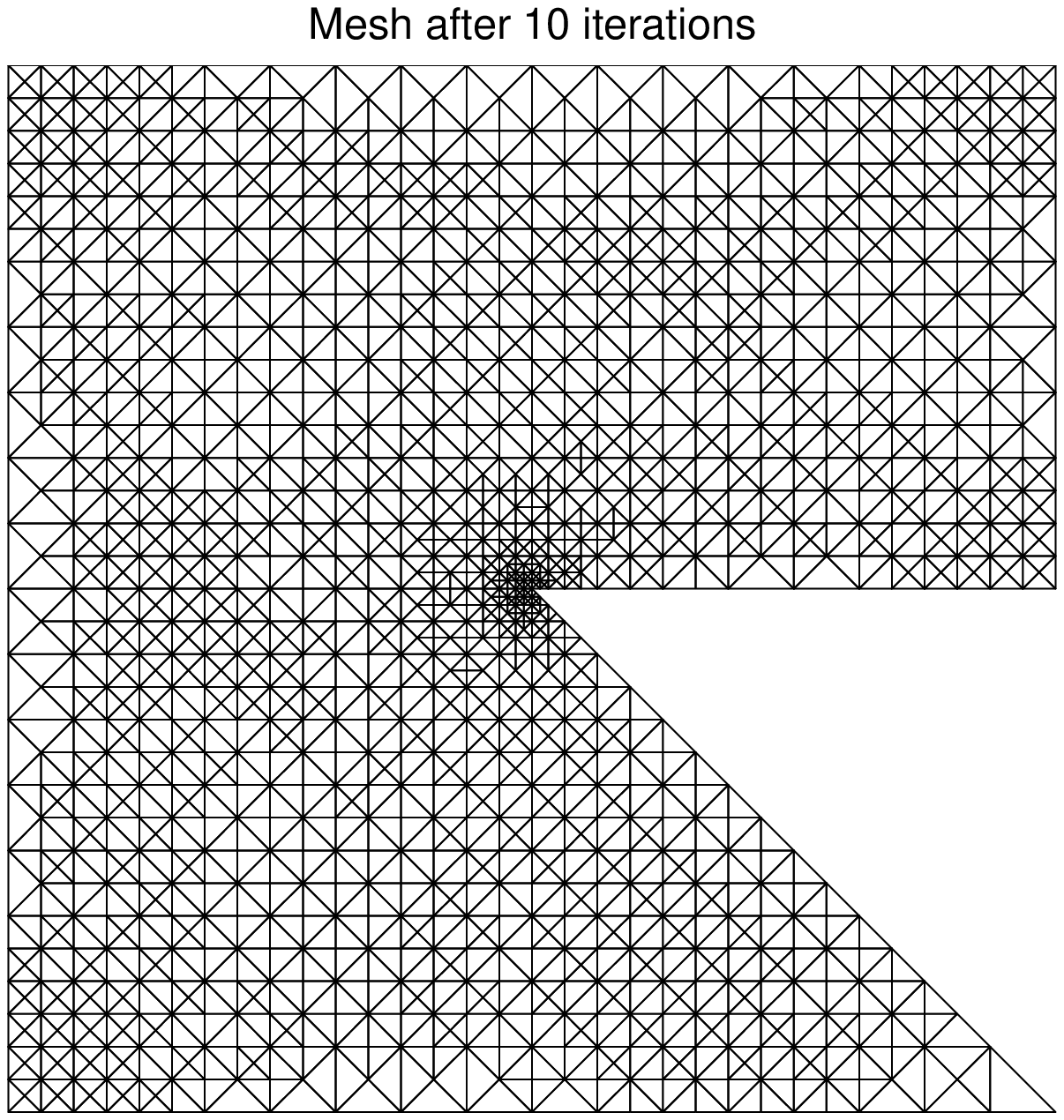}\hspace{0.05cm}
\includegraphics[width=7cm,height=6cm]{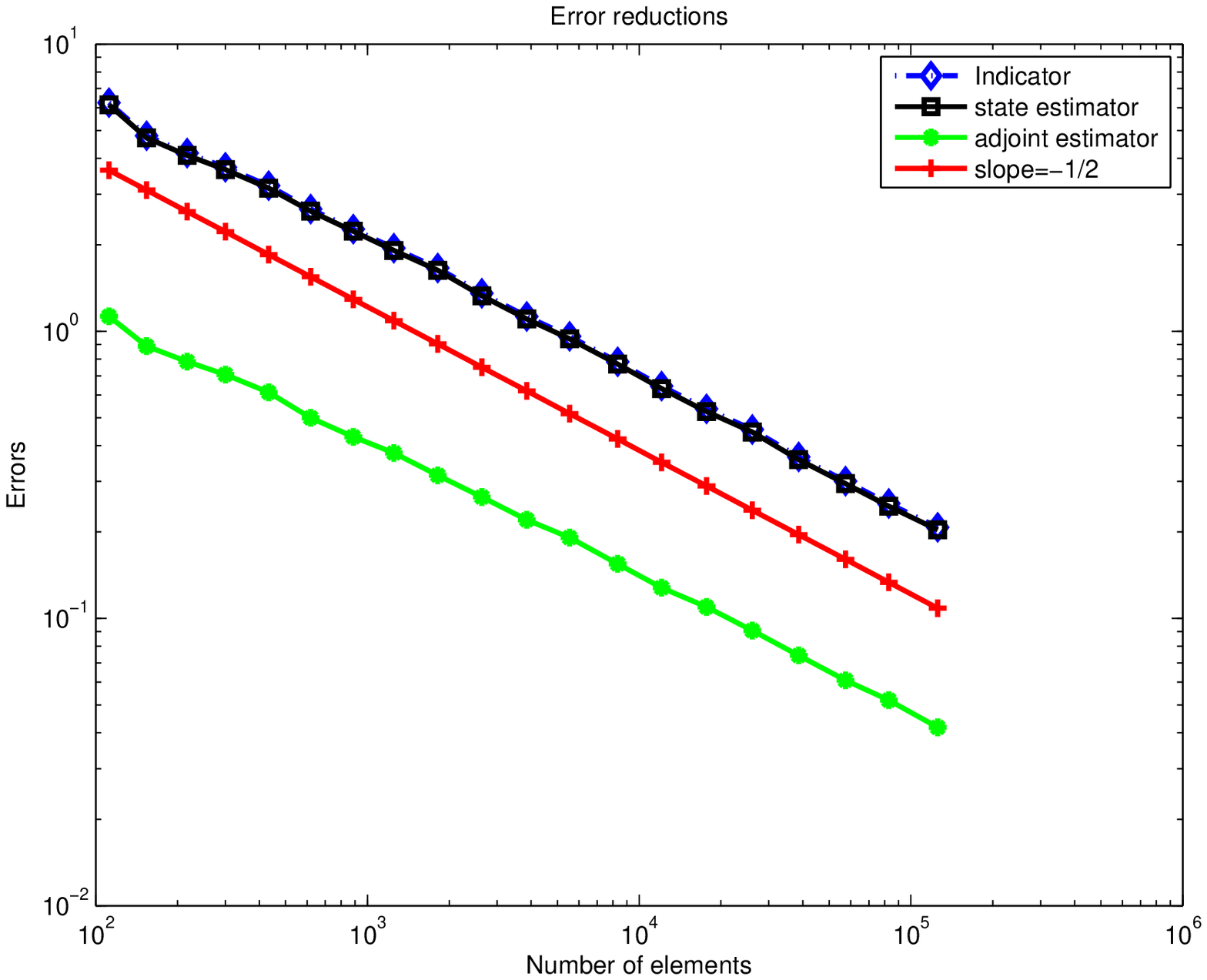}
\caption{The mesh (left) after 10 adaptive iteration and the convergence history of the error estimators on adaptively refined meshes (right) with $\theta=0.4$
for Example \ref{Exm:3} generated by Algorithm \ref{Alg:3.1}.}
\label{fig:8}
\end{figure}


\section{Conclusion and outlook}
\setcounter{equation}{0}
In this paper we give a rigorous convergence analysis of the adaptive finite element algorithm 
for optimal control problems governed by linear elliptic equation. We prove that the AFEM is a contraction, for the sum of the energy errors and the scaled error estimators of the state $y$ and the adjoint state $p$, between two consecutive adaptive loops. We also show that the AFEM yields a decay rate of the energy errors of the state $y$ and the adjoint state $p$ plus oscillations of the state and adjoint state equations in terms of the number of degrees of freedom. 

We expect that the results should also be valid for optimal Neumann boundary control problems (see \cite{Liu5}) by the following observations. The key point for the convergence analysis is the equivalence properties presented in Theorem \ref{Thm:2.2} where the relation between the finite element optimal control approximation and the standard finite element boundary value approximation is established. Consider the governing equation of the Neumann boundary control problem:
\begin{equation}\nonumber
\left\{\begin{array}{llr}
Ly=f \quad&\mbox{in}\ \Omega, \\
 \ A\nabla y\cdot n=u  \quad &\mbox{on}\ \partial\Omega.
\end{array}\right.
\end{equation}
Similar to the proof of Theorem \ref{Thm:2.2} we can conclude from the trace theorem that
\begin{eqnarray}
\|u-u_h\|_{0,\partial\Omega}\lesssim \kappa^{1\over 2}(h)(\|y-y_h\|_{a,\Omega}+\|p-p_h\|_{a,\Omega}),\nonumber
\end{eqnarray}
where $u_h$ is the discrete optimal control. Then we can obtain the counterpart of (\ref{equivalence_1})-(\ref{equivalence_2}) for Neumann boundary control problems
\begin{eqnarray}
\|y-y_h\|_{a,\Omega}&=& \|y^h-y_h\|_{a,\Omega}+O(\kappa^{1\over 2}(h))\Big(\|y-y_h\|_{a,\Omega}
+\|p- p_h\|_{a,\Omega}\Big),\nonumber\\
\|p-p_h\|_{a,\Omega}
&=& \|p^h- p_h\|_{a,\Omega} +O(\kappa^{1\over 2}(h))\Big(\|y-y_h\|_{a,\Omega}
+\|p- p_h\|_{a,\Omega}\Big)\nonumber
\end{eqnarray}
provided $h_0\ll 1$. Thus, the convergence and complexity analysis of AFEM carries out to the Neumann boundary control problems. 

There are many important issues remained unsolved for the convergence analysis of AFEM for optimal control problems compared to AFEM for boundary value problems. Firstly, at this moment we only prove the optimality of AFEM for energy errors of the state and adjoint state variables, the convergence for the optimal control $u$ is sub-optimal. To prove the optimality of AFEM for the optimal control $u$ it seems that we should work on the optimality of AFEM for boundary value problems under $L^2$-norms, as done in \cite{Demlow}. This complicates the convergence analysis with additional restrictions to the adaptive algorithms and will be postponed to future work. 

Secondly, the convergence analysis of the adaptive finite element algorithm 
for other kind of optimal control problems like Stokes control problems (see \cite{Liu3}), and non-standard finite element algorithm such as mixed finite element methods (see \cite{Chen}) remains open and will be addressed in forthcoming papers. 

Thirdly, we only prove the convergence of AFEM for optimal control problems with control constraints by using variational control discretization. The full control discretization concept by using piecewise constant or piecewise linear finite elements is also very important among the numerical methods for control problems. This kind of control discretizations results in an additional discretised control space and an additional contribution to the a posteriori error estimators (see \cite{Kohls}) which should be incorporated within the adaptive algorithm and the corresponding convergence analysis. We also intend to generalise our approach in this paper to analyse the convergence of AFEM for optimal control problems with full control discretization in the future. 


\section*{Acknowledgements}
The first author was supported by the National Basic Research Program of China under grant 2012CB821204 and the National Natural Science Foundation of China under grant 11201464. The second author acknowledged the support of the National Natural Science Foundation of China under grant 11171337. 


 \medskip


\begin{thebibliography}{99}

\bibitem{Apel}
T. Apel, A. R\"{o}sch and G. Winkler, 
Optimal control in non-convex domains: a priori discretization error estimates, 
Calcolo, 44 (2007), pp. 137-158.

\bibitem{Babuska}
I. Babu\v{s}ka and W.C. Rheinboldt, Error estimates for adaptive finite element computations,
SIAM J. Numer. Anal., 15 (1978), pp. 736-754.

\bibitem{Becker} 
R. Becker, H. Kapp and R. Rannacher, 
Adaptive finite element methods for optimal control of partial differential
equations: Basic concept, 
SIAM J. Control Optim., 39 (2000), pp. 113-132.

\bibitem{Mao}
R. Becker and S.P. Mao,
Quasi-optimality of an adaptive finite element method for an optimal control problems,
Comput. Methods Appl. Math., 11 (2011), pp. 107-128.

\bibitem{Bergounioux}
M. Bergounioux, K. Ito and K. Kunisch, 
Primal-dual strategy for constrained optimal control problems,
SIAM J. Control Optim., 37 (1999), pp. 1176-1194.

\bibitem{Binev}
P. Binev, W. Dahmen and R. DeVore, 
Adaptive finite element methods with convergence rates,
Numer. Math., 97 (2004), pp. 219-268.

\bibitem{Cascon}
J.M. Cascon, C. Kreuzer, R.H. Nochetto and K.G. Siebert, 
Quasi-optimal convergence rate for an adaptive finite element method, 
SIAM J. Numer. Anal., 46 (2008), no. 5, pp. 2524-2550.

\bibitem{Chen}Y.P. Chen and W.B. Liu, A posteriori error estimates for mixed finite element solutions of convex optimal control problems, J. Comput. Appl. Math., 211 (2008), no. 1, pp. 76-89.

\bibitem{Ciarlet}
P.G. Ciarlet, 
The Finite Element Methods for Elliptic Problems, North-Holland, Amsterdam, 1978.

\bibitem{Zhou}
X.Y. Dai, L.H. He and A.H. Zhou, Convergence and quasi-optimal complexity of adaptive finite element computations for multiple eigenvalues,
IMA J. Numer. Anal., 2014, DOI:10.1093/imanum/dru059.

\bibitem{Dai}
X.Y. Dai, J.C. Xu and A.H. Zhou,
Convergence and optimal complexity of adaptive finite element eigenvalue computations,
Numer. Math., 110 (2008), pp. 313-355.

\bibitem{Demlow} A. Demlow and R. Stevenson, Convergence and quasi-optimality of an adaptive finite element method for controlling $L^2$ errors, Numer. Math., 117 (2011), no. 2, pp. 185-218.
 

\bibitem{Dorfler}
W. D\"{o}rfler, A convergent adaptive algorithm for PoissonÕs equation,
SIAM J. Numer. Anal., 33 (1996), pp. 1106-1124.


\bibitem{Gaevskaya}
A. Gaevskaya, R.H.W. Hoppe, Y. Iliash and M. Kieweg,
Convergence analysis of an adaptive finite element method for distributed control problems with control constraints,
in Control of coupled partial differential equations, vol. 155 of Internat. Ser. Numer. Math., Birkh\"{a}user, Basel, 2007, pp. 47-68.


\bibitem{Grisvard2}
P. Grisvard, 
Singularities in Boundary Value Problems, Masson, Paris, and Springer-Verlag, Berlin, 1992.
    
\bibitem{He}L.H. He and A.H. Zhou, Convergence and complexity of adaptive finite element methods for elliptic partial differential equations, Inter. J. Numer. Anal. Model., 8 (2011), no. 4, pp. 615-640.
        
\bibitem{Hoppe}
M. Hinterm\"{u}ller and R.H.W. Hoppe, 
Goal-oriented adaptivity in control constrained optimal control of partial differential
equations, SIAM J. Control Optim., 47 (2008), no. 4, pp. 1721-1743.

\bibitem{Hoppe1}
M. Hinterm\"{u}ller, R.H.W. Hoppe, Y. Iliash and M. Kieweg, 
An a posteriori error analysis of adaptive finite element methods for distributed elliptic control problems with control constraints, 
ESAIM: Control Optim. Calc. Var., 14 (2008), pp. 540-560.

\bibitem{Hintermueller03SIOPT}
M. Hinterm\"{u}ller, K. Ito and K. Kunisch, The primal-dual active set strategy as a
semismooth Newton method, SIAM J. Optim., 13 (2003), pp. 865-888.

\bibitem{Hinze05COAP}
M. Hinze, A variational discretization concept in control constrained optimization:
The linear-quadratic case, Comput. Optim. Appl., 30 (2005), pp. 45-63.

\bibitem{Hinze09book}
M. Hinze, R. Pinnau, M. Ulbrich and S. Ulbrich,
 Optimization with PDE Constraints, Math. Model. Theo. Appl.,
 23, Springer, New York, 2009.



 \bibitem{Kohls} 
 K. Kohls, A. R\"{o}sch and K.G. Siebert, 
 A posteriori error analysis of optimal control problems with control constraints,
 SIAM J. Control Optim., 52 (2014), pp. 1832-1861.

\bibitem{Kohls1} 
 K. Kohls, A. R\"{o}sch and K.G. Siebert, 
Convergence of adaptive finite elements for control constrained optimal control problems, Preprint-Nr.: SPP1253-153, 2013. 

\bibitem{Li}R. Li, W.B. Liu, H.P. Ma and T. Tang, Adaptive finite element approximation for distributed elliptic optimal control problems, SIAM J. Control Optim., 41 (2002), no. 5, pp. 1321-1349. 

\bibitem{Lions} 
J.L. Lions, 
Optimal Control of Systems Governed by Partial Differential Equations, Springer-Verlag, Berlin, 1971.

\bibitem{Liu2} 
W.B. Liu and N.N. Yan, 
A posteriori error analysis for convex distributed optimal control problems, 
Adv. Comp. Math., 15 (2001), no. 1-4, pp. 285-309.

\bibitem{Liu5}
W.B. Liu and N.N. Yan, A posteriori error estimates for convex boundary control problems, SIAM J. Numer. Anal., 39 (2001), no. 1, pp. 73-99. 

\bibitem{Liu3} 
W.B. Liu and N.N. Yan, 
A posteriori error estimates for optimal problems governed by Stokes equations, 
SIAM J. Numer. Anal., 40 (2003), pp. 1850-1869.

\bibitem{Liu4} 
W.B. Liu and N.N. Yan, 
A posteriori error estimates for optimal control problems governed by parabolic equations, 
Numer. Math., 93 (2003), pp. 497-521.

\bibitem{LiuYan08book}W.B. Liu and N.N. Yan, 
Adaptive Finite Element Methods for Optimal Control Governed by PDEs, Science press, Beijing, 2008.

\bibitem{Mekchay}K. Mekchay and R.H. Nochetto, Convergence of adaptive finite element methods for general second order linear elliplic PDEs, SIAM J. Numer. Anal., 43 (2005), pp. 1803-1827.

\bibitem{Morin}
P. Morin, R.H. Nochetto and K.G. Siebert, Data oscillation and convergence of adaptive FEM,
 SIAM J. Numer. Anal., 38 (2000), pp. 466-488.

\bibitem{Nochetto}P. Morin, R.H. Nochetto, and K.G. Siebert, Convergence of adaptive finite element methods, SIAM Rev., 44 (2002), pp. 631-658.


\bibitem{Stevenson}
R. Stevenson, Optimality of a standard adaptive finite element method, Found. Comput. Math., 7 (2007), pp. 245-269.

\bibitem{Rob}
R. Stevenson, The completion of locally refined simplicial partitions created by bisection, Math. Comput., 77 (2008), pp. 227-241.

\bibitem{Verfuth}
R. Verf\"{u}rth,
A Review of a Posteriori Error Estimates and Adaptive Mesh Refinement Techniques, 
Wiley-Teubner, New York, 1996.



\end{thebibliography}
\end{document}